\documentclass[11pt,a4paper,hidelinks]{article}

\usepackage{LPZpreamble}

\usepackage{multicol}
\usepackage[normalem]{ulem} 
\makeatletter
\newcommand\mkcal@old[1]{  \expandafter\gdef\csname #1#1#1\endcsname{\ensuremath{\mathcal{#1}}}}
\forcsvlist{\mkcal@old}{A,B,C,D,E,F,G,H,I,J,K,L,M,N,O,P,Q,R,S,T,U,V,W,X,Y,Z}
\newcommand\mkscr@old[1]{  \expandafter\gdef\csname #1#1#1s\endcsname{\ensuremath{\mathscr{#1}}}}
\forcsvlist{\mkscr@old}{A,B,C,D,E,F,G,H,I,J,K,L,M,N,O,P,Q,R,S,T,U,V,W,X,Y,Z}
\makeatother

\renewcommand{\geq}{\geqslant}
\renewcommand{\leq}{\leqslant}
\renewcommand{\subset}{\subseteq}

\renewcommand{\Re}{\operatorname{Re}}
\renewcommand{\Im}{\operatorname{Im}}
\renewcommand{\bar}[1]{\overline{#1}}
\renewcommand{\tilde}[1]{\widetilde{#1}}
\renewcommand{\colon}{\mathbin{:}\ }

\newcommand{\C}{\ensuremath{\mathbb{C}}}
\newcommand{\R}{\ensuremath{\mathbb{R}}}
\newcommand{\Q}{\ensuremath{\mathbb{Q}}}
\newcommand{\Z}{\ensuremath{\mathbb{Z}}}

\newcommand{\id}{\ensuremath{\operatorname{id}}}
\newcommand{\ii}{\ensuremath{\mathfrak{i}}}
\newcommand{\e}{\ensuremath{\operatorname{e}}}

\newcommand{\img}{\ensuremath{\operatorname{im}}}
\renewcommand{\ker}{\ensuremath{\operatorname{ker}}}
\newcommand{\coker}{\ensuremath{\operatorname{coker}}}

\newcommand{\ord}[1]{\left\langle #1 \right\rangle}

\newcommand{\RP}{\mathbb{RP}}

\newcommand{\Hom}{\operatorname{Hom}}

\newcommand{\Aut}{\operatorname{Aut}}

\newcommand{\cone}{\operatorname{cone}}

\DeclarePairedDelimiterX\setc[2]{\{}{\}}{\,#1 \;\delimsize\vert\; #2\,}

\renewcommand{\CCC}{\mathscr{C}}

\newcommand{\Dcat}{\mathsf{D}}
\newcommand{\Dcatb}{\Dcat^{\mathrm{b}}}

\newcommand{\Obj}{\operatorname{Obj}}

\newcommand{\Coh}{\mathsf{Coh}}

\newcommand{\Rsf}{\mathsf{R}}

\newcommand{\sheaf}[1]{\mathscr{O}_{#1}}

\newcommand{\aff}{\mathbb{A}}
\newcommand{\pro}{\mathbb{P}}
\newcommand{\Hlg}{\operatorname{H}}
\newcommand{\hlg}{\operatorname{h}}
\newcommand{\Ext}{\operatorname{Ext}}
\newcommand{\ext}{\operatorname{ext}}

\newcommand{\SHom}{\operatorname{\mathscr{H}\!\mkern-1mu\mathit{om}}}

\newcommand{\SExt}{\operatorname{\mathscr{E}\!\mathit{xt}}}
        \newcommand{\STor}{\operatorname{\mathscr{T}\!\!\mathit{or}}}

\newcommand{\Pic}{\operatorname{Pic}}
\newcommand{\ideal}{\mathcal{I}} 
\newcommand{\Nor}[2]{\mathcal{N}_{#1 \mid #2}}

\newcommand{\dashto}{\dashrightarrow}

\newcommand{\rk}{\operatorname{rk}}
\newcommand{\chern}[1]{\operatorname{c}_{{#1}}}
\newcommand{\Chern}[1]{\operatorname{ch}_{{#1}}}
\newcommand{\td}{\operatorname{td}}
\newcommand{\Gr}{\operatorname{Gr}}

\newcommand{\Gro}{\operatorname{K_0}}
\newcommand{\Knum}{\operatorname{K_{num}}}
\newcommand{\Ktop}{\operatorname{K_{top}}}

\newcommand{\Ku}{\operatorname{\mathcal{K}\mkern-1mu\mathit{u}}}

\newcommand{\Modss}{\mathrm{M}_{\sigma_Y}}
\newcommand{\Mods}{\mathrm{M}^{\mathsf{s}}_{\sigma_Y}}

\newcommand{\ModssX}{\mathrm{M}_{\sigma_X}}
\newcommand{\ModssY}{\mathrm{M}_{\sigma_Y}}
\newcommand{\ModsX}{\mathrm{M}_{\sigma_X}^{\mathsf{s}}}
\newcommand{\ModsY}{\mathrm{M}_{\sigma_Y}^{\mathsf{s}}}
\newcommand{\SKu}{\mathsf{S}_{\Ku(Y)}}

\newcommand{\GL}{\operatorname{GL}}

\newcommand{\Sym}{\operatorname{Sym}}

\newcommand{\Stab}{\operatorname{Stab}}

\newcommand{\dual}{{\raisebox{2pt}{\scalebox{0.7}[0.3]{$\bm\vee$}}}}

\newcommand{\lmut}[1]{\mathbf{L}_{#1}}
\newcommand{\rmut}[1]{\mathbf{R}_{#1}}

\newcommand{\Bl}{\operatorname{Bl}}

\newcommand{\pr}{\operatorname{pr}} 

\newcommand{\tv}{\widetilde{v}}

\newcommand{\tY}{\widetilde{Y}}
\newcommand{\tZ}{\widetilde{Z}}

\newcommand{\tA}{\widetilde{\cA}}

\newcommand{\RHom}{\operatorname{\mathsf{R}Hom}}

\newcommand{\Dperf}{\operatorname{\mathsf{D}}^{\mathsf{perf}}}

\newcommand{\ev}{\operatorname{ev}}

\newcommand{\ch}{\operatorname{ch}}

\newcommand{\ts}{\widetilde{\sigma}}

\title{Stability Conditions and Moduli Spaces \\ on the Kuznetsov Component of Cubic Fivefolds}
\author{Peize Liu}
\date{}

\begin{document}

\maketitle

\begin{abstract}
    We study the Kuznetsov component $\Ku(Y)$ of a smooth cubic fivefold $Y$. Using a quadric surface fibration, we construct a family of Serre-invariant Bridgeland stability conditions on $\Ku(Y)$. 
    For every non-zero numerical class $v$, we prove that the associated Bridgeland moduli space on $\Ku(Y)$ is non-empty. When $Y$ is general and $v$ is primitive, the moduli space contains a smooth locus, on which the restriction to a general hyperplane section preserves stability.
    Consequently, we obtain Lagrangian immersions into hyper-K\"ahler varieties arising as moduli spaces on the Kuznetsov component of cubic fourfolds, extending the work of Li--Lin--Pertusi--Zhao on cubic 3-folds. As an example, we recover the geometric construction of Illiev--Manivel, which realizes the Fano surface of planes of the cubic fivefold as a Lagrangian subvariety in a hyper-K\"ahler fourfold.
\end{abstract}

\tableofcontents

\setlength{\abovedisplayshortskip}{0pt plus 3pt}
\setlength{\belowdisplayshortskip}{6pt plus 3pt minus 3pt}
\setlength{\abovedisplayskip}{8pt plus 3pt minus 3pt}
\setlength{\belowdisplayskip}{8pt plus 3pt minus 3pt} 

\section{Introduction}
The study of stability conditions on bounded derived categories of coherent sheaves on varieties was initiated by Bridgeland in \cite{bri07}. Bridgeland stability conditions have since been constructed on curves, surfaces, Fano 3-folds, Abelian 3-folds, quintic 3-folds, etc. On the other hand, for a Fano variety $X$, it has a semi-orthogonal component $\Ku(X)$ of $\Dcatb(X)$, known as the Kuznetsov component of $X$, which encodes the interesting geometric information about the variety $X$. Extensive research has been conducted on the Kuznetsov component for Fano 3-folds (e.g.\ \cite{kuzFano}), cubic 4-folds (e.g.\ \cite{kuz4fold}), and Gushel--Mukai varieties (e.g.\ \cite{KP_GM}). A general method of constructing stability conditions on Kuznetsov components was introduced in \cite{BLMS}.

After this paper first appeared, Li \cite{Li26} proved that $\Dcatb(X)$ admits Bridgeland stability conditions for every smooth projective variety $X$. Meanwhile, the existence of stability conditions on Kuznetsov components remains open in general and continues to be an important problem.

\subsection*{Main results}

A smooth cubic 5-fold $Y$ is a smooth cubic hypersurface in $\pro^6$. Its bounded derived category admits the semi-orthogonal decomposition
\begin{equation}
    \Dcatb(Y) = \ord{\, \Ku(Y),\, \sheaf{Y},\, \sheaf{Y}(1),\, \sheaf{Y}(2),\, \sheaf{Y}(3)\, }, \label{SOD:cubic_5}
\end{equation}
where the full subcategory  $\Ku(Y)$ with objects
\[
    \Obj(\Ku(Y)) = \left\{E\in\Dcatb(Y)\mid
    \Ext^\bullet(\sheaf{Y}(i),E)=0\text{ for }i=0,1,2,3\right\}
\]
is the \textbf{Kuznetsov component} of $Y$. The Kuznetsov component $\Ku(Y)$ is a smooth and proper triangulated category. Its Serre functor satisfies $\sfS_{\Ku(Y)}^3 \cong [7]$ (\cite{KuzV14}), and its numerical Grothendieck group has rank $2$ (see \cref{subsec:Knum_KuY}). The category $\Ku(Y)$ also determines $Y$: by \cite[Theorem 1.4]{zhang24}, a Fourier--Mukai equivalence $\Ku(Y) \simeq \Ku(Y')$ for two smooth cubic $5$-folds $Y$ and $Y'$ implies that $Y \cong Y'$.

\paragraph{Stability conditions.} The first main result of the paper is a construction of stability conditions on $\Ku(Y)$ with good properties:

\begin{theorem}[\cref{thm:main,thm:Serre_inv,cor:gl_dim}][thm:main_intro]
    Let $Y$ be a smooth cubic 5-fold. Then $\Ku(Y)$ admits a continuous
    family of stability conditions $\sigma'_{\beta,\alpha}$, parametrized
    by an open subset of $\R^2$. These stability conditions are
    Serre-invariant and lie in a single orbit under the
    $\widetilde{\GL^+_2}(\R)$-action. Moreover, this orbit contains a
    stability condition $\sigma_Y$ such that
    \begin{equation}
        \phi_{\sigma_Y}\ab(\mathsf{S}_{\Ku(Y)}(E))
        =
        \phi_{\sigma_Y}(E) + \frac{7}{3}
    \end{equation}
    for every $\sigma_Y$-semistable object $E \in \Ku(Y)$.
\end{theorem}

We outline the proof, some intermediate results and some immediate corollaries along the way. 

Every smooth cubic 5-fold contains a plane $\varPi_0\cong\pro^2$ \cite[Proposition 1.8]{col86}. Projection from $\varPi_0$ induces a quadric surface fibration $\pi\colon \Bl_{\varPi_0}Y\to\pro^3$. Let $\CCC_0$ and $\CCC_1$ denote the even and odd parts of the associated sheaf of Clifford algebras. Comparing the semi-orthogonal decompositions of $\Dcatb(\tY)$ supplied by Orlov's blow-up formula \cite{orlov92} and Kuznetsov's quadric fibration formula \cite{kuzQuaFib}, and matching them by a sequence of mutations, we obtain an embedding of $\Ku(Y)$ into the derived category of Clifford modules over $\pro^3$ (see \cref{subsec:quad_fib} for the relevant terminology):
\begin{theorem}[\cref{thm:psi}][thm:psi_intro]
    Let $Y$ be a smooth cubic 5-fold. There is a semi-orthogonal decomposition
    \[
        \Dcatb(\pro^3,\CCC_0)=\ord{\Ku(\pro^3,\CCC_0),\, \CCC_1,\, \CCC_2}
    \]
    and an exact equivalence $\Ku(Y)\simeq\Ku(\pro^3,\CCC_0)\coloneqq\ord{\CCC_1,\, \CCC_2}^{\perp}.$
\end{theorem}

Next we construct $\sigma'_{\beta,\alpha}$ in \cref{thm:main_intro} using the embedding $\Ku(Y) \hookrightarrow \Dcatb(\pro^3,\CCC_0)$. We prove a Bogomolov--Gieseker-type inequality on the non-commutative $\pro^3$, following the strategy of \cite[\S 8]{BLMS}. This establishes a family of tilt-stability conditions on $\Dcatb(\pro^3,\CCC_0)$. The stability conditions $\sigma'_{\beta,\alpha}$ are obtained by first rotating the tilt stability on $\Dcatb(\pro^3,\CCC_0)$ and then restricting to the orthogonal of $\CCC_1$ and $\CCC_2$ using the restriction criterion of \cite{BLMS}. 

\paragraph{Serre invariance.} Serre invariance means that stability is unchanged, up to the $\widetilde{\GL^+_2}(\R)$-action, under the Serre functor. Adapting \cite{PY20},  we control the phase of the heart $\cA$ under the twisting by $\CCC_1$. Informally, we show that 
    \[ \cA \otimes_{\CCC_0} \CCC_1 \subset \ord{\cA,\, \cA[1]}. \]
See \cref{prop:tensor_heart} for a precise statement. Using this we prove that $\sigma'_{\beta,\alpha}$ are invariant under the rotation functor $\mathsf{O} = \lmut{\CCC_0}\circ(-\otimes_{\CCC_0}\CCC_1)$, from which Serre invariance follows, as the Serre functor of the Kuznetsov component is given as $\mathsf{O}^2[-3]$ (see \cref{cor:Serre}).

We also identify, in the same $\widetilde{\GL^+_2}(\R)$-orbit, a Gepner-type stability condition $\sigma_Y$ for the pair $(\sfO,2/3)$, thereby confirming a case of \cite[Conjecture 1.1]{Tod13}. In particular, the Serre functor acts on $\sigma_Y$ by a phase shift of $7/3$. The argument extends techniques developed for Fano 3-folds in \cite{PY20,PRo23,FP23}, while the fractional Calabi--Yau dimension $7/3$ lies beyond the range of the general uniqueness criterion in \cite{FP23}.

\paragraph{Moduli spaces.} Having constructed these stability conditions, we turn to their stable objects. For a plane $\varPi\subset Y$, the projected ideal sheaf
\begin{equation}
    \cF_\varPi[1] \coloneqq \pr_Y(\ideal_{\varPi}(1))
    = \ker\ab(\sheaf{Y}^{\oplus 4} \twoheadrightarrow \ideal_{\varPi}(1))[1]
\end{equation}
defines an object of $\Ku(Y)$ of character $-\kappa_1$. A wall-crossing calculation proves its $\sigma_Y$-stability and identifies the Fano surface $\cF_2(Y)$ of planes with a connected component of the lowest-dimensional moduli space $\ModssY(\Ku(Y),-\kappa_1)$. Restriction to a very general cubic 4-fold hyperplane section then gives a Lagrangian immersion.

\begin{theorem}[\cref{prop:moduli_dim2,cor:Lag_imm}][thm:moduli_dim2_intro]
    Let $Y$ be a general cubic 5-fold, and let $X\subset Y$ be a very general cubic 4-fold. The Fano surface $\cF_2(Y)$ of planes in $Y$ is isomorphic to a connected component of the moduli space $\ModssY(\Ku(Y),-\kappa_1)$. Moreover, there is an unramified and generically injective morphism $\cF_2(Y)\to\cF_1(X)$ to the Fano variety of lines in $X$. Its image is a Lagrangian subvariety of the hyper-K\"ahler fourfold $\cF_1(X)$.
\end{theorem}
This recovers the construction of Illiev--Manivel \cite[\S 2.2.2]{IM08} from the perspective of Bridgeland moduli spaces.

For higher-dimensional moduli spaces, we adapt \cite{LLPZ} to prove non-emptiness for every numerical class and to construct smooth loci for primitive classes. The main additional difficulty is that the stability condition $\sigma_Y$ has global dimension $\frac{7}{3}>2$, so the obstruction space $\Ext^2(E,E)$ need not vanish for every stable object. We overcome this by proving generic $\Hom$- and $\Ext^2$-vanishing. This leads to the following higher-dimensional analogue of \cref{thm:moduli_dim2_intro}.

\begin{theorem}[\cref{thm:moduli_nonempty,prop:Ext2_zero_general,thm:res,cor:Lag_imm_general}][thm:res_intro]
    Let $Y$ be a general cubic 5-fold, and let $i\colon X\hookrightarrow Y$ be a very general cubic 4-fold. The moduli space $\ModssY(\Ku(Y),v)$ is non-empty for every non-zero class $v\in\Knum(\Ku(Y))$. If $v$ is primitive, there is a non-empty smooth open subset
    \[
        \UUUs_Y(v)\subset\ModssY(\Ku(Y),v)
    \]
    such that $i^*E$ is stable for every $E\in\UUUs_Y(v)$. Furthermore, restriction induces a rational map
    \[
        r_v\colon\overline{\UUUs_Y(v)}\dashrightarrow\ModssX(\Ku(X),i^*v)
    \]
    which is generically unramified and whose image is Lagrangian in the smooth projective hyper-K\"ahler variety $\ModssX(\Ku(X),i^*v)$.
\end{theorem}
In particular, \cref{thm:res_intro} supplies a new class of Lagrangian subvarieties in K3$^{[n]}$-type hyper-K\"ahler varieties, extending the work of \cite{FGLZ,LLPZ}. 

\subsection*{Background and related work}

\paragraph{Stability conditions on Kuznetsov components.}
Kuznetsov components appeared systematically in the study of Fano $3$-folds, cubic 4-folds, and Gushel--Mukai varieties \cite{kuzFano,kuz4fold,KP_GM}. Stability conditions on the Kuznetsov components of cubic $3$-folds were first studied in \cite{BMMS}, where they led to a categorical Torelli theorem. The technique of restricting stability conditions from an ambient derived category to the orthogonal complement of an exceptional collection was developed in \cite{BLMS} and applied to cubic 4-folds and to the Picard-rank-one Fano $3$-folds covered by \cite[Theorems 1.1 and 1.2]{BLMS}. The construction was extended to families in \cite{BLMNPS} and adapted to Gushel--Mukai varieties in \cite{PPZ22}. A common categorical model behind several of these constructions is provided by quadric fibrations and their sheaves of Clifford algebras, as introduced in \cite{kuzQuaFib}. More recently, an alternative method for inducing bounded t-structures on semi-orthogonal components was developed in \cite{KLP}.

Chunyi Li \cite{Li26} proved the existence of stability conditions on $\Dcatb(X)$ for every smooth projective variety $X$. His construction lies near the large-volume limit, whereas the restriction method of \cite{BLMS} typically uses stability conditions near the small-volume limit. Constructing stability conditions on semi-orthogonal components of $\Dcatb(X)$ therefore remains open in general.

\paragraph{Serre-invariant and Gepner-type stability conditions.}

The symmetry of the stability manifold under the Serre functor has been analysed in many of the examples above. Serre-invariant stability conditions were constructed for Fano 3-folds of Picard rank 1 and index 2 (del Pezzo 3-folds) in \cite{PY20}, and for those of index 1 (prime Fano 3-folds) in \cite{PRo23}. Uniqueness for the Serre-invariant $\widetilde{\GL^+_2}(\R)$-orbit are known for del Pezzo 3-folds of degree $\geq 2$ and prime Fano 3-folds of genus $\geq 6$; see \cite{JLLZ,FP23,FLM23}. A general uniqueness criterion for fractional Calabi--Yau categories of dimension at most 2 was formulated in \cite{FP23}. It does not cover cubic 5-folds, whose Kuznetsov component has fractional Calabi--Yau dimension $7/3$. It is therefore an interesting question to investigate whether the Serre-invariant stability condition constructed in this work is unique. 

Serre invariance is not expected in complete generality: its non-existence was proved in \cite{KPSerre} for many Kuznetsov components of Fano complete intersections whose defining degrees are not all equal.

The Kuznetsov component of a hypersurface is fractional Calabi--Yau, and the relation between its Serre and rotation functors suggests the existence of Gepner-type stability conditions in the sense of \cite{Tod13}; see also \cite{kuzCY}. Such conditions are known for cubic 3-folds \cite{PY20} and for general cubic 4-folds containing a plane \cite{Tod13}, with respect to the pair $(\sfO,2/3)$. We construct a Gepner-type stability condition for the same pair on a cubic 5-fold, confirming a new case of \cite[Conjecture 1.1]{Tod13}.

\paragraph{Moduli spaces and hyper-K\"ahler geometry.}
By Mukai's classical result \cite{Muk84Symp}, moduli spaces of Gieseker-stable sheaves on K3 surfaces provide a basic source of K3$^{[n]}$-type hyper-K\"ahler varieties. For non-commutative K3 surfaces, such as the Kuznetsov components of cubic 4-folds and Gushel--Mukai 4-folds, moduli spaces of Bridgeland-stable objects recover classical examples and produce larger families of K3$^{[n]}$-type hyper-K\"ahler varieties \cite{BLMS,BLMNPS,PPZ22}. By contrast, less is known for Kuznetsov components that are not of K3 type. Moduli spaces on the Kuznetsov components of del Pezzo 3-folds, including cubic 3-folds, have been studied in \cite{APR22,PR23,LiuZhang22,FP23,Bayer_desing_cubic}. A lattice-theoretic induction in \cite{LLPZ} establishes non-emptiness for moduli spaces on Fano 3-folds.

A categorical framework producing Lagrangian families from Bridgeland moduli spaces on Gushel--Mukai 4-folds and their hyperplane sections was developed in \cite{FGLZ}. The same work proposed an analogous statement for cubic 4-folds and their hyperplane sections, which was subsequently proved in \cite{LLPZ}. Our \cref{cor:Lag_imm_general} is an analogue of \cite[Theorem 8.3]{LLPZ} for cubic 5-folds and their hyperplane sections. The smallest numerical classes correspond to Fano surfaces of planes, whose geometry was studied in \cite{col86,Mbo23}. For these surfaces, we recover the construction of Illiev--Manivel \cite[\S 2.2.2]{IM08}, whereas arbitrary primitive classes produce new Lagrangian subvarieties. This result could have potential applications to a moduli-theoretic approach to the compactified Lagrangian fibration obtain in \cite[Theorem 1.2]{LLX}.

\subsection*{Plan of the paper}
\cref{sec:quad_fib} sets up the geometric and categorical framework of quadric fibrations for the cubic 5-fold $Y$ and identify the $\Ku(Y)$ with $\Ku(\pro^3,\CCC_0) \subset \Dcatb(\pro^3,\CCC_0)$ (\cref{thm:psi}). \cref{sec:Knum} then computes the numerical lattices on both sides. In particular, \cref{prop:Knum_Ku,prop:Knum_Ku_C0} give compatible integral bases and describe the actions of the Serre and rotation functors. 

\cref{sec:stab} and \cref{sec:Serre} construct stability conditions. After proving the Bogomolov--Gieseker-type inequality of \cref{prop:BG_noncomm_Pm}, tilt stability on $\Dcatb(\pro^3,\CCC_0)$ is restricted to the Kuznetsov component, yielding the $(\beta,\alpha)$ family in \cref{thm:main}. Then we change to another set $(\xi,\eta)$ of parameters for the stability conditions for better visualization in \cref{subsec:tensor}. We understand how twisting by $\CCC_1$, and hence the rotation functor, transforms the tilted hearts. This leads to the Serre-invariance \cref{thm:Serre_inv}. The normalization in \cref{cor:gl_dim} yields the Gepner-type condition. 

\cref{sec:moduli} and \cref{sec:high_moduli} study moduli spaces. First in \cref{subsec:P_stab}, $\cF_{\varPi}$ is shown to be stable by a wall-crossing argument, giving the non-emptiness of the minimal moduli spaces, and \cref{prop:moduli_dim2} identifies the Fano surface with a connected component of these spaces. Restriction to a smooth cubic $4$-fold produces the Lagrangian immersion in \cref{cor:Lag_imm}. We then record the geometry of the Fano surface in \cref{subsec:plane_incid}. The higher-dimensional theory is then built inductively from extensions of stable objects. We first give non-emptiness for every non-zero class in \cref{thm:moduli_nonempty}, and then the generic Hom-vanishing results of \cref{subsec:Hom_van} control the smooth locus of these moduli spaces. Finally, \cref{thm:res} proves that a general object remains stable after restriction to a very general cubic $4$-fold, and \cref{cor:Lag_imm_general} realizes the image as Lagrangian in the target hyper-K\"ahler variety.

\subsection*{Notations and conventions}
\begin{itemize}[nosep, leftmargin = 1.0em]
  \item Throughout the paper we work over the base field $\C$. All varieties are smooth and projective unless stated otherwise; all categories and functors are $\C$-linear, and the Hom spaces are finite-dimensional over $\C$.
  \item For a category $\cT$ we write $A\in\cT$ to indicate that $A$ is an object of $\cT$. A subcategory $\cK\subset\cT$ is understood to be full and additive. Any equivalence $\cK\simeq\cT$ of triangulated categories is assumed to be exact. For objects or full subcategories $\cA_1,\ldots,\cA_m$ of a triangulated category, $\ord{\cA_1,\ldots,\cA_m}$ denotes the smallest full triangulated subcategory containing them. By contrast, $\ord{\cA_1,\ldots,\cA_m}_{\ext}$ denotes their extension closure in the ambient Abelian or triangulated category.
  \item Write $\Dcatb(X) \coloneqq \Dcatb(\Coh(X))$ for the bounded derived category of coherent sheaves and $\Dperf(X)$ for its full subcategory of perfect complexes. In derived categories, $f_*$ and $f^*$ denote the derived push-forward $\mathsf{R}f_*$ and the derived pull-back $\mathsf{L}f^*$, respectively. Likewise, a tensor product used in a derived category is derived and the superscript $\mathsf{L}$ is often suppressed.
  \item Use cohomological grading:  $\Ext^i(A,B)=\Hom(A,B[i])$. Set
    \[
        \hlg^i(X,E)=\dim_{\C}\Hlg^i(X,E),\quad
        \hom(A,B)=\dim_{\C}\Hom(A,B),\quad
        \ext^i(A,B)=\dim_{\C}\Ext^i(A,B).
    \]
    Denote by $\mathcal H^i_{\cA}(E)$ the cohomological object of the complex $E$ with respect to a specified heart $\cA$.
  \item For a projective variety $X$, $\Gro(X)$ means $\Gro(\Dcatb(X))=\Gro(\Coh(X))$.  Define $\Knum(\cT)$ by quotienting $\Gro(\cT)$ by the right radical of the Euler pairing $\chi(-,-)$.
  \item Adopt the geometric convention: $\pro^n$ parametrizes lines in $\aff^{n+1}$; and for a vector bundle $F$ on $X$, the projective bundle parametrizes one-dimensional subspaces of the fibres of $F$:
    \begin{equation}
        \pro_X(F)=\operatorname{Proj}_X\ab(\operatorname{Sym}^{\bullet}F^{\dual})
    \end{equation}
    If $q\colon\pro_X(F)\to X$ is the projection, then $\sheaf{\pro_X(F)}(-1)\hookrightarrow q^*F$ is the tautological subbundle and $q_*\sheaf{\pro_X(F)}(1)=F^{\dual}$. 
\end{itemize}

\subsection*{Acknowledgements}
I am very grateful to my doctoral supervisor Chunyi Li for proposing the problem and for his invaluable guidance and support throughout this research. 
I would also like to thank Arend Bayer, Yiran Cheng, Hanfei Guo, Zhiyu Liu, Laura Pertusi, Chunkai Xu, and Shizhuo Zhang for helpful conversations and comments. Special thanks to Zhiyu Liu for his conceptual and computational insights which helped shape the proofs of \cref{lem:Ext_F_Pi,lem:k1+k2_ext2}.
After completing this project, Song Yang kindly shared with me his early notes on the cubic fivefold problem, whose approach to \cref{thm:psi} contained some ideas that I found inspiring.

This work is supported by the Warwick Mathematics Institute Centre for Doctoral Training, and I gratefully acknowledge funding from the University of Warwick. This work is also partially supported by the Royal Society URF{\textbackslash}R1{\textbackslash}201129 `Stability condition and application in algebraic geometry'.
Part of the paper was written during visits to Shanghai Center for Mathematical Sciences and Zhejiang University, and I would like to thank these institutes for their hospitality.

\section{Derived category of quadric fibrations}\label{sec:quad_fib}

\subsection{Geometry of quadric fibrations}\label{subsec:quad_fib}

We start by describing the classical construction of a quadric fibration obtained by blowing up a cubic hypersurface along a linear subspace. The following setting is adapted from \cite[\S 4]{kuz4fold} and \cite[\S 1.5]{huy23}.

\begin{setting}[][set:geom_quad_fib]
    Let $V$ be a $(n+2)$-dimensional vector space over $\C$, let $A \subset V$ be a $(k+1)$-dimensional subspace, and set $B \coloneqq V/A$. Projectivizing the quotient map $V \twoheadrightarrow B$ gives the rational map $\pro^{n+1} = \pro(V) \dashto \pro(B) = \pro^{n-k}$, namely the projection away from the $k$-plane $\varPi_0 = \pro(A)$. Its indeterminacy is resolved by blowing up $\pro(V)$ along $\varPi_0$. Let $\tau\colon \Bl_{\varPi_0}\pro(V) \to \pro(V)$ be the blow-up and $q\colon \Bl_{\varPi_0}\pro(V) \to \pro(B)$ the resulting morphism. The morphism $q$ is a $\pro^{k+1}$-fibration and identifies $\Bl_{\varPi_0}\pro(V)$ with the projective bundle $\pro_{\pro(B)}(\FFF)$, where $\FFF \coloneqq \left(q_*\tau^*\sheaf{\pro(V)}(1) \right)^\dual$ is locally free of rank $k+2$. To compute $\FFF$ explicitly, consider the short exact sequence 
\[\begin{tikzcd}
	0 & {q^*\sheaf{\pro(B)}(1)} & {\tau^*\sheaf{\pro(V)}(1)} & {\tau^*\sheaf{\pro(V)}(1)|_{E'}} & 0, 
	\arrow[from=1-1, to=1-2]
	\arrow[from=1-2, to=1-3]
	\arrow[from=1-3, to=1-4]
	\arrow[from=1-4, to=1-5]
\end{tikzcd}\]
where $E' \cong \varPi_0 \times \pro(B)$ is the exceptional divisor of $\tau$. Pushing the sequence forward via $q_*$ gives the split exact sequence 
\[\begin{tikzcd}
	0 & {\sheaf{\pro(B)}(1)} & {\FFF^\dual} & {A \otimes \sheaf{\pro(B)}} & 0.
	\arrow[from=1-1, to=1-2]
	\arrow[from=1-2, to=1-3]
	\arrow[from=1-3, to=1-4]
	\arrow[from=1-4, to=1-5]
\end{tikzcd}\]
Hence we have $\FFF \cong \sheaf{\pro(B)}^{\oplus (k+1)} \oplus \sheaf{\pro(B)}(-1)$.

Let $Y \subset \pro(V)$ be a smooth cubic hypersurface containing the $k$-plane $\varPi_0$. Denote by $\sigma\colon \tilde Y \to Y$ the embedded blow-up of $Y$ under $\tau\colon \pro(\FFF) \to \pro(V)$. Let $E = E' \cap \tilde Y \subset \tilde Y$ be the exceptional divisor.
\end{setting}

\begin{equation}\label{diag:cubic_n_quad_fib}
    \begin{tikzcd}[ampersand replacement=\&,cramped]
        \& E \& {\tilde Y} \& {\pro(\FFF)} \\
        \varPi_0 \& Y \& {\pro(V)} \&\& {\pro(B)}
        \arrow["\iota", hook, from=1-2, to=1-3]
        \arrow["p"', from=1-2, to=2-1]
        \arrow["\alpha", hook, from=1-3, to=1-4]
        \arrow["\sigma"', from=1-3, to=2-2]
        \arrow["\pi", from=1-3, to=2-5]
        \arrow["\tau"', crossing over, from=1-4, to=2-3]
        \arrow["q", from=1-4, to=2-5]
        \arrow["j", hook, from=2-1, to=2-2]
        \arrow[hook, from=2-2, to=2-3]
    \end{tikzcd}
\end{equation}

\begin{lemma}[][lem:cubic5_quad_fib_Pic]
    The morphism $\pi = q \circ \alpha\colon \tilde Y \to \pro(B)$ is a fibration in $k$-dimensional quadrics with the discriminant locus $D \in |\sheaf{\pro(B)}(k+4)|$. Let $H' \coloneqq \tau^*\sheaf{\pro(V)}(1)$ and $h' \coloneqq q^*\sheaf{\pro(B)}(1)$. Then $\tilde Y = 2H' + h'$ in $\Pic \pro(\FFF)$. Denote by $H$ and $h$ the restriction of the classes $H'$ and $h'$ in $\tilde Y$ respectively; then we have $E = H - h$ and $K_{\tilde Y} = -kH - (n-k-1)h$ in $\Pic \tilde Y$.
\end{lemma}
\begin{proof}
    In $\Pic \pro(\FFF)$ we have $\tilde Y = 3H' - E'$ and $h' = H' - E'$. Hence $\tilde Y = 2H' + h'$. Note that 
    \[q_*\sheaf{\pro(\FFF)}(\tilde Y) = q_*\sheaf{\pro(\FFF)}(2H' + h') \cong \Sym^2\FFF^\dual \otimes \sheaf{\pro(B)}(1).\]
    Hence $\tilde Y$ is defined by a section of $\Sym^2\FFF^\dual \otimes \sheaf{\pro(B)}(1)$, or equivalently, by a quadratic form
    \[ Q\colon \FFF \longrightarrow \FFF^\dual \otimes \sheaf{\pro(B)}(1). \]
    Therefore the fibre of $\pi\colon \tilde Y \to \pro(B)$ over $y \in \pro(B)$ is the quadric defined by $Q_y \in \Sym^2(\FFF^\dual|_y)$. Moreover, the fibre is smooth if and only if $Q_y$ is non-degenerate. The discriminant locus $D$, defined as the degeneracy locus of $Q$, is a section of $\det Q\colon \det \FFF \to \det \FFF^\dual \otimes \sheaf{\pro(B)}(k+2)$, which is identified with the line bundle $(\det \FFF^\dual)^2 \otimes \sheaf{\pro(B)}(k+2) \cong \sheaf{\pro(B)}(k+4)$.

    By restricting the relation $h' = H' - E'$ on $\tilde Y$, we obtain $E = H - h$. Finally, the canonical divisor of $\tilde Y$ is given by 
    \[K_{\tilde Y} = \sigma^*K_Y + (n-k-1)E = -(n-1)H + (n-k-1)E = -kH - (n-k-1)h. \qedhere\]
\end{proof}
By \cite{kuzQuaFib}, the quadric fibration $\pi\colon \tY \to \pro(B)$ gives rise to a sheaf of Clifford algebras on $\pro(B)$, with the even part $\CCC_0$ and the odd part $\CCC_1$ given by 
\[\CCC_0 \coloneqq \bigoplus_{m=0}^\infty {\textstyle\bigwedge^{2m}} \FFF \otimes \sheaf{\pro(B)}(-m), \qquad \CCC_1 \coloneqq \bigoplus_{m=0}^\infty {\textstyle\bigwedge^{2m+1}} \FFF \otimes \sheaf{\pro(B)}(-m).\]
Recall that we set $\CCC_{2\ell} \coloneqq \CCC_0 \otimes\sheaf{\pro(B)}(\ell)$ and $\CCC_{2\ell+1} \coloneqq \CCC_1 \otimes\sheaf{\pro(B)}(\ell)$ for any $\ell \in \Z$.

The Abelian category and the derived category of coherent right $\CCC_0$-modules are denoted by $\Coh(\pro(B),\CCC_0)$ and $\Dcatb(\pro(B),\CCC_0)$. These categories are introduced and studied by Kuznetsov (\cite{kuzQuaFib}). We will frequently use the following result from \emph{loc.\ cit.}
\begin{lemma}[{\cite[Lemma 3.8, Corollary 3.9]{kuzQuaFib}}][lem:C_Hom]
    For any $k,\ell \in \Z$, the sheaves $\CCC_k$ are flat over $\CCC_0$, and we have the isomorphisms 
    \vspace*{-0.5em}\[
    \Rsf\!\SHom_{\CCC_0}(\CCC_k,\CCC_\ell) \cong \CCC_{\ell-k}; \qquad \CCC_k \otimes_{\CCC_0} \CCC_{\ell} \cong \CCC_{k+\ell}.\]
\end{lemma}

\begin{setting}[][set:cubic_5]
    In this chapter, we focus on the case $n = 5$ and $k = 2$. Every smooth cubic 5-fold contains a 2-plane $\varPi_0 \cong \pro^2$; moreover, its Fano scheme of planes $\cF_2(Y)$ is a two-dimensional projective variety \cite[Proposition 1.8]{col86}. Blowing up a plane $\varPi_0 \subset Y$ induces a quadric surface fibration $\tilde Y \to \pro^3$. 
    \begin{equation}\label{diag:cubic_5_quad_fib}
        \begin{tikzcd}[ampersand replacement=\&,cramped]
        \& E \& {\tilde Y} \& {\pro_{\pro^3}(\sheaf{}^{\oplus 3} \oplus \sheaf{}(-1))} \\
        \varPi_0 \& Y \& {\pro^6} \&\& {\pro^3}
        \arrow["\iota", hook, from=1-2, to=1-3]
        \arrow["p"', from=1-2, to=2-1]
        \arrow["\alpha", hook, from=1-3, to=1-4]
        \arrow["\sigma"', from=1-3, to=2-2]
        \arrow["\pi", from=1-3, to=2-5]
        \arrow["\tau"', crossing over, from=1-4, to=2-3]
        \arrow["q", from=1-4, to=2-5]
        \arrow["j", hook, from=2-1, to=2-2]
        \arrow[hook, from=2-2, to=2-3]
    \end{tikzcd}
    \end{equation}
    The even and odd parts of the Clifford algebra associated to this quadric fibration are given by 
    \begin{equation}
        \CCC_0 = \sheaf{} \oplus \sheaf{}(-1)^{\oplus 3} \oplus \sheaf{}(-2)^{\oplus 3} \oplus \sheaf{}(-3), \qquad \CCC_1 = \sheaf{}^{\oplus 3} \oplus \sheaf{}(-1)^{\oplus 2} \oplus \sheaf{}(-2)^{\oplus 3}. \label{eq:Cliff_01}
    \end{equation} 
\end{setting}

\subsection{Mutations in the derived category}

\paragraph{Basic properties of mutations.} We recall the definition of mutation functors. For the properties one can refer to \cite[\S 3]{BLMS} or \cite[\S 7.1.3]{huy23}.

\begin{definition}
    Let $\cT$ be a triangulated category, and $\cD \subset \cT$ an admissible subcategory, with the inclusion functor $\beta_*\colon \cD \hookrightarrow \cT$. We have the adjunction triple $\beta^* \dashv \beta_* \dashv \beta^!$.
    \begin{enumerate}[nosep]
        \item $\lmut{\cD} \colon \cT \to \cD^{\perp}$, called the \textbf{left mutation} of $\cD$, gives the distinguished triangle in $\cT$:
              \[\begin{tikzcd}[ampersand replacement=\&]
                      {\beta_*\beta^!F} \& F \& {\lmut{\cD}F} \& {}
                      \arrow["{\operatorname{ev}}", from=1-1, to=1-2]
                      \arrow[from=1-2, to=1-3]
                      \arrow["{+1}", from=1-3, to=1-4]
                  \end{tikzcd}\]
        \item $\rmut{\cD} \colon \cT \to {}^{\perp}\cD$, called the \textbf{right mutation} of $\cD$, gives the distinguished triangle in $\cT$:
              \[\begin{tikzcd}[ampersand replacement=\&]
                      {\rmut{\cD}F} \& F \& {\beta_*\beta^*F} \& {}
                      \arrow[from=1-1, to=1-2]
                      \arrow["{\operatorname{coev}}", from=1-2, to=1-3]
                      \arrow["{+1}", from=1-3, to=1-4]
                  \end{tikzcd}\]
    \end{enumerate}
    Strictly speaking, $\lmut{\cD}F$ and $\rmut{\cD}F$ should be embedded into $\cT$ via the corresponding inclusion functors.
\end{definition}
If $\cD = \ord{E}$ for $E \in \cT$ an exceptional object, the mutation functors $\lmut{E}$ and $\rmut{E}$ can be described explicitly as
\begin{align}
    \lmut{E}F &= \cone\ab( \begin{tikzcd}[ampersand replacement=\&, cramped]
                {\displaystyle\bigoplus_{k \in \Z}\Hom(E,F[k]) \otimes E[-k]} \& F
                \arrow["{\operatorname{ev}}", from=1-1, to=1-2]
            \end{tikzcd} )  \in \ord{E}^\perp; \\
    \rmut{E}F &= \cone\ab( \begin{tikzcd}[ampersand replacement=\&, cramped]
                F \& {\displaystyle\bigoplus_{k \in \Z}\Hom(F,E[k])^\dual \otimes E[k]}
                \arrow["{\operatorname{coev}}", from=1-1, to=1-2]
            \end{tikzcd} )[-1] \in {}^\perp\!\ord{E}.
\end{align}

We can list some immediate properties of mutation functors.

\begin{lemma}[][lem:mutations_basic]
    Let $\cD \subset \cT$ be an admissible subcategory, and $F \in \cT$ any object.
    \begin{enumerate}[dense]
        \item If $\varphi$ is an auto-equivalence of $\cT$, then $\varphi \circ \lmut{\cD} \cong \lmut{\varphi(\cD)} \circ \varphi$, and $\varphi \circ \rmut{\cD} \cong \rmut{\varphi(\cD)} \circ \varphi$.

        \item $\lmut{\cD}|_{{}^{\perp}\cD}$ and $\rmut{\cD}|_{\cD^{\perp}}$ are mutually inverse. In particular, if $F \in \cD^{\perp}$, then $\lmut{\cD}F \cong F$ and $\lmut{\cD}\rmut{\cD}F \cong F$; and if $F \in {}^{\perp}\cD$, then $\rmut{\cD}F \cong F$ and $\rmut{\cD}\lmut{\cD}F \cong F$.

        \item Let $A,B \in \cT$ be exceptional objects such that $\RHom(A,B) = 0$. Then $\lmut{A} \circ \rmut{B} \cong \rmut{B} \circ \lmut{A}$.
    \end{enumerate}
\end{lemma}

The following result is very useful in computing mutations.

\begin{lemma}[Mutation triangle][lem:mut_trig]
    Let $A,B,C \in \cT$ be exceptional objects. Assume that $\RHom(B,A) = 0$ and there exists a distinguished triangle:
    \begin{equation}
        \begin{tikzcd}[ampersand replacement=\&,cramped]
            A \& B \& C \& {A[1].}
            \arrow["f", from=1-1, to=1-2]
            \arrow["g", from=1-2, to=1-3]
            \arrow["{\delta}", from=1-3, to=1-4]
        \end{tikzcd} \label{trig:mut_lem}
    \end{equation}
    Then we have the following mutation relations:
    \begin{align}
        \lmut{B}C &\cong A[1]; & \lmut{C}A &\cong B; 
        & \lmut{A}B &\cong C; \\
        \rmut{C}B &\cong A; & \rmut{A}C &\cong B; & \rmut{B}A &\cong C[-1].
    \end{align}
\end{lemma}

\paragraph{Semi-orthogonal decompositions of the cubic 5-fold.}
Now return to \cref{set:cubic_5}, where $\tilde Y$ is the blow-up of the smooth cubic 5-fold $Y$ along a plane $\varPi_0$. The goal of this subsection is to compare two semi-orthogonal decompositions. This comparison produces an admissible embedding $\Ku(Y) \hookrightarrow \Dcatb(\pro^3,\CCC_0)$ (\cref{thm:psi}).

On one side, since $\sigma\colon \tilde Y \to Y$ is a smooth blow-up, Orlov's formula (\cite[Theorem 2.6]{orlov92}) gives the semi-orthogonal decomposition
\begin{equation}\label{SOD:smooth_blowup}
    \Dcatb(\tilde Y)
         = \ord{\sigma^*\Dcatb(Y),\, \iota_*p^*\Dcatb(\varPi_0),\, \iota_*(p^*\Dcatb(\varPi_0) \otimes \sheaf{E}(-E))}. 
\end{equation}
On the other hand, since $\pi\colon \tilde Y \to \pro^3$ is a quadric fibration, by Kuznetsov's formula (\cite[Theorem 4.2]{kuzQuaFib}) there is a fully faithful functor $\varPhi\colon \Dcatb(\pro^3,\CCC_0) \to \Dcatb(\tilde Y)$ that induces the semi-orthogonal decomposition:
\begin{equation}\label{SOD:quad_fib}
    \Dcatb(\tilde Y)
     = \ord{\varPhi\Dcatb(\pro^3,\CCC_0),\, \pi^*\Dcatb(\pro^3),\, \pi^*\Dcatb(\pro^3) \otimes \sheaf{\tilde{Y}}(H)}.
\end{equation}
By \cite[Lemma 4.10]{kuzQuaFib}, the left adjoint $\varPsi\colon \Dcatb(\tilde Y) \to \Dcatb(\pro^3,\CCC_0)$ of $\varPhi$ is given by $F \longmapsto \pi_*(F \otimes \cE \otimes \sheaf{\tilde Y}(h))[2]$, where $\cE$ is a right $\pi^*\CCC_0$-module that fits into the short exact sequence of right $q^*\CCC_0$-modules:
\begin{equation}\label{ses:E}
    \begin{tikzcd}[ampersand replacement=\&,cramped]
            0 \& {q^*\CCC_{-1}(-3H')} \& {q^*\CCC_0(-2H')} \& {\alpha_*\cE} \& 0.
            \arrow[from=1-1, to=1-2]
            \arrow[from=1-2, to=1-3]
            \arrow[from=1-3, to=1-4]
            \arrow[from=1-4, to=1-5]
        \end{tikzcd} 
\end{equation}
Note that the sheaf $\cE$ on $\tilde Y$ is locally free of rank $4$ as an $\sheaf{\tilde Y}$-module by \cite[Lemma 4.7]{kuzQuaFib}.
 
\medskip
The main result of this section is the following. 
\begin{theorem}[][thm:psi]
    The derived category $\Dcatb(\pro^3,\CCC_0)$ admits the following semi-orthogonal decomposition
    \begin{equation}
        \Dcatb(\pro^3,\CCC_0) \simeq \ord{\Ku(\pro^3,\CCC_0),\, \CCC_1,\, \CCC_2}, \label{eq:twist_P3_SOD}
    \end{equation}
    such that the functor $\varPsi\circ \sigma^*\colon \Ku(Y) \to \Dcatb(\pro^3,\CCC_0)$ induces an equivalence between $\Ku(Y)$ and $\Ku(\pro^3,\CCC_0)$.
\end{theorem}

The proof of \cref{thm:psi} will take up the rest of this section. First we study the orthogonality and mutation relations for line bundles and torsion sheaves on $\Dcatb(\tY)$.

\begin{lemma}[][lem:mut]
    The following orthogonality and mutation relations hold for line bundles and torsion sheaves on $\tilde Y$.
    \begin{enumerate}[dense]
        \item $\iota_*\sheaf{E}$ is completely orthogonal to $\ord{\sheaf{\tilde Y}(H),\, \sheaf{\tilde Y}(2H)}$;
        \item $\lmut{\sheaf{\tilde Y}}\iota_*\sheaf{E} \cong \sheaf{\tilde Y}(-H+h)[1]$;
        \item $\rmut{\sheaf{\tilde Y}(-H+h)}\iota_*\sheaf{E} \cong \sheaf{\tilde Y}$;
        \item $\sheaf{\tilde Y}$ is completely orthogonal to $\ord{\sheaf{\tilde Y}(-H+2h),\, \sheaf{\tilde Y}(-H+3h)}$.
    \end{enumerate}
\end{lemma}
\begin{proof}
    Since $E = H-h$ in $\Pic \tY$, we have the following short exact sequence on $\tilde Y$:
    \begin{equation}
        \begin{tikzcd}[ampersand replacement=\&,cramped]
            0 \& {\sheaf{\tilde Y}(-H+h)} \& \sheaf{\tilde Y} \& {\iota_*\sheaf{E}} \& 0.
            \arrow[from=1-1, to=1-2]
            \arrow[from=1-2, to=1-3]
            \arrow[from=1-3, to=1-4]
            \arrow[from=1-4, to=1-5]
            \end{tikzcd} \label{ses:lem_mut}
    \end{equation}
    \begin{enumerate}
        \item It is clear from the SOD \eqref{SOD:smooth_blowup} that $\iota_*\sheaf{E}$ is left orthogonal to $\ord{\sheaf{\tilde Y}(H),\, \sheaf{\tilde Y}(2H)}$. Conversely, we apply $\Hom(\sheaf{\tilde Y}(aH),-)$ to \eqref{ses:lem_mut}, where $a = 1,2$, and consider the long exact sequence. The middle term is given by
        \[\Ext^\bullet(\sheaf{\tilde Y}(aH),\sheaf{\tilde Y}) = \Hlg^\bullet(\tilde Y, \sheaf{\tilde Y}(-aH)) = \Hlg^\bullet(Y,\sheaf{Y}(-a)) = 0, \qquad a \in \{1,2\}.\]
        whereas the left term is given by
        \[\Ext^\bullet(\sheaf{\tilde Y}(aH),\sheaf{\tilde Y}(-H+h)) = \Hlg^\bullet(\tilde Y, \sheaf{\tilde Y}(-(a+1)H+h)) = 
        \Hlg^\bullet(\tilde Y, \sheaf{\tilde Y}(-aH-E))
        =\Hlg^\bullet(Y,\ideal_{\varPi_0}(-a)).\]
        Note that $\ideal_{\varPi_0}(-a)$ fits into the short exact sequence 
        \begin{equation}
            \begin{tikzcd}[ampersand replacement=\&,cramped]
                0 \& {\ideal_{\varPi_0}(-a)} \& \sheaf{Y}(-a) \& {\sheaf{\varPi_0}(-a)} \& 0.
                \arrow[from=1-1, to=1-2]
                \arrow[from=1-2, to=1-3]
                \arrow[from=1-3, to=1-4]
                \arrow[from=1-4, to=1-5]
                \end{tikzcd} 
        \end{equation} 
        where $\sheaf{Y}(-a)$ and $\sheaf{\varPi_0}(-a)$ have no cohomology for $a=1,2$. Then $\Hlg^\bullet(Y,\ideal_{\varPi_0}(-a)) = 0$. We deduce from the long exact sequence of \eqref{ses:lem_mut} that $\Ext^\bullet(\sheaf{\tilde Y}(aH),\iota_*\sheaf{E}) = 0$. Hence $\iota_*\sheaf{E}$ is also right orthogonal to $\ord{\sheaf{\tilde Y}(H),\, \sheaf{\tilde Y}(2H)}$.

        \item In view of \cref{lem:mut_trig} and the sequence \eqref{ses:lem_mut}, it suffices to check that $\RHom(\sheaf{\tY},\sheaf{\tY}(-H+h)) = 0$. But this has already been encoded in the SOD \eqref{SOD:quad_fib}. 
        
        \item This mutation also follows from \cref{lem:mut_trig} and \eqref{ses:lem_mut}.
        
        \item $\sheaf{\tilde Y}$ is right orthogonal to $\ord{\sheaf{\tilde Y}(-H+2h),\sheaf{\tilde Y}(-H+3h)}$, since 
        \begin{equation}
            \begin{aligned}
                \Ext^\bullet(\sheaf{\tilde Y}(-H+ah),\sheaf{\tilde Y}) 
                &= \Hlg^\bullet(\tilde Y, \sheaf{\tilde Y}(H-ah)) = \Hlg^\bullet(\pro^3, \FFF^\dual \otimes \sheaf{\pro^3}(-a)) \\ &= \Hlg^\bullet(\pro^3,\sheaf{\pro^3}(-a+1) \oplus \sheaf{\pro^3}(-a)^{\oplus 3}) = 0,
            \end{aligned}
        \end{equation}
        for $a \in \{2,3\}$; and $\sheaf{\tilde Y}$ is also left orthogonal to $\ord{\sheaf{\tilde Y}(-H+2h),\sheaf{\tilde Y}(-H+3h)}$, since the subcategory $\pi^*\Dcatb(\pro^3)$ is left orthogonal to $\pi^*\Dcatb(\pro^3) \otimes \sheaf{\tilde Y}(-H)$ by the Orlov blow-up formula. \qedhere
    \end{enumerate}    
\end{proof}

\begin{proposition}[][prop:SOD_big_mut]
    The subcategory $\varPhi\Dcatb(\pro^3,\CCC_0)$ of $\Dcatb(\tilde Y)$ admits the semi-orthogonal decomposition:
    \begin{equation}
        \varPhi\Dcatb(\pro^3,\CCC_0)
        =
        \langle
            \lmut{\cD}\sigma^*\Ku(Y),\, 
            \lmut{\cD'}\sheaf{\tilde Y}(3H),\, 
            \lmut{\cD'}\iota_*\sheaf{E}(H+h)
        \rangle,
    \end{equation}
    where
    \begin{equation}
        \begin{aligned}
            \cD &\coloneqq \langle \sheaf{\tilde Y}(-2h),\, \sheaf{\tilde Y}(-h),\, \sheaf{\tilde Y}(H-h) \rangle; \\ 
            \cD' &\coloneqq \langle \pi^*\Dcatb(\pro^3),\, \sheaf{\tilde Y}(H-h),\, \sheaf{\tilde Y}(H),\, \sheaf{\tilde Y}(H+h),\, \pi^*\Dcatb(\pro^3) \otimes \sheaf{\tY}(2H) \rangle.
        \end{aligned}
    \end{equation}
\end{proposition}

\begin{remark}
    The mutation functors are somewhat cumbersome at this stage, but we will show very soon that most of the terms are killed by the left adjoint functor $\varPsi$ of $\varPhi$; the only relevant term in later computations is the left mutation by $\pi^*\Dcatb(\pro^3) \otimes \sheaf{\tY}(2H)$.
\end{remark}

\begin{proof}
    Note that $\Pic \tilde Y = \Z H \oplus \Z h$. For notational convenience, we denote the sheaves $\sheaf{\tilde Y}(aH+bh)$ and $\iota_*\sheaf{E}(aH+bh)$ by $(a,b)$ and $[a,b]$ respectively, where $a,b \in \Z$. 

    Starting from the SOD \eqref{SOD:smooth_blowup} of $\Dcatb(\tilde Y)$, we perform the following sequence of mutations. In each step, the objects marked with a straight underline are mutated through the objects marked with a wavy underline. See also \cref{fig:vis_mut_cubic5} for the visualization of the mutation process.
    \begin{enumerate}[label=\textbf{Step \arabic*.}, leftmargin=*, nosep]
                \item We expand the SOD \eqref{SOD:smooth_blowup} as follows: for the first term, we pull-back the SOD \eqref{SOD:cubic_5} of $\Dcatb(Y)$ via $\sigma^*$; for the second term, we pull-back the full exceptional collection $\ord{\sheaf{}(1),\, \sheaf{}(2),\, \sheaf{}(3)}$ of $\Dcatb(\varPi_0)$; and for the third term, we pull-back the full exceptional collection $\ord{\sheaf{}(2),\, \sheaf{}(3),\, \sheaf{}(4)}$ of $\Dcatb(\varPi_0)$. Then we have the following SOD in the $(H,h)$-coordinates:
        \begin{equation}
            \begin{aligned}
                \Dcatb(\tilde Y)
                =
                \langle
                &\sigma^*\Ku(Y),\, (0,0),\, (1,0),\, (2,0),\, (3,0),\, [1,0],\, [2,0],\, [3,0],\, [1,1],\, [2,1],\, [3,1]
                \rangle .
            \end{aligned}
        \end{equation}
                        \item Left mutate $[1,0]$ through $\langle (2,0),\, (3,0) \rangle$, and then left mutate $[2,0]$ through $(3,0)$, using the orthogonality in \cref{lem:mut}.(1) in both cases:
        \begin{equation}
            \begin{aligned}
                \Dcatb(\tilde Y)
                =
                \langle
                &\sigma^*\Ku(Y),\, (0,0),\, (1,0),\, \underline{[1,0]},\, \uwave{(2,0),\, (3,0)},\, [2,0],\, [3,0],\, [1,1],\, [2,1],\, [3,1]
                \rangle . \\ 
                \langle
                &\sigma^*\Ku(Y),\, (0,0),\, (1,0),\, [1,0],\, (2,0),\, \underline{[2,0]},\, \uwave{(3,0)},\, [3,0],\, [1,1],\, [2,1],\, [3,1]
                \rangle .
            \end{aligned}
        \end{equation}
                        \item Left mutate $[a,0]$ through $(a,0)$ for $a=1,2,3$, using \cref{lem:mut}.(2). In the displayed SOD, we suppress the resulting shifts, since shifting an object does not change the admissible subcategory it generates:
        \begin{equation}
            \begin{aligned}
                \Dcatb(\tilde Y)
                =
                \langle
                &\sigma^*\Ku(Y),\, (0,0),\, \underline{(0,1)},\, \uwave{(1,0)},\, \underline{(1,1)},\, \uwave{(2,0)},\, \underline{(2,1)},\, \uwave{(3,0)},\, [1,1],\, [2,1],\, [3,1]
                \rangle .
            \end{aligned}
        \end{equation}
                        \item Move $(0,0)$ across $\sigma^*\Ku(Y)$, and then right mutate $(0,0)$ through its left orthogonal using the Serre functor $\mathsf S=(-\otimes (-2,-2))[5]$:
        \begin{equation}
            \begin{aligned}
                \Dcatb(\tilde Y)
                =
                \langle
                &\uwave{(0,0)},\, \underline{\rmut{(0,0)}\sigma^*\Ku(Y)},\, (0,1),\, (1,0),\, (1,1),\, (2,0),\, (2,1),\, (3,0),\, [1,1],\, [2,1],\, [3,1]
                \rangle \\ 
                =
                \langle
                &\uwave{\rmut{(0,0)}\sigma^*\Ku(Y),\, (0,1),\, (1,0),\, (1,1),\, (2,0),\, (2,1),\, (3,0),\, [1,1],\, [2,1],\, [3,1]},\, \underline{(2,2)}
                \rangle .
            \end{aligned}
        \end{equation}
                        \item Right mutate $[3,1]$ through $(2,2)$, using \cref{lem:mut}.(3):
        \begin{equation}
            \begin{aligned}
                \Dcatb(\tilde Y)
                =
                \langle
                &\rmut{(0,0)}\sigma^*\Ku(Y),\, (0,1),\, (1,0),\, (1,1),\, (2,0),\, (2,1),\, (3,0),\, [1,1],\, [2,1],\, \uwave{(2,2)},\, \underline{(3,1)}
                \rangle .
            \end{aligned}
        \end{equation}
                        \item Left mutate $(3,0)$ and $[1,1]$ through $\cB'=\langle (0,1),\, (1,0),\, (1,1),\, (2,0),\, (2,1) \rangle$. 
        This gives
        \begin{equation}
            \begin{aligned}
                \Dcatb(\tilde Y)
                =
                \langle
                &\rmut{(0,0)}\sigma^*\Ku(Y),\, \underline{\lmut{\cB'}(3,0)},\, \underline{\lmut{\cB'}[1,1]}, \\
                &\underbrace{(0,1),\, (1,0),\, (1,1),\, (2,0),\, (2,1)}_{=\cB'},\, [2,1],\, (2,2),\, (3,1)
                \rangle .
            \end{aligned}
        \end{equation}
                        \item Left mutate $[2,1]$ through $(2,1)$, using \cref{lem:mut}.(2):
        \begin{equation}
            \begin{aligned}
                \Dcatb(\tilde Y)
                =
                \langle
                &\rmut{(0,0)}\sigma^*\Ku(Y),\, \lmut{\cB'}(3,0),\, \lmut{\cB'}[1,1],\, (0,1),\, (1,0), \\
                &(1,1),\, (2,0),\, \underline{(1,2)},\, \uwave{(2,1)},\, (2,2),\, (3,1)
                \rangle .
            \end{aligned}
        \end{equation}
                        \item Transpose $(1,2)$ with $(2,0)$, using \cref{lem:mut}.(4):
        \begin{equation}
            \begin{aligned}
                \Dcatb(\tilde Y)
                =
                \langle
                &\rmut{(0,0)}\sigma^*\Ku(Y),\, \lmut{\cB'}(3,0),\, \lmut{\cB'}[1,1],\, (0,1),\, (1,0), \\
                &(1,1),\, \underline{(1,2)},\, \uwave{(2,0)},\, (2,1),\, (2,2),\, (3,1)
                \rangle .
            \end{aligned}
        \end{equation}
                        \item Left mutate $\langle (2,0),\, (2,1),\, (2,2),\, (3,1) \rangle$ through its right orthogonal, using the Serre functor.
        \begin{equation}
            \begin{aligned}
                \Dcatb(\tilde Y)
                =
                \langle
                &\underline{(0,-2),\, (0,-1),\, (0,0),\, (1,-1)},\, \uwave{\rmut{(0,0)}\sigma^*\Ku(Y),\, \lmut{\cB'}(3,0),\, \lmut{\cB'}[1,1]}, \\[-1ex]
                &\uwave{(0,1),\, (1,0),\, (1,1),\, (1,2)}
                \rangle .
            \end{aligned}
        \end{equation}
                        \item Left mutate the three objects after $\cB=\langle (0,-2),\, (0,-1),\, (0,0),\, (1,-1) \rangle$ through $\cB$:
        \begin{equation}
            \begin{aligned}
                \Dcatb(\tilde Y)
                =
                \langle
                &\underline{\lmut{\cB}\rmut{(0,0)}\sigma^*\Ku(Y),\, \lmut{\cB}\lmut{\cB'}(3,0),\, \lmut{\cB}\lmut{\cB'}[1,1]}, \\
                &\underbrace{(0,-2),\, (0,-1),\, (0,0),\, (1,-1)}_{=\cB},\, (0,1),\, (1,0),\, (1,1),\, (1,2)
                \rangle .
            \end{aligned}
        \end{equation}
                        \item Transpose $(0,1)$ with $(1,-1)$, using \cref{lem:mut}.(4):
        \begin{equation}
            \begin{aligned}
                \Dcatb(\tilde Y)
                =
                \langle
                &\lmut{\cB}\rmut{(0,0)}\sigma^*\Ku(Y),\, 
                 \lmut{\cB}\lmut{\cB'}(3,0),\, 
                 \lmut{\cB}\lmut{\cB'}[1,1], \\
                &(0,-2),\, (0,-1),\, (0,0),\, \underline{(0,1)},\, \uwave{(1,-1)},\, (1,0),\, (1,1),\, (1,2)
                \rangle .
            \end{aligned}
        \end{equation}
                        \item Using the pull-backs of full exceptional sequences $\pi^*\Dcatb(\pro^3) = \ord{(0,-2),\, (0,-1),\, (0,0),\, (0,1)}$ and $\pi^*\Dcatb(\pro^3) \otimes \sheaf{\tY}(H)= \ord{(1,-1),\, (1,0),\, (1,1),\, (1,2)}$, we have:
        \begin{equation}
            \Dcatb(\tilde Y)
                =
                \langle
                \lmut{\cB}\rmut{(0,0)}\sigma^*\Ku(Y),\, 
                 \lmut{\cB}\lmut{\cB'}(3,0),\, 
                 \lmut{\cB}\lmut{\cB'}[1,1],\, \pi^*\Dcatb(\pro^3),\, \pi^*\Dcatb(\pro^3) \otimes \sheaf{\tY}(H)
                \rangle .
        \end{equation}
    \end{enumerate}
    \begin{figure}[h]
        \centering
        \begin{tikzpicture}[
  x=.8cm,y=.8cm,
  >=Stealth,
  every node/.style={font=\scriptsize},    linept/.style={circle, draw=green!55!black, fill=white, line width=1pt, inner sep=1.5pt},    ghostpt/.style={circle, draw=gray!60, fill=white, line width=.55pt, inner sep=1.25pt},    torspt/.style={text=red!75!black, font=\scriptsize\bfseries, inner sep=0pt},    bothpt/.style={text=blue!45!red!85!black, font=\scriptsize\bfseries, inner sep=0pt},    axis/.style={black, line width=.55pt, -{Stealth[length=1.7mm,width=1.2mm]}},    gridline/.style={gray!28, line width=.23pt},    rmut/.style={orange!85!black, dashed, line width=.78pt,
    shorten <=2.5pt, shorten >=2.5pt,
    -{Stealth[length=1.75mm,width=1.25mm]}},        rmutspecial/.style={orange!85!black, dashed, line width=.78pt,
    shorten <=2.4pt, shorten >=2.4pt},      lmut/.style={orange!85!black, line width=.78pt,
    shorten <=2.5pt, shorten >=2.5pt,
    -{Stealth[length=1.75mm,width=1.25mm]}},      orth/.style={green!55!black, densely dashed, line width=.62pt, opacity=.76},    serre/.style={teal!65!black, densely dashed, line width=.75pt,
    shorten <=2.5pt, shorten >=2.5pt,
    -{Stealth[length=1.75mm,width=1.25mm]}},      orderguide/.style={green!55!black, line width=.58pt, opacity=.46,
    shorten <=2.0pt, shorten >=2.0pt,
    -{Stealth[length=1.8mm,width=1.15mm]}},      paneltitle/.style={font=\small\bfseries, anchor=south, fill=white, inner sep=1pt},
  tick/.style={font=\tiny, fill=white, inner sep=.45pt},    orthlabel/.style={font=\scriptsize, text=green!45!black, fill=white, inner sep=.6pt}
]

\def\Axes#1#2#3#4{  \draw[step=1, gridline] (#1,#3) grid (#2,#4);
  \draw[axis] (#1,0) -- ({#2+.35},0) node[right] {$H$};
  \draw[axis] (0,#3) -- (0,{#4+.35}) node[left] {$h$};
}
\def\Xticks#1{  \foreach \x in {#1}{
    \draw (\x,.06)--(\x,-.06) node[tick, below=2pt] {$\x$};
  }
}
\def\Yticks#1{  \foreach \y in {#1}{
    \draw (.06,\y)--(-.06,\y) node[tick, left=2pt] {$\y$};
  }
}
\def\Odot#1#2{\node[linept] at (#1,#2) {};}
\def\Gdot#1#2{\node[ghostpt] at (#1,#2) {};}
\def\Edot#1#2{\node[torspt] at (#1,#2) {$\times$};}
\def\Both#1#2{\node[linept] at (#1,#2) {};  \node[torspt] at (#1,#2) {$\times$};}
\def\PlotO#1{\foreach \x/\y in {#1}{\Odot{\x}{\y}}}
\def\PlotG#1{\foreach \x/\y in {#1}{\Gdot{\x}{\y}}}
\def\PlotE#1{\foreach \x/\y in {#1}{\Edot{\x}{\y}}}
\def\PlotB#1{\foreach \x/\y in {#1}{\Both{\x}{\y}}}

\begin{scope}[xshift=0cm,yshift=0cm]
  \Axes{-.35}{3.35}{-.55}{2.6}
  \Xticks{0,1,2,3}
  \Yticks{1,2}
  \node[paneltitle] at (1.5,2.6) {Step $1$};

    \PlotO{0/0,1/0,2/0,3/0}
  \PlotE{1/0,2/0,3/0,1/1,2/1,3/1}
\end{scope}
\begin{scope}[xshift=5cm,yshift=0cm]
  \Axes{-.35}{3.35}{-.55}{2.6}
  \Xticks{0,1,2,3}
  \Yticks{1,2}
  \node[paneltitle] at (1.5,2.6) {Step $2\to3$};

    \draw[lmut] (1,0) -- (0,1);
  \draw[lmut] (2,0) -- (1,1);
  \draw[lmut] (3,0) -- (2,1);

    \PlotO{0/0,0/1,1/0,1/1,2/0,2/1,3/0}
  \PlotE{1/1,2/1,3/1}
\end{scope}
\begin{scope}[xshift=10cm,yshift=0cm]
  \Axes{-.35}{3.35}{-.55}{2.6}
  \Xticks{0,1,2,3}
  \Yticks{1,2}
  \node[paneltitle] at (1.5,2.6) {Step $4$};

    \draw[rmut] (0,0) .. controls (0.5,1.8) and (1.5,1.8) .. (2,2);

    \PlotO{0/1,1/0,1/1,2/0,2/1,2/2,3/0}
  \PlotE{1/1,2/1,3/1}
  \PlotG{0/0}
\end{scope}
\begin{scope}[xshift=0cm,yshift=-3.5cm]
  \Axes{-.35}{3.35}{-.55}{2.6}
  \Xticks{0,1,2,3}
  \Yticks{1,2}
  \node[paneltitle] at (1.5,2.6) {Step $5$};

    \draw[rmutspecial] (3,1) -- (2,2);

    \PlotO{0/1,1/0,1/1,2/0,2/1,2/2,3/0,3/1}
  \PlotE{1/1,2/1}
\end{scope}
\begin{scope}[xshift=5cm,yshift=-3.5cm]
  \Axes{-.35}{3.35}{-.55}{2.6}
  \Xticks{0,1,2,3}
  \Yticks{1,2}
  \node[paneltitle] at (1.5,2.6) {Step $6$};

    \PlotO{0/1,1/0,1/1,2/0,2/1,2/2,3/1}
  \PlotE{2/1}
\end{scope}
\begin{scope}[xshift=10cm,yshift=-3.5cm]
  \Axes{-.35}{3.35}{-.55}{2.6}
  \Xticks{0,1,2,3}
  \Yticks{1,2}
  \node[paneltitle] at (1.5,2.6) {Step $7$};

      \draw[lmut] (2,1) -- (1,2);

    \PlotO{0/1,1/0,1/1,1/2,2/0,2/1,2/2,3/1}
\end{scope}
\begin{scope}[xshift=0cm,yshift=-7cm]
  \Axes{-.35}{3.35}{-.55}{2.6}
  \Xticks{0,1,2,3}
  \Yticks{1,2}
  \node[paneltitle] at (1.5,2.6) {Step $8$};

  \draw[orderguide] (0,1) -- (1,0);
  \draw[orderguide] (1,0) -- (1,2);
  \draw[orderguide] (2,0) -- (2,2);
  \draw[orderguide] (2,2) -- (3,1);
  \draw[orth] (1,2) -- (2,0);

    \PlotO{0/1,1/0,1/1,1/2,2/0,2/1,2/2,3/1}
\end{scope}
\begin{scope}[xshift=5cm,yshift=-7cm]
  \Axes{-.35}{3.35}{-2.55}{2.6}
  \Xticks{0,1,2,3}
  \Yticks{-2,-1,1,2}
  \node[paneltitle] at (1.5,2.6) {Step $9$};

    \draw[lmut] (2,2) .. controls (1.5,0.2) and (0.5,0.2) .. (0,0);
  \draw[lmut] (2,1) .. controls (1.5,-0.8) and (0.5,-0.8) .. (0,-1);
  \draw[lmut] (2,0) .. controls (1.5,-1.8) and (0.5,-1.8) .. (0,-2);
  \draw[lmut] (3,1) .. controls (2.5,-0.8) and (1.5,-0.8) .. (1,-1);

    \PlotO{0/-2,0/-1,0/0,0/1,1/-1,1/0,1/1,1/2}
  \PlotG{2/0,2/1,2/2,3/1}
\end{scope}
\begin{scope}[xshift=10cm,yshift=-7cm]
  \Axes{-.35}{3.35}{-2.55}{2.6}
  \Xticks{0,1,2,3}
  \Yticks{-2,-1,1,2}
  \node[paneltitle] at (1.5,2.6) {Step $10\to12$};

  \draw[orderguide] (0,-2) -- (0,1);
  \draw[orderguide] (1,-1) -- (1,2);
  \draw[orth] (1,-1) -- (0,1);

    \PlotO{0/-2,0/-1,0/0,0/1,1/-1,1/0,1/1,1/2}
\end{scope}
\end{tikzpicture}
        \caption[Visualizing the proof of \cref{prop:SOD_big_mut}.]{Visualizing the proof of \cref{prop:SOD_big_mut}.  The coordinate $(a,b)$ represents the exceptional objects
        $\sheaf{\tY}(aH+bh)$ and $\iota_*\sheaf{E}(aH+bh)$ appearing in the SOD. Green circles denote line bundles,
        red crosses denote torsion sheaves, and `$\otimes$' marks coordinates where both occur.  
        Orange arrows and dashed lines denote 
        mutation operations.
        Green dashed segments indicate complete
        orthogonality, while the green guides indicate the order of the SODs.}
        \label{fig:vis_mut_cubic5}
    \end{figure}
    Comparing the last line with the SOD \eqref{SOD:quad_fib}, we find that $\varPhi\Dcatb(\pro^3,\CCC_0)$ is equivalent to the subcategory of $\Dcatb(\tilde Y)$ defined by
    \begin{equation}
        \ord{\pi^*\Dcatb(\pro^3),\, \pi^*\Dcatb(\pro^3) \otimes \sheaf{\tY}(H)}^{\perp} = 
            \ord{\lmut{\cB}\rmut{(0,0)}\sigma^*\Ku(Y),\, 
            \lmut{\cB}\lmut{\cB'}(3,0),\, 
            \lmut{\cB}\lmut{\cB'}[1,1]}.
    \end{equation}
    Finally, we simplify the mutation functors. For the first term, since $(0,0) \in \ord{(1,-1)}^{\perp}$, by \cref{lem:mutations_basic}.(3), we have 
    \begin{align}
        \lmut{\cB}\rmut{(0,0)}\sigma^*\Ku(Y)
        &= \lmut{\ord{(0,-2),\, (0,-1),\, (0,0),\, (1,-1)}}\rmut{(0,0)}\sigma^*\Ku(Y) \\ 
        &= \lmut{\ord{(0,-2),\, (0,-1),\, (0,0)}}\rmut{(0,0)}\lmut{(1,-1)}\sigma^*\Ku(Y).
    \end{align}
    Since $\lmut{(1,-1)}\sigma^*\Ku(Y)$ lies in $\langle(0,0)\rangle^\perp$, by \cref{lem:mutations_basic} we have
    \begin{equation}
        \lmut{\cB}\rmut{(0,0)}\sigma^*\Ku(Y) = \lmut{\ord{(0,-2),\, (0,-1),\, (1,-1)}}\sigma^*\Ku(Y)
        = \lmut{\cD}\sigma^*\Ku(Y).
    \end{equation}
    For the second and the third term, we note that $\ord{(3,0),\, [1,1]} \subset \ord{(2,2),\, (2,3)}^\perp$. Therefore the mutation functor $\lmut{\ord{(2,2),\, (2,3)}}$ acts as the identity on $\ord{(3,0),\, [1,1]}$. Hence we have 
    \begin{align}
        &\lmut{\cB}\lmut{\cB'}\ord{(3,0),\, [1,1]}
        = \lmut{\cB}\lmut{\cB'}\lmut{\ord{(2,2),\, (2,3)}}\ord{(3,0),\, [1,1]} \\
        &= \lmut{\ord{(0,-2),\, (0,-1),\, (0,0),\, (0,1),\, (1,-1),\, (1,0),\, (1,1),\, (2,0),\, (2,1),\, (2,2),\, (2,3)}}\ord{(3,0),\, [1,1]} \\
        &= \lmut{\pi^*\Dcatb(\pro^3)}\lmut{\ord{(1,-1),\, (1,0),\, (1,1)}}\lmut{\pi^*\Dcatb(\pro^3) \otimes \sheaf{\tY}(2H)}\ord{(3,0),\, [1,1]} \\ 
        &= \lmut{\cD'}\ord{(3,0),\, [1,1]}. 
    \end{align}
    We conclude that
    \[ \varPhi\Dcatb(\pro^3,\CCC_0)
        =
        \langle
            \lmut{\cD}\sigma^*\Ku(Y),\, 
            \lmut{\cD'}(3,0),\, 
            \lmut{\cD'}[1,1]
        \rangle. \qedhere \]
\end{proof}

The proof of \cref{thm:psi} is reduced to computing the action of $\varPsi$ on the SOD factors of $\varPhi\Dcatb(\pro^3,\CCC_0)$ as given in the previous proposition. This computation is carried out in \cref{lem:psi_kill,lem:lmut_B,lem:LB1}.

\begin{lemma}[][lem:psi_kill]
    For $a \in \{0,1\}$ and $b \in \Z$, $\varPsi(\sheaf{\tilde Y}(aH+bh)) = 0$. In particular, we have isomorphisms of functors:\vspace{-\parskip}
    \[\varPsi \circ \lmut{\sheaf{\tilde Y}(aH+bh)} \cong  \varPsi \cong \varPsi \circ \rmut{\sheaf{\tilde Y}(aH+bh)}.\]
\end{lemma}
\begin{proof}
    Note that $\varPsi(\sheaf{\tilde Y}(aH+bh)) = \pi_*(\cE \otimes \sheaf{\tilde Y}(aH + (b+1)h))[2]$. Twisting the sequence \eqref{ses:E} by $\sheaf{\tilde Y}(aH + (b+1)h)$ and pushing forward along $q\colon \pro(\FFF) \to \pro^3$, we obtain the distinguished triangle:
    \[\begin{tikzcd}[cramped, sep = small]
            {\CCC_{-1}(b+1) \otimes q_*\sheaf{\pro(\FFF)}((a-3)H')} & {\CCC_{0}(b+1) \otimes q_*\sheaf{\pro(\FFF)}((a-2)H')} & {\varPsi(\sheaf{\tilde Y}(aH+bh))[-2]} & {}
            \arrow[from=1-1, to=1-2]
            \arrow[from=1-2, to=1-3]
            \arrow["{+1}", from=1-3, to=1-4]
        \end{tikzcd}\]
    For $a=0,1$, since $q$ is a $\pro^3$-fibration, we have $q_*\sheaf{\pro(\FFF)}((a-2)H') = q_*\sheaf{\pro(\FFF)}((a-3)H') = 0$. Hence $\varPsi(\sheaf{\tilde Y}(aH+bh)) = 0$.

    For any $F \in \Dcatb(\tilde Y)$, the left mutation $\lmut{\sheaf{\tilde Y}(aH+bh)}$ followed by $\varPsi$ gives the distinguished triangle
    \[\begin{tikzcd}
            {\RHom(\sheaf{\tilde Y}(aH+bh),F) \otimes \varPsi(\sheaf{\tilde Y}(aH+bh))} & {\varPsi(F)} & {\varPsi \circ \lmut{\sheaf{\tilde Y}(aH+bh)}F} & {}
            \arrow[from=1-1, to=1-2]
            \arrow[from=1-2, to=1-3]
            \arrow["{+1}", from=1-3, to=1-4]
        \end{tikzcd}\]
    Hence there is a functorial isomorphism $\varPsi(F) \cong \varPsi \circ \lmut{\sheaf{\tilde Y}(aH+bh)}F$. That is, $\varPsi = \varPsi \circ \lmut{\sheaf{\tilde Y}(aH+bh)}$. Similarly we also have $\varPsi = \varPsi \circ \rmut{\sheaf{\tilde Y}(aH+bh)}$.
\end{proof}
By the lemma, we have the isomorphisms of functors:
\begin{align*}
    \varPsi\circ \lmut{\cD}  \cong \varPsi; \qquad\text{and}\qquad
    \varPsi\circ \lmut{\cD'} \cong \varPsi \circ \lmut{\beta_*\Dcatb(\pro^3)},
\end{align*}
where $\beta_*\colon \Dcatb(\pro^3) \to \Dcatb(\tY)$ is the fully faithful functor given by $F \longmapsto \pi^*F \otimes \sheaf{\tY}(2H)$.

The above observation simplifies the SOD of $\Dcatb(\pro^3,\CCC_0)$ as follows:
\begin{align}
    \Dcatb(\pro^3,\CCC_0) &\simeq \varPsi\ab(\ord{\pi^*\Dcatb(\pro^3),\, \pi^*\Dcatb(\pro^3) \otimes \sheaf{\tY}(H)}^{\perp}) \\
    &= \ord{\varPsi\sigma^*\Ku(Y),\, 
    \varPsi\circ\lmut{\beta_*\Dcatb(\pro^3)}\sheaf{\tilde Y}(3H), \
    \varPsi\circ\lmut{\beta_*\Dcatb(\pro^3)}\iota_*\sheaf{E}(H+h)}.  \label{equ:psi_K}
\end{align}

\begin{lemma}[][lem:lmut_B]
    For $F \in \Dcatb(\tilde Y)$, the left mutation $\lmut{\beta_*\Dcatb(\pro^3)}F$ of $F$ through $\beta_*\Dcatb(\pro^3) = \pi^*\Dcatb(\pro^3) \otimes \sheaf{\tY}(2H)$ fits into the distinguished triangle:
    \[\begin{tikzcd}
            {\pi^*\pi_*(F \otimes \sheaf{\tilde Y}(-2H)) \otimes \sheaf{\tilde Y}(2H)} & F & {\lmut{\beta_*\Dcatb(\pro^3)}F} & {}
            \arrow[from=1-1, to=1-2]
            \arrow[from=1-2, to=1-3]
            \arrow["{+1}", from=1-3, to=1-4]
        \end{tikzcd}\]
\end{lemma}
\begin{proof}
        The fully faithful functor $\beta_*$ admits the right adjoint given by $\beta^!\colon \Dcatb(\tilde Y) \to \Dcatb(\pro^3)$, $F \longmapsto \pi_*(F \otimes \sheaf{\tilde Y}(-2H))$. By definition, the left mutation of $F \in \Dcatb(\tilde Y)$ fits into the distinguished triangle
    \[\begin{tikzcd}
            {\beta_*\beta^!F} & F & {\lmut{\beta_*\Dcatb(\pro^3)}F} & {}
            \arrow[from=1-1, to=1-2]
            \arrow[from=1-2, to=1-3]
            \arrow["{+1}", from=1-3, to=1-4]
        \end{tikzcd}\]
        where $\beta_*\beta^!F = \pi^*\pi_*(F \otimes \sheaf{\tilde Y}(-2H)) \otimes \sheaf{\tilde Y}(2H)$.
\end{proof}
We finish the proof of \cref{thm:psi} by computing $\varPsi\circ\lmut{\beta_*\Dcatb(\pro^3)}\sheaf{\tilde Y}(3H)$ and $\varPsi\circ\lmut{\beta_*\Dcatb(\pro^3)}\iota_*\sheaf{E}(H+h)$ in the following lemma.
\begin{lemma}[][lem:LB1] 
    \begin{enumerate}[dense]
    \item $\varPsi(\sheaf{\tilde Y}(2H)) \cong \CCC_2[2].$
        \item $\varPsi\circ\lmut{\beta_*\Dcatb(\pro^3)}\sheaf{\tilde Y}(3H) \cong \CCC_1[3].$
    \item $\varPsi\circ\lmut{\beta_*\Dcatb(\pro^3)}\iota_*\sheaf{E}(H+h) \cong \CCC_2[2].$
    \end{enumerate}
\end{lemma}
\begin{proof}
    \begin{enumerate}
        \item $\varPsi(\sheaf{\tilde Y}(2H)) = \pi_*(\cE \otimes \sheaf{\tilde Y}(2H+h))[2]$ fits into the distinguished triangle:
        \[\begin{tikzcd}[row sep=tiny]
            {\CCC_{-1} \otimes \pi_*\sheaf{\tilde Y}(-H+h)[2]} & {\CCC_0 \otimes \pi_*\sheaf{\tilde Y}(h)[2]} & {\pi_*(\cE \otimes \sheaf{\tilde Y}(2H+h))[2]} & {} 
            \arrow[from=1-1, to=1-2]
            \arrow[from=1-2, to=1-3]
            \arrow["{+1}", from=1-3, to=1-4]
        \end{tikzcd}\]
        Since $\pi_*\sheaf{\tilde Y}(-H) = 0$, we have 
        \begin{equation}
            \varPsi(\sheaf{\tilde Y}(2H)) = \pi_*(\cE \otimes \sheaf{\tilde Y}(2H+h))[2] \cong \CCC_0 \otimes \pi_*\sheaf{\tilde Y}(h)[2] \cong \CCC_2[2].
        \end{equation}
        \item By \cref{lem:lmut_B}, $\varPsi\circ\lmut{\beta_*\Dcatb(\pro^3)}\sheaf{\tilde Y}(3H)$ fits into the distinguished triangle:
    \begin{equation}
        \begin{tikzcd}
                {\varPsi(\pi^*\pi_*\sheaf{\tilde Y}(H) \otimes \sheaf{\tilde Y}(2H))} & \varPsi(\sheaf{\tilde Y}(3H)) & {\varPsi\circ\lmut{\beta_*\Dcatb(\pro^3)}\sheaf{\tilde Y}(3H)} & {}
                \arrow[from=1-1, to=1-2]
                \arrow[from=1-2, to=1-3]
                \arrow["{+1}", from=1-3, to=1-4]
            \end{tikzcd} \label{ses:H}
    \end{equation}
    Note that $\pi^*\pi_*\sheaf{\tilde Y}(H) \otimes \sheaf{\tilde Y}(2H) = \pi^*\FFF^\dual \otimes \sheaf{\tilde Y}(2H)$. By projection formula and (1), we have 
    \begin{align}
        \varPsi(\pi^*\pi_*\sheaf{\tilde Y}(H) \otimes \sheaf{\tilde Y}(2H)) 
        &= \pi_*(\pi^*\FFF^\dual \otimes \cE \otimes \sheaf{\tilde Y}(2H+h))[2] \\
        &= \FFF^\dual \otimes \pi_*(\cE \otimes \sheaf{\tilde Y}(2H+h))[2] \\ 
        &= \pi_*\sheaf{\tY}(H) \otimes \CCC_2[2] = \pi_*\sheaf{\tY}(H+h) \otimes \CCC_0[2].
    \end{align}
                                        Meanwhile, $\varPsi(\sheaf{\tilde Y}(3H)) = \pi_*(\cE \otimes \sheaf{\tilde Y}(3H+h))[2]$ fits into the distinguished triangle:
    \[\begin{tikzcd}[row sep=tiny]
            {\CCC_{-1} \otimes q_*\sheaf{\tY}(h)[2]} & {\CCC_0 \otimes q_*\sheaf{\tilde Y}(H+h)[2]} & {\pi_*(\cE \otimes \sheaf{\tilde Y}(3H+h))[2]} & {} \\
            {\CCC_1[2]} & {\varPsi(\pi^*\pi_*\sheaf{\tilde Y}(H) \otimes \sheaf{\tilde Y}(2H))} & {\varPsi(\sheaf{\tilde Y}(3H))}
            \arrow[from=1-1, to=1-2]
            \arrow[from=1-2, to=1-3]
            \arrow[equal, from=1-1, to=2-1]
            \arrow[equal, from=1-2, to=2-2]
            \arrow["{+1}", from=1-3, to=1-4]
            \arrow[equal, from=1-3, to=2-3]
        \end{tikzcd}\]
    Comparing this with the sequence \eqref{ses:H}, we deduce that 
    \begin{equation}
        \varPsi\circ\lmut{\beta_*\Dcatb(\pro^3)}\sheaf{\tilde Y}(3H) \cong (\CCC_{1}[2])[1] = \CCC_1[3].
    \end{equation}

    \item By \cref{lem:lmut_B}, $\lmut{\beta_*\Dcatb(\pro^3)}\iota_*\sheaf{E}(H+h)$ fits into the distinguished triangle
    \[\begin{tikzcd}
            {\pi^*\pi_*\iota_*\sheaf{E}(-H+h) \otimes \sheaf{\tilde Y}(2H)} & \iota_*\sheaf{E}(H+h) & {\lmut{\beta_*\Dcatb(\pro^3)}\iota_*\sheaf{E}(H+h)} & {}
            \arrow[from=1-1, to=1-2]
            \arrow[from=1-2, to=1-3]
            \arrow["{+1}", from=1-3, to=1-4]
        \end{tikzcd}\]
    We claim that $\pi_*\iota_*\sheaf{E}(-H+h) \cong \sheaf{\pro^3}[-1]$. Consider the short exact sequence:
    \[\begin{tikzcd}
            0 & {\sheaf{\tilde Y}(-2H+2h)} & {\sheaf{\tilde Y}(-H+h)} & {\iota_*\sheaf{E}(-H+h)} & 0.
            \arrow[from=1-1, to=1-2]
            \arrow[from=1-2, to=1-3]
            \arrow[from=1-3, to=1-4]
            \arrow[from=1-4, to=1-5]
        \end{tikzcd}\]
    Pushing forward along $\pi\colon \tilde Y \to \pro^3$, since $\pi_*\sheaf{\tilde Y}(-H) = 0$, we have $\pi_*\iota_*\sheaf{E}(-H+h) \cong \pi_*\sheaf{\tilde Y}(-2H+2h)[1]$. Next, $\pi_*\sheaf{\tilde Y}(-2H+2h)$ fits into the distinguished triangle:
    \[\begin{tikzcd}
            {q_*\sheaf{\pro(\FFF)}(-4H'+h')} & {q_*\sheaf{\pro(\FFF)}(-2H'+2h')} & {\pi_*\sheaf{\tilde Y}(-2H+2h)} & {}
            \arrow[from=1-1, to=1-2]
            \arrow[from=1-2, to=1-3]
            \arrow["{+1}", from=1-3, to=1-4]
        \end{tikzcd}\]
    Since $q_*\sheaf{\pro(\FFF)}(-2H') = 0$, we have $\pi_*\iota_*\sheaf{E}(-H+h) \cong \pi_*\sheaf{\tilde Y}(-2H+2h)[1] \cong q_*\sheaf{\pro(\FFF)}(-4H'+h')[2]$. The relative canonical bundle of the $\pro^3$-fibration $q\colon \pro(\FFF) \to \pro^3$ is given by $\omega_q = q^*\det\FFF^\dual \otimes \sheaf{\pro(\FFF)}(-4H')
    = \sheaf{\pro(\FFF)}(-4H' + h')$. Hence, by Grothendieck--Verdier duality,
    \[ \pi_*\iota_*\sheaf{E}(-H+h) \cong q_*\sheaf{\pro(\FFF)}(-4H'+h')[2] \cong q_*\omega_q[2] = \sheaf{\pro^3}[-1].\]
    Therefore $\pi^*\pi_*\iota_*\sheaf{E}(-H+h) \otimes \sheaf{\tilde Y}(2H) \cong \sheaf{\tilde Y}(2H)[-1]$. Now by applying $\varPsi$ we have the distinguished triangle
    \[\begin{tikzcd}
            {\varPsi(\sheaf{\tilde Y}(2H))[-1]} & \varPsi(\iota_*\sheaf{E}(H+h)) & {\varPsi\circ\lmut{\beta_*\Dcatb(\pro^3)}\iota_*\sheaf{E}(H+h)} & {}
            \arrow[from=1-1, to=1-2]
            \arrow[from=1-2, to=1-3]
            \arrow["{+1}", from=1-3, to=1-4]
        \end{tikzcd}\]
    For the first term, by (1) it is isomorphic to $\CCC_2[1]$; and for the second term, note that $\iota_*\sheaf{E}(H+h)$ fits into the short exact sequence
    \[\begin{tikzcd}
            0 & {\sheaf{\tilde Y}(2h)} & {\sheaf{\tilde Y}(H+h)} & {\iota_*\sheaf{E}(H+h)} & 0.
            \arrow[from=1-1, to=1-2]
            \arrow[from=1-2, to=1-3]
            \arrow[from=1-3, to=1-4]
            \arrow[from=1-4, to=1-5]
        \end{tikzcd}\]
    Applying $\varPsi$ to the sequence and using \cref{lem:psi_kill}, the first two terms vanish, and we deduce that $\varPsi(\iota_*\sheaf{E}(H+h)) = 0$. Finally we conclude that
    \[ \varPsi\circ\lmut{\beta_*\Dcatb(\pro^3)}\iota_*\sheaf{E}(H+h) \cong \varPsi(\sheaf{\tilde Y}(2H)) \cong \CCC_2[2]. \qedhere\]

    \end{enumerate}
\end{proof}

\begin{proof}[Proof of \cref{thm:psi}]
    By \cref{prop:SOD_big_mut} and \eqref{equ:psi_K}, we have 
    \begin{align*}
        \Dcatb(\pro^3,\CCC_0) \simeq \varPsi(\KKK) = \ord{\varPsi\sigma^*\Ku(Y),\, 
    \varPsi\circ\lmut{\beta_*\Dcatb(\pro^3)}\sheaf{\tilde Y}(3H), \
    \varPsi\circ\lmut{\beta_*\Dcatb(\pro^3)}\iota_*\sheaf{E}(H+h)}.
    \end{align*}
    By \cref{lem:LB1}, we have $\varPsi\circ\lmut{\beta_*\Dcatb(\pro^3)}\sheaf{\tilde Y}(3H) \cong \CCC_1[3]$ and $\varPsi\circ\lmut{\beta_*\Dcatb(\pro^3)}\iota_*\sheaf{E}(H+h) \cong \CCC_2[2]$,
    which yields the desired form:
    \[\Dcatb(\pro^3,\CCC_0) \simeq \ord{\varPsi\sigma^*\Ku(Y),\, \CCC_1,\, \CCC_2}. \qedhere \]
\end{proof}

We conjecture that a generalization of \cref{thm:psi} holds for higher-dimensional cubic hypersurfaces. 

\begin{conjecture}[][conj:quad_fib]
    Let $Y$ be a smooth cubic $n$-fold ($n \geq 3$) containing a $k$-plane $\varPi_0 \subset Y$. The quadric fibration $\Bl_{\varPi_0}Y \to \pro^{n-k}$ defines a sheaf of Clifford algebras whose even part is $\CCC_0$. Then the bounded derived category of $\CCC_0$-modules on $\pro^{n-k}$ admits the semi-orthogonal decomposition
    \[ \Dcatb(\pro^{n-k},\CCC_0) = \ord{\Ku(\pro^{n-k},\CCC_0),\, \CCC_1,\, \ldots,\, \CCC_{2n-3k-2}},\]
    such that there is an equivalence $\Ku(Y) \simeq \Ku(\pro^{n-k},\CCC_0)$.
\end{conjecture}
When $2n-3k-2=0$, the exceptional collection displayed in the conjecture is understood to be empty.
Note that the smoothness of $Y$ forces $k \leq n/2$ (\cite[Exercise 1.1.5]{huy23}). For $n \geq 3$ we have $2n-3k-2 \geq 0$. The known cases for \cref{conj:quad_fib} are $(n,k) = (3,1)$ in \cite[Proposition 2.19]{BMMS}, $(4,1)$ in \cite[Proposition 7.7]{BLMS}, $(4,2)$ in \cite[Theorem 4.3]{kuz4fold}, $(5,2)$ in \cref{thm:psi} of this section, and $(7,3)$ in \cite{LPZCubic7000}. A systematic approach to general $(n,k)$ would be interesting.

\section{Numerical Grothendieck groups}\label{sec:Knum}
In this section we study the numerical lattice of the Kuznetsov component $\Ku(Y)$ of a smooth cubic 5-fold $Y$. We will describe it both as a sublattice of $\Knum(Y)$ and as a sublattice of $\Knum(\pro^3,\CCC_0)$ via the embedding $\varPsi\sigma^*\colon \Ku(Y) \hookrightarrow \Dcatb(\pro^3,\CCC_0)$.

\subsection[The numerical K-group for Ku(Y)]{The numerical K-group for $\Ku(Y)$}\label{subsec:Knum_KuY}

The semi-orthogonal decomposition \eqref{SOD:cubic_5} for $\Dcatb(Y)$ can be rotated:
    \begin{align}
        \Dcatb(Y) &= \ord{\Ku(Y),\ \sheaf{Y},\ \sheaf{Y}(1),\ \sheaf{Y}(2),\ \sheaf{Y}(3)} \\ 
        &= \ord{\sheaf{Y}(-1),\ \Ku(Y),\ \sheaf{Y},\ \sheaf{Y}(1),\ \sheaf{Y}(2)} \\ 
        &= \ord{\sheaf{Y}(-2),\ \sheaf{Y}(-1),\ \Ku(Y),\ \sheaf{Y},\ \sheaf{Y}(1)}. \label{SOD:DbY_RRLL}
    \end{align}

\begin{definition}
    We define the projection functor $\pr_Y\colon \Dcatb(Y) \to \Ku(Y)$ as
    \begin{equation}
        \pr_Y \coloneqq \rmut{\sheaf{Y}(-1)}\circ\rmut{\sheaf{Y}(-2)}\circ\lmut{\sheaf{Y}}\circ \lmut{\sheaf{Y}(1)},
    \end{equation}
    with respect to the SOD \eqref{SOD:DbY_RRLL}.
\end{definition}

\begin{definition}[][def:Ku_objs]
    Let $\varPi \subset Y$ be a plane. The ideal sheaf of $\varPi$ in $Y$ is denoted by $\ideal_{\varPi}$. We define three objects in $\Ku(Y)$ associated to $\varPi$ by projecting the twists of $\ideal_{\varPi}$ into $\Ku(Y)$:
    \begin{equation}
        \PPP_{\varPi} \coloneqq 
        \operatorname{pr}_Y(\ideal_{\varPi});
                \qquad
        \FFF_{\varPi} \coloneqq 
        \operatorname{pr}_Y(\ideal_{\varPi}(1))[-1];
                \qquad
        \KKK_{\varPi} \coloneqq 
        \operatorname{pr}_Y(\ideal_{\varPi}(-1)).
            \end{equation}        We will later show that the classes of $\FFF_{\varPi}$ and $\PPP_{\varPi}$ form the preferred basis of $\Knum(\Ku(Y))$.
\end{definition}

\begin{proposition}[][prop:Ku_objs]
    The objects $\PPP_{\varPi}$, $\FFF_{\varPi}$, $\KKK_{\varPi}\in \Ku(Y)$ admit the following descriptions:
    \begin{enumerate}[nosep]
        \item $\PPP_{\varPi} = \rmut{\sheaf{Y}(-1)}(\ideal_{\varPi})$ is a complex that fits into the distinguished triangle
              \begin{equation}
                  \begin{tikzcd}[ampersand replacement=\&,cramped]
                      {\sheaf{Y}(-1)[1]} \& {\PPP_{\varPi}} \& {\ideal_{\varPi}} \& {}
                      \arrow[from=1-1, to=1-2]
                      \arrow[from=1-2, to=1-3]
                      \arrow["{+1}", from=1-3, to=1-4]
                  \end{tikzcd} \label{seq:P_Pi_def}
              \end{equation}
        \item $\FFF_{\varPi} = \lmut{\sheaf{Y}}(\ideal_{\varPi}(1))[-1]$ is a reflexive sheaf of rank $3$ that sits inside a short exact sequence of sheaves
              \begin{equation}
                  \begin{tikzcd}[ampersand replacement=\&,cramped]
                      0 \& {\FFF_{\varPi}} \& {\sheaf{Y}^{\oplus 4}} \& {\ideal_{\varPi}(1)} \& 0.
                      \arrow[from=1-1, to=1-2]
                      \arrow[from=1-2, to=1-3]
                      \arrow[from=1-3, to=1-4]
                      \arrow[from=1-4, to=1-5]
                  \end{tikzcd} \label{seq:F_Pi_def}
              \end{equation}
        \item $\KKK_{\varPi} = \rmut{\sheaf{Y}(-1)}\rmut{\sheaf{Y}(-2)}(\ideal_{\varPi}(-1))$
                has cohomology sheaves given by
        \begin{equation}
            \mathcal H^i(\KKK_{\varPi}) \cong \begin{cases}
                      \coker(\sheaf{Y}(-2) \hookrightarrow \sheaf{Y}(-1)^{\oplus 4}), & i = 0;            \\
                      \sheaf{\varPi}(-1),                                             & i = 1;            \\
                      0,                                                              & \text{otherwise}.
                  \end{cases} \label{seq:K_Pi_def}
        \end{equation}
                                                                                                                                                                                                        \end{enumerate}
\end{proposition}
\begin{proof}
    First we claim that $\ideal_{\varPi} \in \ord{\sheaf{Y},\sheaf{Y}(1),\sheaf{Y}(2)}^\perp$. Indeed, consider the short exact sequence:
    \begin{equation}
        \begin{tikzcd}[row sep = 0]
            0 & {\ideal_{\varPi}(-a)} & \sheaf{Y}(-a) & \sheaf{\varPi}(-a) & 0.
            \arrow[from=1-1, to=1-2]
            \arrow[from=1-2, to=1-3]
            \arrow[from=1-3, to=1-4]
            \arrow[from=1-4, to=1-5]
        \end{tikzcd} \label{seq:plane_twist}
    \end{equation}
    For $1 \leq a \leq 2$, $\sheaf{Y}(-a)$ and $\sheaf{\varPi}(-a)$ have no cohomology; and for $a = 0$, $\Hlg^\bullet(Y,\sheaf{Y}) \to \Hlg^\bullet(Y,\sheaf{\varPi})$ is an isomorphism. Therefore the long exact sequence shows that $\Ext^\bullet(\sheaf{Y}(a),\ideal_{\varPi}) = \Hlg^\bullet(Y,\ideal_{\varPi}(-a)) = 0$ for $a = 0,1,2$, and hence the claim follows. By Serre duality, we also have $\ideal_{\varPi} \in {}^\perp\ord{\sheaf{Y}(-3),\sheaf{Y}(-2)}$. It follows that 
    \begin{align*}
        \operatorname{pr}_Y(\ideal_{\varPi}) = \rmut{\sheaf{Y}(-1)}(\ideal_{\varPi}); \quad 
        \operatorname{pr}_Y(\ideal_{\varPi}(1)) = \lmut{\sheaf{Y}}(\ideal_{\varPi}(1)); \quad
        \operatorname{pr}_Y(\ideal_{\varPi}(-1)) = \rmut{\sheaf{Y}(-1)}\rmut{\sheaf{Y}(-2)}(\ideal_{\varPi}(-1)).
    \end{align*}
                                            \begin{enumerate}
        \item For the object $\PPP_{\varPi}$, by definition it fits into the distinguished triangle:
              \begin{equation}
                  \begin{tikzcd}[ampersand replacement=\&,cramped]
                      {\PPP_{\varPi}} \& {\ideal_{\varPi}} \& {\displaystyle\bigoplus_{k=0}^5\Ext^k(\ideal_{\varPi},\sheaf{Y}(-1))^\dual \otimes \sheaf{Y}(-1)[k] } \& {}
                      \arrow[from=1-1, to=1-2]
                      \arrow[from=1-2, to=1-3]
                      \arrow["{+1}", from=1-3, to=1-4]
                  \end{tikzcd} \label{ses:P_RMut}
              \end{equation}
              where 
              \[ \Ext^\bullet(\ideal_{\varPi},\sheaf{Y}(-1)) \cong \Hlg^{\bullet}(\ideal_{\varPi}(-3)[5])^\dual = \Hlg^{\bullet}(\sheaf{\varPi}(-3)[4])^\dual = \C[-2],\]
              by Serre duality on $Y$ and $\varPi$, and the long exact sequence of \eqref{seq:plane_twist} with $a = 3$. Hence \eqref{ses:P_RMut} becomes
              \begin{equation}
                  \begin{tikzcd}[ampersand replacement=\&,cramped]
                      {\PPP_{\varPi}} \& {\ideal_{\varPi}} \& {\sheaf{Y}(-1)[2]} \& {}
                      \arrow[from=1-1, to=1-2]
                      \arrow[from=1-2, to=1-3]
                      \arrow["{+1}", from=1-3, to=1-4]
                  \end{tikzcd}
              \end{equation}
              After rotating this triangle we obtain \eqref{seq:P_Pi_def}.

        \item For the description of $\FFF_{\varPi}$, by definition it fits into the distinguished triangle:
              \begin{equation}
                  \begin{tikzcd}[ampersand replacement=\&,cramped]
                      {\FFF_{\varPi}} \& {\displaystyle\bigoplus_{k = 0}^5 \Hlg^k(Y,\ideal_{\varPi}(1)) \otimes \sheaf{Y}[-k]} \& {\ideal_{\varPi}(1)} \& {}
                      \arrow[from=1-1, to=1-2]
                      \arrow[from=1-2, to=1-3]
                      \arrow["{+1}", from=1-3, to=1-4]
                  \end{tikzcd} 
              \end{equation}
              The long exact sequence of \eqref{seq:plane_twist} with $a = -1$ shows that $\Hlg^i(Y,\ideal_{\varPi}(1)) = 0$ for $i \geq 2$, and
              \begin{equation}
                  \begin{tikzcd}[column sep = small]
                      0 & {\Hlg^0(Y,\ideal_{\varPi}(1))} & {\Hlg^0(Y,\sheaf{Y}(1))} & {\Hlg^0(\varPi,\sheaf{\varPi}(1))} & {\Hlg^1(Y,\ideal_{\varPi}(1))} & 0.
                      \arrow[from=1-1, to=1-2]
                      \arrow[from=1-2, to=1-3]
                      \arrow[from=1-3, to=1-4]
                      \arrow[from=1-4, to=1-5]
                      \arrow[from=1-5, to=1-6]
                  \end{tikzcd}
              \end{equation}
              The map $\Hlg^0(Y,\sheaf{Y}(1)) \to \Hlg^0(\varPi,\sheaf{\varPi}(1))$ is the restriction map induced by the linear embedding $\varPi \hookrightarrow \pro^6$, and is surjective. Hence $\hlg^1(Y,\ideal_{\varPi}(1)) = 0$ and $\hlg^0(Y,\ideal_{\varPi}(1)) = 4$. Moreover, $\Hlg^0(Y,\ideal_{\varPi}(1))$ consists of the linear forms on $\pro^6$ vanishing on $\varPi$, and these forms generate $\ideal_{\varPi}(1)$. Thus the evaluation map $\Hlg^0(Y,\ideal_{\varPi}(1)) \otimes \sheaf{Y} \cong \sheaf{Y}^{\oplus 4} \xrightarrow{\operatorname{ev}} \ideal_{\varPi}(1)$ is surjective.
                            Therefore $\FFF_{\varPi}$ is the kernel of $\sheaf{Y}^{\oplus 4} \xrightarrow{\operatorname{ev}} \ideal_{\varPi}(1)$, and is a torsion-free sheaf of rank $3$. It is locally free away from the codimension-$3$ locus $\varPi$. At every point of $\varPi$, the depth lemma applied to \eqref{seq:F_Pi_def} gives $\operatorname{depth}(\FFF_{\varPi})\geq 2$, because $\ideal_{\varPi}(1)$ is torsion-free. Since $Y$ is smooth, Serre's criterion for reflexive sheaves shows that $\FFF_{\varPi}$ is reflexive.

        \item For $\KKK_{\varPi}$, since $\KKK_{\varPi}(1) \cong \rmut{\sheaf{Y}}\PPP_{\varPi}$, we have the distinguished triangle
            \begin{equation}
                \begin{tikzcd}[ampersand replacement=\&,cramped]
                    {\KKK_{\varPi}(1)} \& \PPP_{\varPi} \& {\displaystyle\bigoplus_{k} \Ext^k(\PPP_{\varPi},\sheaf{Y})^\dual \otimes \sheaf{Y}[k]} \& {}
                    \arrow[from=1-1, to=1-2]
                    \arrow[from=1-2, to=1-3]
                    \arrow["{+1}", from=1-3, to=1-4]
                \end{tikzcd} \label{seq:K_right_mutate}
            \end{equation}
            To compute $\Ext^\bullet(\PPP_{\varPi},\sheaf{Y})$ using
            \eqref{seq:P_Pi_def}, we need the two groups
            $\Ext^\bullet(\sheaf{Y}(-1)[1],\sheaf{Y})$ and
            $\Ext^\bullet(\ideal_{\varPi},\sheaf{Y})$. The former is
            \[
                \Ext^\bullet(\sheaf{Y}(-1)[1],\sheaf{Y}) \cong \C^7[-1].
            \]
            For the latter, Serre duality gives
            \[
                \Ext^\bullet(\ideal_{\varPi},\sheaf{Y})
                \cong \Hlg^{\bullet}(Y,\ideal_{\varPi}(-4)[5])^\dual.
            \]
            Take the long exact sequence of \eqref{seq:plane_twist} with $a=4$.
            Since
            \[
                \Hlg^{\bullet}(Y,\sheaf{Y}(-4)) \cong \C[-5],
                \qquad
                \Hlg^\bullet(\varPi,\sheaf{\varPi}(-4)) \cong \C^3[-2],
            \]
            we obtain
            $\Hlg^\bullet(Y,\ideal_{\varPi}(-4))=\C^3[-3]\oplus\C[-5]$.
            Therefore
            \[ \Ext^\bullet(\ideal_{\varPi},\sheaf{Y}) = \C[0] \oplus \C^3[-2]. \]
            Applying $\Hom(-,\sheaf{Y})$ to \eqref{seq:P_Pi_def}, we have $\Hom(\PPP_{\varPi},\sheaf{Y}) \cong \Hom(\ideal_{\varPi},\sheaf{Y}) \cong \C$ and the long exact sequence
            \begin{equation}
                  \begin{tikzcd}[column sep = small]
                      0 & {\Ext^1(\PPP_{\varPi},\sheaf{Y})} & {\Hom(\sheaf{Y}(-1),\sheaf{Y})} & {\Ext^2(\ideal_{\varPi},\sheaf{Y})} & {\Ext^2(\PPP_{\varPi},\sheaf{Y})} & 0.
                      \arrow[from=1-1, to=1-2]
                      \arrow[from=1-2, to=1-3]
                      \arrow[from=1-3, to=1-4]
                      \arrow[from=1-4, to=1-5]
                      \arrow[from=1-5, to=1-6]
                  \end{tikzcd}
            \end{equation}
            The map
            \[\Hlg^0(Y,\sheaf{Y}(1)) \cong \Hom(\sheaf{Y}(-1),\sheaf{Y}) \to \Ext^2(\ideal_{\varPi},\sheaf{Y}) \cong \Ext^3(\sheaf{\varPi},\sheaf{Y}) \cong \Hlg^0(\varPi,\sheaf{\varPi}(1))\]
            is just the restriction map, which is surjective. We deduce that 
            \[ \Ext^\bullet(\PPP_{\varPi},\sheaf{Y}) = \C[0] \oplus \C^4[-1]. \]
            Then the sequence \eqref{seq:K_right_mutate} becomes 
            \begin{equation}
                \begin{tikzcd}[ampersand replacement=\&,cramped]
                    {\KKK_{\varPi}(1)} \& \PPP_{\varPi} \& {\sheaf{Y} \oplus \sheaf{Y}^{\oplus 4}[1]} \& {}
                    \arrow[from=1-1, to=1-2]
                    \arrow[from=1-2, to=1-3]
                    \arrow["{+1}", from=1-3, to=1-4]
                \end{tikzcd} \label{seq:K_right_mutate_2}
            \end{equation}
            Taking the cohomology, we have the long exact sequence:
            \[\begin{tikzcd}[ampersand replacement=\&,cramped,column sep=scriptsize,row sep=0]
            	0 \& {\mathcal H^{-1}(\KKK_{\varPi}(1))} \& {\mathcal H^{-1}(\PPP_{\varPi})} \& {\sheaf{Y}^{\oplus 4}} \& {} \& \\
            	\& {\mathcal H^{0}(\KKK_{\varPi}(1))} \& {\mathcal H^{0}(\PPP_{\varPi})} \& {\sheaf{Y}} \& {\mathcal H^{1}(\KKK_{\varPi}(1))} \& {0.}
            	\arrow[from=1-1, to=1-2]
            	\arrow[from=1-2, to=1-3]
            	\arrow[from=1-3, to=1-4]
            	\arrow[from=1-4, to=1-5]
            	\arrow[from=2-2, to=2-3]
            	\arrow[from=2-3, to=2-4]
            	\arrow[from=2-4, to=2-5]
            	\arrow[from=2-5, to=2-6]
            \end{tikzcd}\]
            The maps $\mathcal H^{-1}(\PPP_{\varPi}) \cong \sheaf{Y}(-1) \to \sheaf{Y}^{\oplus 4}$ and $\mathcal H^{0}(\PPP_{\varPi}) \cong \ideal_{\varPi} \to \sheaf{Y}$ are injective. Hence $\mathcal H^{-1}(\KKK_{\varPi}(1)) = 0$, $\mathcal H^{0}(\KKK_{\varPi}(1)) = \coker(\sheaf{Y}(-1) \to \sheaf{Y}^{\oplus 4})$, and $\mathcal H^{1}(\KKK_{\varPi}(1)) \cong \sheaf{\varPi}$. Now \eqref{seq:K_Pi_def} follows from twisting by $\sheaf{Y}(-1)$. \qedhere

                                                                                                                                                                                                                \end{enumerate}
                    \end{proof}

The Kuznetsov component $\Ku(Y)$ admits a Serre functor $\mathsf{S}_{\Ku(Y)}$ and a \textbf{rotation functor}\footnote{Also called a \textbf{degree shift functor}, since under the equivalence of $\Ku(Y) \simeq \operatorname{HMF}^{\mathsf{gr}}(W)$ with some category of graded matrix factorisations, the functor $\mathsf{O}_{\Ku(Y)}$ is realized as the degree shift $(1)$; see \cite[Corollary 2.14]{orlov09sing} or \cite[Theorem 7.1.34]{huy23}.} 
$\mathsf{O}_{\Ku(Y)} \coloneqq \lmut{\sheaf{Y}}(- \otimes \sheaf{Y}(1))$, both of which are automorphisms of $\Ku(Y)$. As introduced and studied in \cite{KuzV14} (see also \cite[\S 7.1]{huy23}), these functors satisfy the following relations.
\begin{lemma}[{\cite[Corollary 4.3]{KuzV14}, \cite[Proposition 7.1.26]{huy23}}][lem:O_S_KuY]
    The Kuznetsov component $\Ku(Y)$ is a fractional Calabi--Yau category of dimension $7/3$. That is $\mathsf S_{\Ku(Y)}^3 \simeq [7]$. Moreover, the rotation functor satisfies $\mathsf{O}_{\Ku(Y)} \simeq \mathsf S_{\Ku(Y)}^{-1}[3]$ and $\mathsf{O}_{\Ku(Y)}^3 \simeq [2]$.
\end{lemma}
\begin{remark}[][rmk:S&O:PFK]
    Note that the objects $\PPP_{\varPi}$, $\FFF_{\varPi}$, and $\KKK_{\varPi}$ are directly related by the rotation functor:
    \begin{align*}
        \mathsf{O}_{\Ku(Y)}\KKK_{\varPi} & = \PPP_{\varPi}; \qquad
        \mathsf{O}_{\Ku(Y)}\PPP_{\varPi} = \lmut{\sheaf{Y}}\PPP_{\varPi}(1) \cong \lmut{\sheaf{Y}}\ideal_{\varPi}(1) = \FFF_{\varPi}[1];                   \\
        \mathsf{O}_{\Ku(Y)}\FFF_{\varPi} & \cong \mathsf{O}_{\Ku(Y)}^2\PPP_{\varPi}[-1] \cong \mathsf{O}_{\Ku(Y)}^{-1}\PPP_{\varPi}[1] = \KKK_{\varPi}[1].
    \end{align*}
    In particular we have $\mathsf S_{\Ku(Y)}\FFF_{\varPi} \cong \PPP_{\varPi}[2]$ and $\mathsf S_{\Ku(Y)}\KKK_{\varPi} \cong \FFF_{\varPi}[2]$. These relations are visually displayed in \cref{fig:hexa}.
\end{remark}

\begin{proposition}[][prop:Knum_Ku]
    Let $Y$ be a smooth cubic 5-fold and $\varPi \subset Y$ a $2$-plane. The numerical Grothendieck group $\Knum(\Ku(Y))$ has an integral basis $\{\kappa_1,\kappa_2\}$, where 
    \begin{align*}
        \kappa_1 &\coloneqq [\FFF_{\varPi}]; & \Chern{}(\kappa_1) &= 3 - H - \dfrac{1}{2}H^2 + \dfrac{1}{6}H^3 + \dfrac{1}{8}H^4 - \dfrac{13}{360}H^5; \\ 
        \kappa_2 &\coloneqq [\PPP_{\varPi}]; & \Chern{}(\kappa_2) &= \phantom{3-{}}H - \dfrac{1}{2}H^2 - \dfrac{1}{6}H^3 + \dfrac{1}{8}H^4 + \dfrac{13}{360}H^5.
    \end{align*}
    The Euler pairing with respect to this basis is given by the matrix
    \[ \chi_{\Ku(Y)} = \begin{pmatrix}
            -1 & -1 \\ 0 & -1
        \end{pmatrix}. \]
    The Serre functor and the rotation functor act on the basis by
    \begin{equation}
        \begin{tikzcd}[ampersand replacement=\&,cramped,row sep=0, column sep = small]
            {\ab[\mathsf{S}_{\Ku(Y)}]_*:} \& \kappa_1 \& {\kappa_2,} \& \kappa_2 \& {\kappa_2-\kappa_1;} \\
            {[\mathsf{O}_{\Ku(Y)}]:} \& \kappa_1 \& {\kappa_2-\kappa_1,} \& \kappa_2 \& {-\kappa_1.}
            \arrow[maps to, from=1-2, to=1-3]
            \arrow[maps to, from=1-4, to=1-5]
            \arrow[maps to, from=2-2, to=2-3]
            \arrow[maps to, from=2-4, to=2-5]
        \end{tikzcd}
    \end{equation}
\end{proposition}
\vspace*{\parskip}
\begin{proof}
                            by Lefschetz hyperplane theorem and Poincar\'e duality, we have $\dim \operatorname{CH}_{\operatorname{num}}^{k}(Y)_{\Q} = 1$ for all $0 \leq k \leq 5$. Since $\Knum(Y) \otimes \Q \cong \operatorname{CH}_{\operatorname{num}}^{\bullet}(Y)_{\Q}$, the numerical K-group for $Y$ is of rank $6$. Then using the semi-orthogonal decomposition \eqref{SOD:cubic_5}, we have 
    \begin{equation}
        \Knum(Y) \cong \Knum(\Ku(Y)) \oplus \Z[\sheaf{Y}] \oplus \Z[\sheaf{Y}(1)] \oplus \Z[\sheaf{Y}(2)] \oplus \Z[\sheaf{Y}(3)].
    \end{equation}
    Hence $\Knum(\Ku(Y))$ is a lattice of rank $2$. We have shown that $\kappa_1 = [\FFF_{\varPi}] \in \Knum(\Ku(Y))$ and  $\kappa_2 = [\PPP_{\varPi}] \in \Knum(\Ku(Y))$. Their Chern characters are given by
    \begin{align*}
        \Chern{}(\FFF_{\varPi})
         & = 4 - \left( \Chern{}(\sheaf{Y}) - \Chern{}(\sheaf{\varPi}) \right)\Chern{}(\sheaf{Y}(1)) \\
         & = 3 - H - \dfrac{1}{2}H^2 + \dfrac{1}{6}H^3 + \dfrac{1}{8}H^4 - \dfrac{13}{360}H^5;       \\
        \Chern{}(\PPP_{\varPi})
         & = \Chern{}(\sheaf{H}) - \Chern{}(\sheaf{\varPi})                                          \\
         & = H - \dfrac{1}{2}H^2 - \dfrac{1}{6}H^3 + \dfrac{1}{8}H^4 + \dfrac{13}{360}H^5.
    \end{align*}
    In particular $\kappa_1$ and $\kappa_2$ are linearly independent, so they span the vector space $\Knum(\Ku(Y)) \otimes_{\Z} \Q$. To compute the Euler pairing with respect to this basis, we use the Hirzebruch--Riemann--Roch theorem. The Todd class of $Y$, computed using the Euler sequence and the normal sequence, is given by
    \begin{align*}
        \td(Y) & = \dfrac{\td(\mathscr{T}_{\pro^6}|_{Y})}{\td(\sheaf{Y}(3))} = \left( \dfrac{H}{1-\exp(-H)} \right)^7\cdot \dfrac{1-\exp(-3H)}{3H} \\
               & = 1 + 2H + \dfrac{25}{12}H^2 + \dfrac{3}{2}H^3 + \dfrac{73}{90}H^4 + \dfrac{1}{3}H^5.
    \end{align*}
    A Hirzebruch--Riemann--Roch computation gives
    \begin{equation}
        \chi(\kappa_1,\kappa_1) = \chi(\kappa_2,\kappa_2) = \chi(\kappa_1,\kappa_2) = -1; \qquad \chi(\kappa_2,\kappa_1) = 0.
    \end{equation}
    Hence $\chi_{\Ku(Y)} = \begin{pmatrix}
            -1 & -1 \\ 0 & -1
        \end{pmatrix}$ with respect to $\{\kappa_1,\kappa_2\}$. Note that this pairing matrix is unimodular, which means that $\{\kappa_1,\kappa_2\}$ is also an integral basis of the lattice $\Knum(\Ku(Y))$. 
                                                By \cref{rmk:S&O:PFK}, we have that $\mathsf{S}_{\Ku(Y)}\kappa_1 = \kappa_2$, $\mathsf{S}_{\Ku(Y)}\kappa_2 = \kappa_2 - \kappa_1$, and $\mathsf{O}_{\Ku(Y)}\kappa_1 = \kappa_2-\kappa_1$, $\mathsf{O}_{\Ku(Y)}\kappa_2 = - \kappa_1$.
\end{proof}

\begin{remark}
    It is instructive to visualize $\Knum(\Ku(Y))$ as the hexagonal lattice in the plane, as in \cref{fig:hexa}. The Serre functor acts on this lattice by rotation through $\pi/3$, while the rotation functor acts by rotation through $2\pi/3$.
\end{remark}
\begin{figure}[H]
    \centering
    \begin{tikzpicture}[line cap=round,line join=round,>=Stealth,,x=2.2cm,y=2.2cm]
                        \draw [line width=1pt,dashed] (1,1.7320508075688776)-- (0,0);
            \draw [line width=1pt,dashed] (0,0)-- (3,0);
            \draw [line width=1pt,dashed] (0,0)-- (-1,1.732050807568878);
                                    \draw [fill=black] (0,0) circle (1.5pt);
            \draw [fill=black] (1,0) circle (1.5pt);
            \draw[color=black] (1.0034205452991069,-0.18) node {$\kappa_1 = [\FFF_{\varPi}]$};
            \draw [fill=black] (2,0) circle (1.5pt);
            \draw [fill=black] (3,0) circle (1.5pt);
            \draw[color=black] (3.2417476725036076,0.18590524258913643) node {$\kappa_1$-axis};
            \draw [fill=black] (-1,0) circle (1.5pt);
            \draw[color=black] (-0.9027827771016672,-0.18) node {$-\kappa_1 =  [\cF_{\varPi}[1]]$};
            \draw [fill=black] (0.5,0.8660254037844386) circle (1.5pt);
            \draw[color=black] (1,1.0051597231911384) node {$\kappa_2 = [\PPP_{\varPi}]$};
            \draw [fill=black] (-0.5,0.8660254037844388) circle (1.5pt);
            \draw [fill=black] (3.5,0.8660254037844388) circle (1.5pt);
            \draw[color=black] (-0.24165810443849864,1.051783961924586) node {$\kappa_2-\kappa_1$};
            \draw [fill=black] (-0.5,-0.8660254037844384) circle (1.5pt);
            \draw[color=black] (-0.4683005265718972,-0.679973476746313) node {$-\kappa_2$};
            \draw [fill=black] (3.5,-0.8660254037844384) circle (1.5pt);
            \draw [fill=black] (0.5,-0.866025403784439) circle (1.5pt);
            \draw[color=black] (0.537903034305488,-0.679973476746313) node {$\kappa_1-\kappa_2 = [\KKK_{\varPi}]$};
            \draw [fill=black] (1,1.7320508075688776) circle (1.5pt);
            \draw[color=black] (1.236905406965366,1.9176626812600355) node {$\kappa_2$-axis};
            \draw [fill=black] (0,1.7320508075688779) circle (1.5pt);
            \draw [fill=black] (3,1.7320508075688776) circle (1.5pt);
            \draw [fill=black] (2,1.7320508075688779) circle (1.5pt);
            \draw [fill=black] (1.5,0.8660254037844395) circle (1.5pt);
            \draw[color=black] (1.7631841610997429,1.051783961924586) node {$\kappa_1+\kappa_2$};
            \draw [fill=black] (1.5,-0.8660254037844396) circle (1.5pt);
            \draw [fill=black] (2.5,-0.8660254037844404) circle (1.5pt);
            \draw [fill=black] (2.5,0.8660254037844392) circle (1.5pt);
            \draw [fill=black] (-1,1.732050807568878) circle (1.5pt);
                        \def\r{0.15}
            \def\gap{0.1}
            \draw[shift={(0,0)}, ->, line width=0.7pt, smooth, samples=250, variable=\t, domain=0:420, color=cyan!50!black] plot ({(\r + \gap*\t/360)*cos(\t)}, {(\r + \gap*\t/360)*sin(\t)});
            \draw [shift={(0,0)},->,line width=0.7pt, color=cyan!50!black] (0:0.6) arc (0:120:0.6);
            \draw[color=cyan!50!black] (0.8,0.4) node {$\sfO_{\Ku(Y)}$};
            \draw[color=cyan!50!black] (0.1,-0.36) node {$\sfS_{\Ku(Y)}$};
                    \end{tikzpicture}
    \caption[Characters in $\Knum(\Ku(Y))$ in hexagonal coordinates.]{Characters in $\Knum(\Ku(Y))$ in hexagonal coordinates. The indicated arrows show
the actions of the Serre functor and the rotation functor.}
        \label{fig:hexa}
\end{figure}

\subsection[The numerical K-group for Ku(ℙ³,C₀)]{The numerical K-group for $\Ku(\pro^3,\CCC_0)$}

\begin{definition}
    Let $X$ be an $n$-dimensional smooth projective variety and $\cB$ a locally free sheaf of $\sheaf{X}$-algebras. Denote by $\Coh(X,\cB)$ the coherent sheaves of right $\cB$-modules and $\Dcatb(X,\cB)$ the corresponding bounded derived category. In this context, the pair $(X,\cB)$ is often called a \emph{non-commutative} variety. There is a forgetful functor $\operatorname{Forg}\colon \Dcatb(X,\cB) \to \Dcatb(X)$, which forgets the $\cB$-module structure. It admits both left and right adjoints (\cite[Lemma 10.14, Corollary 10.35]{kuz06hyp}):
    \begin{equation}
        (- \otimes_{\sheaf{X}} \cB) \dashv \operatorname{Forg} \dashv (- \otimes_{\sheaf{X}} \cB^\dual). \label{eq:adjunc}
    \end{equation}
\end{definition}

\begin{lemma}[][lem:twisted_Serre]
    Let $X$ be an $n$-dimensional smooth projective variety and $\cB$ a locally free sheaf of $\sheaf{X}$-algebras. Assume that $\Dcatb(X,\cB) = \Dperf(X,\cB)$. Then the Serre functor of $\Dcatb(X,\cB)$ is given by
    \begin{equation}
        \mathsf{S}(E) = E \otimes_{\cB} \ab(\cB^\dual \otimes_{\sheaf{X}} \omega_X )   [n],
    \end{equation}
    where $\cB^\dual$ is the dual of $\cB$ as an $\sheaf{X}$-module.
\end{lemma}
\begin{proof}
    We need to check the functorial isomorphism for $E,F \in \Dcatb(X,\cB)$:
    \begin{equation}\label{equ:Serre:(X,B)}
        \Hom(E,\, F) \cong \Hom\ab(F,\, E \otimes_{\cB} \ab(\cB^\dual\otimes_{\sheaf{X}} \omega_X)[n])^\dual.
    \end{equation}
    Since $\Dcatb(X,\cB) = \Dperf(X,\cB)$, the category is generated by induced perfect modules $G \otimes_{\sheaf{X}} \cB$ with $G \in \Dperf(X)$. It therefore suffices to check \eqref{equ:Serre:(X,B)} on such objects.
                                                    Indeed, using the adjunctions \eqref{eq:adjunc}, we have, for any $F \in \Dcatb(X,\cB)$ and $G \in \Coh(X)$,
    \begin{align*}
        \Hom_{\Dcatb(X,\cB)}(G \otimes_{\sheaf{X}}\cB,\, F)
         & \cong \Hom_{\Dcatb(X)}(G,\, \operatorname{Forg}(F)) \\
         & \cong \Hom_{\Dcatb(X)}(\operatorname{Forg}(F),\, \mathrm{S}_X(G))^\dual \\
         & \cong \Hom_{\Dcatb(X)}(\operatorname{Forg}(F),\, G \otimes_{\sheaf{X}} \omega_X[n])^\dual  \\
         & \cong \Hom_{\Dcatb(X,\cB)}(F,\, G \otimes_{\sheaf{X}} \cB^\dual \otimes_{\sheaf{X}} \omega_X [n])^\dual \\
         & \cong \Hom_{\Dcatb(X,\cB)}\ab(F,\,  \ab(G \otimes_{\sheaf{X}} \cB) \otimes_{\cB} \ab(\cB^\dual\otimes_{\sheaf{X}} \omega_X) [n])^\dual.  \qedhere
    \end{align*}
        \end{proof}

From \cref{thm:psi}, there is an SOD $\Dcatb(\pro^3,\CCC_0) = \ord{\Ku(\pro^3,\CCC_0), \CCC_1,\CCC_2}$ such that $\varPsi\circ\sigma^*\colon \Ku(Y) \to \Ku(\pro^3,\CCC_0)$ is an equivalence of categories. \cref{lem:twisted_Serre} gives a description of the Serre functors:
\begin{corollary}[][cor:Serre]
    \begin{enumerate}[label = (\roman*), nosep, before=\vspace*{-\parskip}]
        \item The Serre functor of $\Dcatb(\pro^3,\CCC_0)$ is given by $\mathsf{S}_{\Dcatb(\pro^3,\CCC_0)} = (- \otimes_{\CCC_0} \CCC_{-2})[3]$.
        \item The Serre functor of $\Ku(\pro^3,\CCC_0)$ satisfies $\mathsf{S}_{\Ku(\pro^3,\CCC_0)}^{-1} = \mathsf{O}_{\Ku(\pro^3,\CCC_0)}^2[-3]$, where
              \begin{equation}
                  \mathsf{O}_{\Ku(\pro^3,\CCC_0)} \coloneqq \lmut{\CCC_1} \circ (- \otimes_{\CCC_0} \CCC_1) = (- \otimes_{\CCC_0} \CCC_1) \circ \lmut{\CCC_0},
              \end{equation}
              also called a \textbf{rotation functor}, is an auto-equivalence of $\Ku(\pro^3,\CCC_0)$.
    \end{enumerate}
\end{corollary}
\begin{proof}
    \begin{enumerate}[label = (\roman*)]
        \item Since $\Dcatb(\pro^3,\CCC_0)$ embeds into $\Dcatb(\tY)$ as an admissible subcategory, it is smooth and proper, and in particular, $\Dcatb(\pro^3,\CCC_0) = \Dperf(\pro^3,\CCC_0)$. Additionally, we have
              \begin{equation}
                  \CCC_0^\dual \otimes \omega_{\pro^3} \cong \left( \sheaf{} \oplus \sheaf{}(1)^{\oplus 3} \oplus \sheaf{}(2)^{\oplus 3} \oplus \sheaf{}(3) \right) \otimes \sheaf{}(-4) \cong \CCC_{0} \otimes \sheaf{}(-1) \cong \CCC_{-2},
              \end{equation}
              and the result follows from \cref{lem:twisted_Serre}.

        \item Let $\varXi_*\colon \Ku(\pro^3,\CCC_0) \to \Dcatb(\pro^3,\CCC_0)$ be the embedding functor. By \cref{thm:psi}, its left adjoint is $\varXi^* = \lmut{\CCC_1} \circ \lmut{\CCC_2}$. The Serre functors of $\Ku(\pro^3,\CCC_0)$ and $\Dcatb(\pro^3,\CCC_0)$ are related by $\mathsf S_{\Ku(\pro^3,\CCC_0)}^{-1} = \varXi^* \circ \mathsf S_{\CCC_0}^{-1} \circ \varXi_*$. Using $\mathsf S_{\CCC_0}^{-1} = - \otimes_{\CCC_0} \CCC_2[-3]$ and the semi-orthogonal decomposition \eqref{eq:twist_P3_SOD}, we find that
              \begin{align*}
                  \mathsf S_{\Ku(\pro^3,\CCC_0)}^{-1} & = \varXi^* \circ \mathsf S_{\CCC_0}^{-1} \circ \varXi_*
                  = \lmut{\CCC_1} \circ \lmut{\CCC_2} \circ (- \otimes_{\CCC_0} \CCC_2) \circ \varXi_*[-3]                                                             \\
                                          & = \lmut{\CCC_1} \circ (- \otimes_{\CCC_0} \CCC_1) \circ \lmut{\CCC_1} \circ (- \otimes_{\CCC_0} \CCC_1) \circ \varXi_*[-3] \\
                                          & = \mathsf{O}_{\Ku(\pro^3,\CCC_0)}^{2}[-3]. \qedhere
              \end{align*}
    \end{enumerate}
\end{proof}

On the twisted variety $(X,\cB)$, the Euler pairing is analogously defined as
\begin{equation}
    \chi_{\cB}(E,F) \coloneqq \sum_{k \in \Z} (-1)^k \dim \Hom_{\Dcatb(X,\cB)}(E,F[k]).
\end{equation}
The next proposition computes the lattice and Euler pairing on the numerical Grothendieck group of $\Ku(\pro^3,\CCC_0)$. For $E \in \Dcatb(\pro^3,\CCC_0)$, the Chern character $\Chern{}(\operatorname{Forg}(E))$ will be simply denoted by $\Chern{}(E)$.
\begin{proposition}[][prop:Knum_Ku_C0]
    The numerical Grothendieck group $\Knum(\Ku(\pro^3,\CCC_0))$ has an integral basis $\{\bar \kappa_1, \bar \kappa_2\}$, where
    \begin{align}
        \bar \kappa_1 & = [\CCC_{-1}] - 4[\CCC_0] + 4[\CCC_1] - [\CCC_2], & \Chern{}(\bar \kappa_1) & = \phantom{-8+1}4h - \phantom{1}5h^2 + \mathcal{O}(h^3). \\
        \bar \kappa_2 & = [\CCC_{-1}] - 3[\CCC_0] + \phantom{4}[\CCC_1], & \Chern{}(\bar \kappa_2) & = -8 + 12h - 10h^2 + \mathcal{O}(h^3).
    \end{align}
    If $\varPi\subset Y$ is a plane disjoint from the centre $\varPi_0$, then
    $\bar\kappa_1=[\varPsi\sigma^*\FFF_{\varPi}]$ and
    $\bar\kappa_2=[\varPsi\sigma^*\PPP_{\varPi}]$.
    The Euler pairing with respect to this basis is given by the matrix
    \[ \chi_{\Ku(\pro^3,\CCC_0)} = \begin{pmatrix}
            -1 & -1 \\ 0 & -1
        \end{pmatrix}. \]
    The Serre functor and the rotation functor act on the basis by
    \vspace*{-0.5em}\begin{equation}
        \begin{tikzcd}[ampersand replacement=\&,cramped,row sep=0, column sep = small]
            {\ab[\mathsf{S}_{\Ku(\pro^3,\CCC_0)}]_* = \ab[\mathsf{O}_{\Ku(\pro^3,\CCC_0)}]_*:} \& \bar \kappa_1 \& {\bar \kappa_2,} \& \bar \kappa_2 \& {\bar \kappa_2 - \bar \kappa_1;}
                        \arrow[maps to, from=1-2, to=1-3]
            \arrow[maps to, from=1-4, to=1-5]
                                \end{tikzcd}
    \end{equation}
\end{proposition}
\vspace*{\parskip}
\begin{proof}
    We begin with a computation of the numerical Grothendieck group of $\Dcatb(\pro^3,\CCC_0)$.  To find a basis, we simply compute the Euler pairing among the objects $\CCC_i$. We have
    \begin{equation}
        \chi_{\CCC_0}(\CCC_i,\CCC_j) = \chi_{\CCC_0}(\CCC_0,\CCC_{j-i}) = \chi(\pro^3,\operatorname{Forg}(\CCC_{j-i})).
    \end{equation}
    From the isometry $\Knum(\Ku(Y)) \cong \Knum(\Ku(\pro^3,\CCC_0))$, we have $\rk \Knum(\Ku(\pro^3,\CCC_0)) = 2$. By \cref{thm:psi}, we have 
    \[ \Knum(\pro^3,\CCC_0) = \Knum(\Ku(\pro^3,\CCC_0)) \oplus \Z[\CCC_1] \oplus \Z[\CCC_2], \]
    and hence $\rk \Knum(\pro^3,\CCC_0) = 4$.
    For $i=-1,0,1,2$, the Clifford sheaves and their Chern characters are given by
    \begin{equation}\label{equ:Clifford_Chern}
        \begin{aligned}
            \CCC_{-1} & = \sheaf{}(-1)^{\oplus 3} \oplus \sheaf{}(-2)^{\oplus 2} \oplus \sheaf{}(-3)^{\oplus 3};      & \Chern{}(\CCC_{-1}) & = 8 - 16h + 19h^2 + \mathcal{O}(h^3); \\
        \CCC_0    & = \sheaf{} \oplus \sheaf{}(-1)^{\oplus 3} \oplus \sheaf{}(-2)^{\oplus 3} \oplus \sheaf{}(-3); & \Chern{}(\CCC_{0})  & = 8 - 12h + 12h^2 + \mathcal{O}(h^3); \\
        \CCC_1    & = \sheaf{}^{\oplus 3} \oplus \sheaf{}(-1)^{\oplus 2} \oplus \sheaf{}(-2)^{\oplus 3};          & \Chern{}(\CCC_{1})  & = 8 - 8h + 7h^2 + \mathcal{O}(h^3);   \\
        \CCC_2    & = \sheaf{}(1) \oplus \sheaf{}^{\oplus 3} \oplus \sheaf{}(-1)^{\oplus 3} \oplus \sheaf{}(-2);  & \Chern{}(\CCC_{2})  & = 8 - 4h + 4h^2 + \mathcal{O}(h^3).
        \end{aligned}
    \end{equation}
    The Euler pairings for $\CCC_{-1}, \CCC_0,\CCC_1,\CCC_2$ are given by
    \[\left(\chi_{\CCC_0}(\CCC_i,\CCC_{j})\right)_{i,j = -1,0,1,2} = \begin{pmatrix}
            1  & 3  & 7 & 14 \\
            0  & 1  & 3 & 7  \\
            -1 & 0  & 1 & 3  \\
            -3 & -1 & 0 & 1
        \end{pmatrix}.\]
    This matrix is unimodular. Therefore $\left\{ [\CCC_{-1}],[\CCC_0],[\CCC_1],[\CCC_2] \right\}$ is an integral basis of $\Knum(\pro^3,\CCC_0)$. The numerical Grothendieck group of the Kuznetsov component of $\Dcatb(\pro^3,\CCC_0)$ is the right orthogonal $\ord{[\CCC_1],[\CCC_2]}^\perp$. A linear-algebraic computation shows that
    \begin{equation}
        \Knum(\Ku(\pro^3,\CCC_0)) = \operatorname{span}\left\{ \bar\kappa_1, \bar\kappa_2\right\}
    \end{equation}
    where $\bar\kappa_1 \coloneqq [\CCC_{-1}] - 4[\CCC_0] + 4[\CCC_1] - [\CCC_2]$ and $\bar\kappa_2 \coloneqq [\CCC_{-1}] - 3[\CCC_0] + [\CCC_1]$, and their Chern characters are given as in the proposition. The Euler pairing with respect to $\{\bar\kappa_1,\bar\kappa_2\}$ is given by
                                                                \[ \chi_{\Ku(\pro^3,\CCC_0)} = \begin{pmatrix}
            -1 & -1 \\ 0 & -1
        \end{pmatrix}. \]
    This proves all the numerical assertions without any auxiliary plane.
    By \cref{lem:disjoint_planes}, we assume that $\varPi\subset Y$ is disjoint from the centre
    $\varPi_0$. Since $\varPsi\circ\sigma^*$ induces an isometry
    $\Knum(\Ku(Y)) \to \Knum(\Ku(\pro^3,\CCC_0))$ and commutes with the
    Serre functor, it remains only to show $\bar \kappa_1 = [\varPsi\sigma^*\FFF_{\varPi}]$;
    this then gives the corresponding identity for $\bar\kappa_2$.

    The object $\varPsi\sigma^*\FFF_{\varPi} \in \Ku(\pro^3,\CCC_0)$ fits into the distinguished triangle
    \begin{equation}
        \begin{tikzcd}[ampersand replacement=\&,cramped]
            {\varPsi\sigma^*\FFF_{\varPi}} \& {\varPsi(\sheaf{\tY}^{\oplus 4})} \& {\varPsi\sigma^*\ideal_{\varPi}(H)} \& {}
            \arrow[from=1-1, to=1-2]
            \arrow[from=1-2, to=1-3]
            \arrow["{+1}", from=1-3, to=1-4]
        \end{tikzcd}
    \end{equation}
    By \cref{lem:psi_kill}, we have $\varPsi(\sheaf{\tilde Y}) = \varPsi(\sheaf{\tilde Y}(H)) = 0$. Hence
    \begin{equation}
        \varPsi\sigma^*\FFF_{\varPi} \cong \varPsi\sigma^*\ideal_{\varPi}(H)[-1] \cong \varPsi\sigma^*\sheaf{\varPi}(H)[-2].
    \end{equation}
    Since $\varPi$ is disjoint from $\varPi_0$, the sheaf $\sigma^*\sheaf{\varPi}(H)$ is supported on the proper transform $\sigma^{-1}\varPi \subset \tilde Y$, which is isomorphic to $\pro^2$. Hence
    \begin{equation}
        \varPsi\sigma^*\FFF_{\varPi} \cong \varPsi\sigma^*\sheaf{\varPi}(H)[-2] \cong \pi_*(\EEE(H+h)|_{\sigma^{-1}\varPi}), \label{equ:P_Pi_Psi}
    \end{equation}
    where $\pi_*(\EEE(H+h)|_{\sigma^{-1}\varPi})$ is a torsion sheaf supported on the hyperplane $\pi(\sigma^{-1}(\varPi)) \cong \pro^2 \subset \pro^3$, and is locally free of rank $4$ on its support (\cite[Lemma 4.7]{kuzQuaFib}). Thus $\Chern{}(\pi_*(\EEE(H+h)|_{\sigma^{-1}\varPi})) = 4h + ch^2 + \mathcal{O}(h^3)$ for some $c \in \Q$. Write its class as $a\bar\kappa_1+b\bar\kappa_2$. The rank-zero condition forces $b=0$, and the coefficient of $h$ then forces $a=1$. Consequently $c=-5$ and
    \begin{equation*}
        [\varPsi\sigma^*\FFF_{\varPi}] = [\pi_*(\EEE(H+h)|_{\sigma^{-1}\varPi})] = \bar \kappa_1. \qedhere
    \end{equation*}
\end{proof}

\section{Constructing stability conditions}\label{sec:stab}
\subsection{General theory}

In this subsection we briefly review the theory of stability conditions on triangulated categories,  introduced by Bridgeland \cite{bri07}. The first step is to define stability on an Abelian category $\cA$, generalizing slope stability on $\Coh(X)$.

\begin{definition}
    Let $\cA$ be a $\C$-linear Abelian category. A \textbf{weak stability function} is a group homomorphism $Z\colon \Gro(\cA) \to \C$ such that, for any $E \in \cA$,
    \[Z(E) \in \left\{ z = m\cdot \e^{\ii\pi\theta} \mid m \geq 0, \ \theta \in (0,1] \right\} = \mathbb{H} \cup \R_{\leq 0}.\]
    If we further require that $Z(E) \ne 0$ for $E \ne 0$, then $Z$ is called a \textbf{stability function}. For physics reasons, $Z$ is also called a \textbf{central charge}.

    For $E \in \cA$, its slope with respect to $Z$ is defined by 
    \[\mu_Z(E) = \begin{cases}
        -\dfrac{\Re Z(E)}{\Im Z(E)}, & \Im Z(E) > 0; \\
        +\infty, & \Im Z(E) = 0.
    \end{cases}\]
    An object $E \in \cA$ is called \textbf{semistable} (\emph{resp.\ }\textbf{stable}) with respect to $Z$, if for any monomorphism $F \hookrightarrow E$ with $F \not\cong E$ in $\cA$, one has $\mu_Z(F) \leq \mu_Z(E)$ (\emph{resp.\ }$ \mu_Z(F) < \mu_Z(E/F)$).
\end{definition}

\begin{remark}
    For an object $E \in \cA$ with $Z(E)\ne0$, its \textbf{phase} is denoted by
    \begin{equation}
        \phi(E) \coloneqq \dfrac{1}{\pi}\arg Z(E)  \in (0,1].
    \end{equation}
    The definition of (semi)stable objects can also be tested using phases instead of slopes. 
\end{remark}

\begin{remark}
    A non-zero object $E \in \cA$ is called \textbf{massless} if $Z(E) = 0$. We assign every massless object the maximal phase $\phi(E)=1$ by convention; with this convention, massless objects are semistable.
\end{remark}

We say that $Z$ satisfies the \textbf{Harder--Narasimhan property} if for any $E \in \cA$, there exists a filtration 
        \[0 = E_0 \subsetneq E_1 \subsetneq \cdots \subsetneq E_\ell =: E\]
by objects $E_i$ in $\cA$, such that the graded factors $E_i/E_{i-1}$ are semistable of phase $\phi_i$, and 
    \[\phi^+(E) \coloneqq \phi_1 > \cdots > \phi_{\ell} =: \phi^-(E).\]

We can now define (weak) pre-stability conditions on a triangulated category. Let $\cT$ be a $\C$-linear triangulated category.

\begin{definition}
    A \textbf{(weak) pre-stability condition} on $\cT$ with respect to $\varLambda$ is a pair $\sigma = (\cA,Z)$, where $\cA$ is a heart of a bounded t-structure of $\cT$, and $Z$ is a (weak) stability function satisfying the Harder--Narasimhan property. An object $E \in \cT$ is called $\bm{\sigma}$\textbf{-semistable} (\emph{resp.\ }$\bm{\sigma}$\textbf{-stable}) if $E = F[k]$ for some $F \in \cA$ and $k \in \Z$, and $F$ is semistable (\emph{resp.\ }stable) with respect to $Z$.
\end{definition}

Alternatively one can define (weak) pre-stability conditions using slicings. We will also write $\sigma=(\cP,Z)$ when using the corresponding slicing $\cP$.

\begin{definition}[][def:slicing]
    Let $\cT$ be a triangulated category. A \textbf{slicing} $\cP$ of $\cT$ is a family of full additive subcategories $\left\{ \cP(\theta)\right\}_{\theta \in \R}$ of $\cT$, satisfying the following conditions:
    \begin{enumerate}[nosep, label = \textnormal{(\roman*)}]
        \item $\cP(\theta)[1] = \cP(\theta+1)$;
        \item For $\theta_1 > \theta_2$ and objects $A \in \cP(\theta_1)$, $B \in \cP(\theta_2)$, one has $\Hom_{\cT}(A,B) = 0$;
        \item (\textbf{Harder--Narasimhan property}) For every $F \in \cT$ there exist real numbers
      \[\phi^+(F) \coloneqq \theta_1>\cdots>\theta_n \eqqcolon \phi^-(F)\]
      and a diagram
      \begin{equation}\label{diag:HN_slicing}
          \begin{tikzcd}[column sep=tiny]
              0 & {F_0} && {F_{1}} && \cdots && {F_{n-2}} && {F_{n-1}} && {F_n} & F \\
              && {A_1} &&& \cdots &&& {A_{n-1}} && {A_n}
              \arrow[equal, from=1-1, to=1-2]
              \arrow[from=1-2, to=1-4]
              \arrow[from=1-4, to=1-6]
              \arrow[from=1-6, to=1-8]
              \arrow[from=1-8, to=1-10]
              \arrow[from=1-4, to=2-3]
              \arrow["{+1}", dashed, from=2-3, to=1-2]
              \arrow[from=1-10, to=2-9]
              \arrow["{+1}", dashed, from=2-9, to=1-8]
              \arrow[from=1-10, to=1-12]
              \arrow["{+1}", dashed, from=2-11, to=1-10]
              \arrow[equal, from=1-12, to=1-13]
              \arrow[from=1-12, to=2-11]
          \end{tikzcd}
      \end{equation}
      where the triangles are distinguished and $A_i \in \cP(\theta_i)$. 
    \end{enumerate}
    If $I \subset \R$ is an interval, we use the notation $\cP(I)$ for the smallest extension-closed full subcategory of $\cT$ containing $\{\cP(\theta) \mid \theta \in I\}$. 
\end{definition}

\begin{remark}
    If $\cP$ is a slicing of $\cT$, then $\cP(0,1]$ is the heart of a bounded t-structure on $\cT$.
\end{remark}

\begin{lemma}
    Consider the data of $(\cP,Z)$ where $\cP$ is a slicing of $\cT$ and $Z\colon \Gro(\cT) \to \C$ is a group homomorphism such that $Z(E) \in \R_{\geq 0}\e^{\ii\pi\theta}$ for $E \in \cP(\theta)$ and $\ker Z \cap \cP(\theta) = \{0\}$ for $\theta \notin \Z$. Then $\sigma = (\cP(0,1], Z|_{\cP(0,1]})$ defines a weak pre-stability condition on $\cT$. Conversely, given a weak pre-stability condition $\sigma = (\cA,Z)$, it induces a slicing $\cP_{\sigma}$ of $\cT$ by 
    \begin{equation}
        \cP_{\sigma}(\theta) \coloneqq \left\{ E \in \cT \mid E \text{ is }\sigma\text{-semistable, and } \phi(E) = \theta \right\}.
    \end{equation}
\end{lemma}

To deform stability conditions and construct moduli spaces, we need the support property, which was formulated by Kontsevich--Soibelman \cite[\S 1.2]{KY08stability}. Fix a finite-rank lattice $\varLambda$ and a surjective group homomorphism $\lambda\colon \Gro(\cT) \cong \Gro(\cA) \to \varLambda$. 

\newcommand{\Ltor}{\varLambda_0}
\begin{definition}[][def:weak_stab_support]
    A \textbf{(weak) stability condition} with respect to $\varLambda$ on $\cT$ is a (weak) pre-stability condition $\sigma=(\cA,Z)$, satisfying the \textbf{support property} with respect to $\varLambda$:
    \begin{enumerate}[dense, label = \textnormal{(\roman*)}]
    \item The central charge factors through the lattice: 
    \begin{equation}
        \begin{tikzcd}[ampersand replacement=\&, column sep=small, cramped]
       Z\colon \Gro(\cA)\cong\Gro(\cT) \arrow[r, "\lambda"] \& \varLambda \arrow[r, "Z_\lambda"] \& \C
        \end{tikzcd}
    \end{equation}
        \item There exists a quadratic form $Q$ on $\varLambda \otimes \R$ such that $Q|_{\ker Z}$  is  negative definite and $Q(\lambda(E)) \geq 0$ for all $\sigma$-semistable objects $E$ in $\cA$.
    \end{enumerate}
\end{definition}

\begin{remark}[{\cite[Lemma A.4]{bay16}}]
    The second part of the support property is equivalent to the following condition: for any fixed norm $\|-\|$ on $\varLambda \otimes \R$, there is a constant $C>0$ such that for any $\sigma$-semistable $E\in\cA$ with $Z(E)\neq 0$, one has $C|Z(E)|\geq\|\lambda(E)\|$. Indeed, given the norm $\|-\|$, the quadratic form can be chosen as $Q(v) = C^2|Z(v)|^2 - \|v\|^2$.
\end{remark}

\begin{remark}
    If $\sigma$ is a stability condition with respect to the numerical lattice $\varLambda = \Knum(\cT)$, then we call $\sigma$ a \textbf{numerical stability condition}. All stability conditions discussed in this paper are numerical.
\end{remark}

The space of stability conditions on $\cT$ with respect to $\varLambda$ is denoted by $\Stab_{\varLambda}(\cT)$, and carries the structure of a complex manifold, a celebrated result by Bridgeland. 

\begin{proposition}[{Bridgeland deformation theorem, \cite[Theorem 1.2]{bri07}}]
    If $\Stab_{\varLambda}(\cT)$ is non-empty, then it can be equipped with the structure of a complex manifold of dimension equal to $\rk\varLambda$, such that the forgetful map 
    \[\begin{tikzcd}[ampersand replacement=\&,cramped,row sep=0]
      {\operatorname{Forg}\colon \Stab_{\varLambda}(\cT)} \& {\Hom_{\Z}(\varLambda,\C)} \\
      {\sigma = (\cA, Z)} \& Z
      \arrow[from=1-1, to=1-2]
      \arrow[maps to, from=2-1, to=2-2]
    \end{tikzcd}\]
    is a local homeomorphism.
\end{proposition}

The stability manifold $\Stab(\cT)$ admits two natural actions:
\begin{enumerate}
    \item A right action of the universal cover of $\GL_2^+(\R)$:
    \[ \tilde{\GL_2^+}(\R) \coloneqq \left\{ (M,g)\ \left|\  \begin{array}{l}
        M \in \GL_2^+(\R),\ g\colon \R \to \R \text{ increasing,}\\ 
        \forall\, \phi\in\R\colon g(\phi + 1) = g(\phi) + 1,\ M\cdot \e^{\ii\pi\phi} \in \R_{>0}\cdot \e^{\ii\pi g(\phi)}
    \end{array}\right.\right\}. \]
    The element $\tilde g = (M,g) \in \tilde{\GL_2^+}(\R)$ acts on $\sigma = (\cA,Z)$ by
    \[ \sigma \cdot \tilde g = (\mathcal{P}_{\sigma}(g(0),g(1)],\ M^{-1}\circ Z), \]
    where $\mathcal P_{\sigma}$ is the slicing of $\cT$ corresponding to $\sigma$. Note that an object $E \in \cT$ is $\sigma$-(semi)stable if and only if it is $(\sigma\cdot \tilde{g})$-(semi)stable.
    
    \item A left action of the auto-equivalences $\Aut_{\varLambda}(\cT) \subset \Aut(\cT)$ preserving the lattice $\varLambda$. A functor $\varGamma \in \Aut_{\varLambda}(\cT)$ acts on $\sigma = (\cA,Z)$ by  
    \[\varGamma \cdot \sigma = (\varGamma(\cA),\ Z\circ\varGamma_*^{-1}),\]
    where $\varGamma_*$ is the automorphism of $\varLambda$ induced by $\varGamma$.
\end{enumerate} 

Constructing stability conditions on the derived category $\Dcatb(X)$ of coherent sheaves on a projective variety $X$ has been a long-standing problem. The existence question was solved by Chunyi Li:

\begin{proposition}[\cite{Li26}][prop:Li_stab_exist]
    Let $X$ be a smooth projective variety over $\C$. Then $\Dcatb(X)$ admits a Bridgeland stability condition.
\end{proposition}

For use in \cref{subsec:Bogomolov}, we define the tilt of a heart with respect to a weak stability condition.

\begin{definition}[][def:tilt_at_slope]
    Given a weak stability condition $\sigma = (\cA, Z)$ on $\cT$ and $\mu \in \R$, consider the full subcategories $\cA^{>\mu}_\sigma$ and $\cA^{\leq\mu}_\sigma$ with objects 
    \[\Obj\cA^{>\mu}_\sigma = \left\{ E \in \cA \mid \mu_Z^-(E) > \mu \right\}; \qquad 
    \Obj\cA^{\leq \mu}_\sigma = \left\{ E \in \cA \mid \mu_Z^+(E) \leq \mu \right\}.\]
    Then $(\cA^{>\mu}_\sigma,\, \cA^{\leq\mu}_\sigma)$ is a torsion pair. The tilted heart, defined as the extension
    \[ \cA^\mu_\sigma \coloneqq \ord{\cA^{>\mu}_\sigma,\, \cA^{\leq\mu}_\sigma[1]}_{\ext},\]
    is a heart of bounded t-structure of $\cT$ (\cite[Proposition 2.1]{HRS}).
\end{definition}

\subsection{A Bogomolov--Gieseker-type inequality}\label{subsec:Bogomolov}

In this subsection we temporarily return to \cref{set:geom_quad_fib}, with $n=m+2$, $k=2$, and $m\geq2$. Thus $Y\subset\pro^{m+3}$ is a smooth cubic hypersurface containing a plane $\varPi_0$, and $\Bl_{\varPi_0}Y\to\pro^m$ is the associated quadric surface fibration. The fibration yields rank $8$ Clifford sheaves
(see \eqref{eq:Cliff_01}):
\begin{equation}
    \CCC_{2j} \cong \sheaf{}(j) \oplus \sheaf{}(j-1)^{\oplus 3} \oplus \sheaf{}(j-2)^{\oplus 3} \oplus \sheaf{}(j-3), \quad \CCC_{2j+1} \cong \sheaf{}(j)^{\oplus 3} \oplus \sheaf{}(j-1)^{\oplus 2} \oplus \sheaf{}(j-2)^{\oplus 3}. \label{equ:CCC_j}
\end{equation}
We recall the notion of slope stability and tilt stability, adapted to the non-commutative projective space $(\pro^m,\CCC_0)$.

\begin{definition}
    Let $\CCC_0$ be a sheaf of even Clifford algebras on $\pro^m$, and let $h$ be the hyperplane divisor of $\pro^m$. Recall that $\ch_{}(E)$ means $\ch_{}(\operatorname{Forg}(E))$ for $E \in \Dcatb(\pro^m,\CCC_0)$. We define $\bm{\mu_h}$\textbf{-stability} on $\Dcatb(\pro^m,\CCC_0)$ as the pair $(\Coh(\pro^m,\CCC_0),\ Z_h)$, where
    \[Z_h(E) \coloneqq -h^{m-1}\ch_{1}(E) + \ii h^{m}\ch_{0}(E).\]
    This gives a weak stability condition on $\Dcatb(\pro^m,\CCC_0)$.
\end{definition}

Next we tilt the $\mu_h$-stability. For $\beta \in \R$, consider the tilted heart
    \[\Coh^\beta(\pro^m,\CCC_0) \coloneqq \ord{\Coh^{> \beta}(\pro^m,\CCC_0), \Coh^{\leq\beta}(\pro^m,\CCC_0)[1]}_{\ext} \subset \Dcatb(\pro^m,\CCC_0),\]
where the torsion pair is taken with respect to $\mu_h$-stability at slope $\beta$ (cf.\ \cref{def:tilt_at_slope}).
    In order to define a tilt stability on $\Dcatb(\pro^m,\CCC_0)$ we introduce a modified Chern character:
\[\ch_{\CCC_0}(E) \coloneqq \ch_{}(\operatorname{Forg}(E))\left( 1-\dfrac{3}{8}h^2 \right),\]
and the associated discriminant:
\begin{align*}
    \Delta_{\CCC_0}(E) & \coloneqq h^{m-2}\left( \ch_{\CCC_0,1}(E)^2 - 2\ch_{\CCC_0,0}(E)\ch_{\CCC_0,2}(E) \right)          \\
                       & = h^{m-2}\left( \ch_{1}(E)^2 - 2\ch_{0}(E)\ch_{2}(E) + \dfrac{3}{4}h^2\ch_{0}(E)^2\right).
\end{align*}
The coefficient $-3/8$ in the modified Chern character is chosen to ensure the following lemma holds. 
  
\begin{lemma}[][]
    $\Delta_{\CCC_0}(\CCC_k) = 0$ for all $k \in \Z$.
\end{lemma}
\begin{proof}
    Note that $\Delta_{\CCC_0}(E \otimes \sheaf{}(kh)) = \Delta_{\CCC_0}(E)$ for any $k \in \Z$. So it suffices to verify that $\Delta_{\CCC_0}(\CCC_0) = \Delta_{\CCC_0}(\CCC_1) = 0$. Note that 
    \[ \ch_{}(\CCC_0) = 8 - 12h + 12h^2 + \mathcal O(h^3); \qquad \ch_{}(\CCC_1) = 8 - 8h + 7h^2 + \mathcal O(h^3). \]
    Note that their unmodified discriminant is given by
    \[\Delta(\CCC_0) = \Delta(\CCC_1) = -48h^m = -\dfrac{3}{4}h^m\ch_{0}(\CCC_0)^2.\]
    In particular the modified discriminants satisfy $\Delta_{\CCC_0}(\CCC_0) = \Delta_{\CCC_0}(\CCC_1) = 0$.
\end{proof}

\begin{theorem}[][prop:BG_noncomm_Pm]
    Let $m\geq2$. Every $\mu_h$-semistable sheaf $E\in\Coh(\pro^m,\CCC_0)$ satisfies the Bogomolov--Gieseker inequality
    \begin{equation}\label{equ:BG_noncomm}
        \Delta_{\CCC_0}(E)
        =h^{m-2}\left(
            \ch_1(E)^2-2\ch_0(E)\ch_2(E)
            +\frac34h^2\ch_0(E)^2
        \right)\geq0.
    \end{equation}
\end{theorem}

\begin{corollary}
    The pair $\sigma_{\beta,\alpha}= (\Coh^\beta(\pro^m,\CCC_0), Z_{\beta,\alpha})$ with
    \begin{equation}
        Z_{\beta,\alpha}(E) \coloneqq -\left( h^{m-2}\ch_{\CCC_0,2}(E) - \alpha\ch_{\CCC_0,0}(E) \right) + \ii \ab(h^{m-1}\ch_{\CCC_0,1}(E) - \beta\ch_{\CCC_0,0}(E))
    \end{equation}
    defines a family of weak stability conditions on $\Dcatb(\pro^m,\CCC_0)$ within the region $\{(\beta,\alpha) \mid \alpha > \frac{1}{2}\beta^2\}$ and with respect to the lattice $\varLambda^2_{\CCC_0} \cong \Z^3$ generated by the vectors $(\ch_{\CCC_0,0}(E),h^{m-1}\ch_{\CCC_0,1}(E),h^{m-2}\ch_{\CCC_0,2}(E))$.
\end{corollary}
In particular the Clifford modules $\CCC_k$ lie on the parabola $\alpha = \frac{1}{2}\beta^2$ which marks the boundary of the tilt stability conditions on $\Dcatb(\pro^m,\CCC_0)$, where $[1:\beta:\alpha] = [\ch_{0,\CCC_0}:\ch_{1,\CCC_0}:\ch_{2,\CCC_0}]$.

For the tilt stability $\sigma_{\beta,\alpha}$ defined as above, its tilt slope is given by
    \[\nu_{\beta,\alpha} = -\dfrac{\Re Z_{\beta,\alpha}}{\Im Z_{\beta,\alpha}} = \dfrac{h^{m-2}\ch_{\CCC_0,2}(E) - \alpha\ch_{\CCC_0,0}(E)}{h^{m-1}\ch_{\CCC_0,1}(E) - \beta\ch_{\CCC_0,0}(E)}.\]

In the remainder of this subsection, we compare the proof of \cref{prop:BG_noncomm_Pm} with \cite[\S 8]{BLMS}, which proves the analogous statement for the even Clifford algebra of a conic fibration. The induction argument is almost identical to that o\cite[Theorem 8.3]{BLMS}, and the major technical difference is the base case $m = 2$. We note that for $m = 2$ the rank of a Clifford module is a multiple of $4$, whereas for $m = 3$ the rank is a multiple of $8$. Therefore we treat the surface case $m = 2$ in full details and only sketch the induction steps.

\begin{proposition}[The $\pro^2$ case][prop:BG_noncomm_P2]
    \cref{prop:BG_noncomm_Pm} holds for $m = 2$.
\end{proposition}

\begin{proof}
    Fix a smooth cubic 4-fold $X$ containing a plane $\varPi$. By \cite[Theorem 4.3]{kuz4fold}, we have the equivalence of the categories $\Ku(X) \simeq \Dcatb(\pro^2,\CCC_0)$. Consider the lattice of $\Knum(\pro^2,\CCC_0)$ defined by 
     \[
        W=\langle[\CCC_0],[\CCC_1],[\CCC_2]\rangle
        \subset\Knum(\pro^2,\CCC_0).
    \]
    The Euler pairings between the basis objects are given by 
    \[ M = \chi_{\CCC_0}(\CCC_i,\CCC_j)_{i,j=0,1,2} = \begin{pmatrix}
            2 & 3 & 6 \\ 3 & 2 & 3 \\ 6 & 3 & 2
        \end{pmatrix}. \]
    In particular $W_{\Q}$ has signature $(1,2)$. By
    \cite[Proposition and Definition 9.5.(c)]{BLMS}, the Euler form on the
    algebraic Mukai lattice of $\Ku(X)$ has signature
    $(\rk\Knum(\Ku(X))-2,2)$. Hence $W^\perp_{\Q} = \ker\ab(\ch\circ\operatorname{Forg})$ is positive definite. This allows us to devise a Riemann--Roch-type theorem on $\Dcatb(\pro^2,\CCC_0)$. For $E \in \Knum(\pro^2,\CCC_0)$, let $[E] = w_E + k_E$, with $w_E = a[\CCC_0] + b[\CCC_1] + c[\CCC_2] \in W_{\Q}$ and $k_E \in W^{\perp}_{\Q}$. We have
    \[ \begin{pmatrix}
            \ch_{0}(E) \\ \ch_{1}(E) \\ \ch_{2}(E)
        \end{pmatrix}
        = \begin{pmatrix}
            8 & 8 & 8 \\ -12 & -8 & -4 \\ 12 & 7 & 4
        \end{pmatrix}\begin{pmatrix}
            a \\b \\ c
        \end{pmatrix}.\]
    Then a change of basis gives
    \begin{align*}
        q(\ch(E)) \coloneqq \chi_{\CCC_0}(w_E,w_E)
                                                          & = -\dfrac{1}{8}\left( \ch_{1}(E)^2 - 2\ch_{0}(E)\ch_{2}(E) + \dfrac{1}{2}\ch_{0}(E)^2 \right).
    \end{align*}
    And we have that 
    \begin{equation}
        q(\ch(E)) \leq q(\ch(E)) + \chi_{\CCC_0}(k_E,k_E) = \chi_{\CCC_0}(E,E).
    \end{equation}
    To prove the inequality $\Delta_{\CCC_0}(E) \geq 0$ we need to bound the rank of $E$. It suffices to prove the inequality for $\mu_h$-stable objects. Suppose that $E$ is $\mu_h$-stable with $\rk(E) \geq 8$, then 
    \begin{equation}\label{equ:q<=2}
        q(\ch(E)) \leq \chi_{\CCC_0}(E,E) = 2 - \ext^1(E,E) \leq 2 \leq \frac{\rk(E)^2}{32}.
    \end{equation}
    Reorganizing the equation, we find that this is just \eqref{equ:BG_noncomm}. Next we need to rule out the possibility of $\rk(E) < 8$, which is given in the following lemmata.
    \begin{lemma}[][lem:Chern_noncomm_P3]
        Let $m = 3$. Every object of $\Dcatb(\pro^3,\CCC_0)$ has rank in $8\Z$.
    \end{lemma}
    \begin{proof}
        In the proof of \cref{prop:Knum_Ku_C0} we have shown that $\left\{ [\CCC_{-1}],[\CCC_0],[\CCC_1],[\CCC_2] \right\}$ is an integral basis of $\Knum(\pro^3,\CCC_0)$. Since $\rk(\CCC_j) \in 8\Z$ for $j = -1,0,1,2$, the same holds for any $E \in \Dcatb(\pro^3,\CCC_0)$.
    \end{proof}

    \begin{lemma}[][lem:Chern_noncomm_P2]
        Let $m = 2$. For every $E \in \Dcatb(\pro^2,\CCC_0)$ with $\ch(E) = r + ah + bh^2$, one has $r \in 4\Z$ and $\frac{1}{4}r \equiv a \bmod 2$.
    \end{lemma}
    \begin{proof}
        Choose a smooth cubic-$5$-fold extension $Y$ of the cubic $4$-fold $X$.
        For $i\colon X\hookrightarrow Y$, the push-forward $i_*E \in \Dcatb(\pro^3,\CCC_0)$. Write uniquely $[i_*E]=\displaystyle\sum_{j=-1}^2n_j[\CCC_{j}]$ with $n_j\in\Z$. Grothendieck--Riemann--Roch theorem gives
        \[
            \ch(i_*E)
            =rh+\left(a-\frac r2\right)h^2
             +\left(b-\frac a2+\frac r6\right)h^3.
        \]
        Recall that the Chern characters of the Clifford modules $\CCC_j$ are given in \eqref{equ:Clifford_Chern}.
                                                                                The comparison of ranks gives $\sum_j n_j=0$. Comparing first Chern classes then yields
        \[
            r=4\sum_{j=-1}^2 j n_j \in 4\Z.
        \]
        Reducing the equality for $\ch_2$ modulo $2$ gives
        \[
            a-\frac r2\equiv\sum_j j n_j=\frac r4\pmod2.
        \]
        Since $r/2$ is even, this is the required parity.
    \end{proof}
    \noindent\textbf{Proof of \cref{prop:BG_noncomm_P2} continued. }Let $E \in \Coh(\pro^2,\CCC_0)$ with
    with $\ch(E) = r + ah + bh^2$. Suppose that $E$ is $\mu_h$-stable, $\Delta_{\CCC_0}(E) < 0$. Then we must have $r < 8$, and \cref{lem:Chern_noncomm_P2} shows that $r = 4$ and $a$ is odd. Note that 
    \begin{equation}
        \Delta_{\CCC_0}(E) = \frac{r^2}{4} - 8q < 0 \implies q(\ch(E)) > \frac{1}{2}.
    \end{equation}
    Combining with \eqref{equ:q<=2}, we have 
    \begin{equation}
        \frac{1}{2} < q = b - \frac{a^2}{8} - 1 \leq \chi_{\CCC_0}(E,E) \leq 2.
    \end{equation}
    As $a$ is odd, and 
    \begin{equation}
        b - \frac{a^2}{2} = -\chern{2}(E) \in \Z,
    \end{equation}
    it follows that $q(\ch(E)) = \frac{11}{8}$ and $\chi_{\CCC_0}(E,E) = 2$. Put $e=[E] \in \Ktop(\Ku(X))$ and $W=\langle[\CCC_0],[\CCC_1],[\CCC_2]\rangle \subset \Ktop(\Ku(X))$. Note that $q(\ch(E))$ is the
    square of the orthogonal projection of $e$ to $W$. Set $N = W + \Z e$. Since
    $\operatorname{disc}(W)=8$, the Schur complement gives
    \[
        \operatorname{disc}(N) = \operatorname{disc}(W+\Z e)
        = \operatorname{disc}(W) \cdot \ab(\chi(e,e)-q(\ch(E)))=5.
    \]
    As $\operatorname{disc}(N)$ is square-free, it is a primitive rank 4 lattice in $\Ktop(\Ku(X))$. Additionally, we note that the distinguished negative $\mathsf{A}_2$-sublattice of $\Ktop(\Ku(X))$ is contained in $N$. The lattice correspondence of \cite[Propositions 2.4 and 2.5]{AddingtonThomas14} therefore associates to $N$ a primitive positive-definite algebraic
    lattice $M\subset \Hlg^{2,2}(X,\Z)$ with
    \[
        h^2\in M,\qquad \rk(M)=3,\qquad  \operatorname{disc}(M)=5.
    \]
    Its covolume is $\sqrt5$. Choose $R<\sqrt3$ sufficiently close to $\sqrt3$ that $\frac{4\pi}{3}R^3>8\sqrt5$. Minkowski's convex-body theorem gives $0\neq v\in M$ with $v^2<3$, so $n\coloneqq v^2$ is $1$ or $2$.

    Set $d=\ord{v,h^2}$ and $\delta \coloneqq 3v-dh^2\in \Hlg^4_{\mathrm{prim}}(X,\Z)$. Since the primitive lattice is even, $\delta^2=3(3n-d^2)\in2\Z$; positive definiteness gives
    \[
        d\equiv n\pmod2, \qquad d^2<3n.
    \]
    Hence $(n,d)=(1,\pm1),(2,0)$, or $(2,\pm2)$. If $(n,d)=(2,0)$, then $v$ is an algebraic class in $\Hlg^4_{\mathrm{prim}}(X,\Z)$ of square $2$, hence primitive in the integral lattice, contradicting
    \cite[Proposition 2.15.(i)]{Laza10period}. In every other case, $\delta=3v-dh^2\in \Hlg^4_{\mathrm{prim}}(X,\Z)$
    is an algebraic class of square $6$, hence primitive in the integral lattice, and $ \ord{\delta,x} = 3\ord{v,x}\in3\Z$ for every $x\in \Hlg^4_{\mathrm{prim}}(X,\Z)$. This contradicts \cite[Proposition 2.15.(ii)]{Laza10period} and proves the inequality.
\end{proof}

\begin{proof}[Proof of \cref{prop:BG_noncomm_Pm}]
    The proof for $m \geq 2$ uses induction on $m$ and on $\rk(E)$. It is a formal consequence of \cite[Theorem 8.3]{BLMS}, which was a variant of Langer's restriction theorem in \cite{lan04}. We sketch the idea and explain only the parts that differ from the conic fibration case. 
    \begin{itemize}[nosep]
        \item Let $\mathrm{BG}_m(r)$ denote the following stronger statement of \cref{prop:BG_noncomm_Pm}: if
        $\psi\colon T\to\pro^m$ is an iterated blow-up along strict transforms of
        codimension-$2$ linear subspaces and $h=\psi^*\sheaf{\pro^m}(1)$, then
        every $\mu_h$-semistable torsion-free $E \in \Coh(T,\psi^*\CCC_0)$ with $\rk(E) \leq r$ satisfies $\Delta_{\psi^*\CCC_0} \geq 0$. 

        \item Let $\mathrm R_m(r)$ denote the effective restriction inequality stated in \cite[Theorem 8.10]{BLMS}: if $E \in \Coh(T,\psi^*\CCC_0)$ is a $\mu_h$-semistable torsion-free sheaf with $\rk(E) \leq r$, and the restriction $E|_D$ to a general divisor $D \in |h|$ is not $\mu_{h}$-semistable, then 
            \begin{equation}
                \sum_{i < j} r_ir_j(\mu_i-\mu_j)^2 \leq \Delta_{\psi^*\CCC_0}(E)
            \end{equation}
        where $\mu_i$ (\emph{resp.\ }$r_i$) are the slopes (\emph{resp.}\ ranks) of the HN factors of $E|_D$. 
        \end{itemize}
    Note that the surface case $\mathrm{BG}_2(r)$ for all $r \in 4\Z$ holds on $\pro^2$ by \cref{prop:BG_noncomm_P2}, and on the blow-ups $T \to \pro^2$ by the same argument as in \cite[Lemma 8.6]{BLMS} (with the constant $11/32$ replaced by $3/8$). 

    We use induction on $r\in8\Z_{>0}$, simultaneously for every $m\geq3$ and every such $T$. The base case is $\mathrm{R}_3(8)$. We claim that if $E \in \Coh(T,\psi^*\CCC_0)$ is of rank $8$, then the restriction $E|_{D_t}$ for a general $D_t \in |h|$ is semistable. If not, then by \cref{lem:Chern_noncomm_P2}, $E|_{D_t}$ must have exactly two HN factors $E_{1,t}$ and $E_{2,t}$, each of rank $4$. Since this is true for general $D_t$, we move $D_t$ within a general pencil $\varSigma \subset |h|$. The relative HN filtration of $E|_D$ over the pencil $\varSigma$ defines the objects $E_1$, $E_2$ on the incidence variety $\tilde T = \{(D_t,y) \in \varSigma \times T \mid y \in D_t\}$, which is a blow-up of $T$. But then $E_1$, $E_2$ are rank $4$ objects of $\tilde T$, contradicting \cref{lem:Chern_noncomm_P3}. By induction on $m$, we obtain $\mathrm{R}_m(8)$ for every $m \geq 3$.

    The next step is $\mathrm{R}_m(r) \implies \mathrm{BG}_m(r)$. If $E \in \Coh(T,\psi^*\CCC_0)$ is $\mu_h$-semistable with $\Delta_{\psi^*\CCC_0}(E) < 0$, then we can choose a flag of general divisors $D = D_2 \subset \cdots \subset D_m = T$, with $D_i \in |h|_{D_{i+1}}|$, and by $\mathrm{R}_m(r)$ we have that $E|_{D}$ is $\mu_h$-semistable with $\Delta_{\psi^*\CCC_0}(E|_D) < 0$. This contradicts $\mathrm{BG}_2(r)$.

    Finally we need to show $\mathrm{BG}_m(r) \implies \mathrm{R}_m(r+8)$. But this follows verbatim from the last part of the proof of \cite[Theorem 8.3]{BLMS}. This concludes the induction. \qedhere

                                    \end{proof}

\subsection{Restricting stability conditions to the Kuznetsov component}\label{subsec:restrict_Ku_smooth}

Consider the rotation of the tilt stability $\sigma_{\beta,\alpha}$ at the tilt slope $\nu_{\beta,\alpha} = t$:
\begin{equation}
    \Coh^t_{\beta,\alpha}(\pro^3,\CCC_0) \coloneqq \ord{\Coh^{\nu>t}_{\beta,\alpha}(\pro^3,\CCC_0),\ \Coh^{\nu\leq t}_{\beta,\alpha}(\pro^3,\CCC_0)[1]}_{\ext}; 
    \qquad 
    Z^t_{\beta,\alpha} \coloneqq \ab(-t-\ii) Z_{\beta,\alpha}.
\end{equation}
\begin{lemma}[{\cite[Proposition 2.15]{BLMS}}]
     The rotated tilt stability $\sigma^t_{\beta,\alpha} \coloneqq (\Coh^t_{\beta,\alpha}(\pro^3,\CCC_0), Z^t_{\beta,\alpha})$ is also a weak stability condition on $\Dcatb(\pro^3,\CCC_0)$.
\end{lemma}

To construct stability conditions on Kuznetsov components, we use the following criterion from \cite{BLMS}, which restricts stability conditions to semi-orthogonal components.

\begin{proposition}[{Restricting stability conditions, \cite[Proposition 5.1]{BLMS}}][prop:ind_stab]
    Let $\sigma = (\cA, Z)$ be a weak stability condition on the triangulated category $\cT$ with a Serre functor $\sfS$. Assume that $\cT$ has the semi-orthogonal decomposition $\cT = \ord{\cD,E_1,...,E_m}$, where $E_i \in \cT$ are exceptional objects. If for $i = 1,...,m$ the following conditions are satisfied:
    \begin{enumerate}[nosep, label = \textnormal{(\roman*)}]
        \item $E_i \in \cA$;
        \item $\sfS(E_i) \in \cA[1]$;
        \item $Z(E_i) \ne 0$;
    \end{enumerate}
    Then $\sigma' = (\cA'\coloneqq \cA \cap \cD,\ Z|_{\cA'})$ is a weak stability condition on $\cD$. Moreover, if $Z(F) \ne 0$ for any non-zero $F \in \cA'$, then $\sigma'$ is a stability condition on $\cD$.
\end{proposition}

Now we present the main theorem of this section.

\begin{theorem}[][thm:main]
    Let $Y$ be a smooth cubic 5-fold over $\C$. Then $\Ku(Y)$ has a family of Bridgeland stability conditions
    \[ \sigma'_{\beta,\alpha} = \left( \cA_{\beta,\alpha},\ Z'_{\beta,\alpha} \right), \]
    where $\cA_{\beta,\alpha} \coloneqq \Coh^{-\frac{5}{4}}_{\beta,\alpha}(\pro^3,\CCC_0) \cap \varPsi\sigma^*\Ku(Y)$ and $Z'_{\beta,\alpha} \coloneqq \ab(\frac{5}{4}-\ii) Z_{\beta,\alpha} \big|_{\cA_{\beta,\alpha}}$. The family is parametrized by the open set 
    \begin{equation}
        V \coloneqq \left\{(\beta,\alpha) \in \R^2\ \bigg|\ -\frac{3}{2} < \beta < -1,\ \frac{1}{2}\beta^2 < \alpha < -\frac{5}{4}\beta - \frac{3}{4} \right\}. \label{equ:para_ab}
    \end{equation}
\end{theorem}
\begin{proof}
    Since the Serre functor of $\Dcatb(\pro^3,\CCC_0)$ is $-\otimes_{\CCC_0} \CCC_{-2}[3]$, to use \cref{prop:ind_stab} on the semi-orthogonal decomposition \eqref{eq:twist_P3_SOD}, we need to check that $\CCC_1, \CCC_2 \in \Coh^{-\frac{5}{4}}_{\beta,\alpha}(\pro^3,\CCC_0)$ and $\CCC_{-1}[2],\CCC_{0}[2] \in \Coh^{-\frac{5}{4}}_{\beta,\alpha}(\pro^3,\CCC_0)$.

    First note that 
    \[\ch_{\CCC_0}(\CCC_j) = 8 +4(j-3)h + \left(j-3\right)^2h^2 + \mathcal{O}(h^3).\]
    Then $\mu_h(\CCC_j) = \frac{1}{2}(j-3)$ for all $j \in \Z$. Fix $\beta \in \left( -\frac{3}{2}, -1 \right)$, we have
    \[\mu_h(\CCC_{-1}) < \mu_h(\CCC_{0}) < \beta < \mu_h(\CCC_{1}) < \mu_h(\CCC_{2}).\]
    Then we have $\CCC_{-1}[1], \CCC_{0}[1], \CCC_1, \CCC_2 \in \Coh^\beta(\pro^3,\CCC_0)$. Meanwhile, the tilt slope $\nu_{\beta,\alpha}$ of $\CCC_j$ is given explicitly by
    \[\nu_{\beta,\alpha}(\CCC_j) = \dfrac{h^2\ch_{\CCC_0,2}(\CCC_j) - \alpha\ch_{\CCC_0,0}(\CCC_j)}{h\ch_{\CCC_0,1}(\CCC_j) - \beta\ch_{\CCC_0,0}(\CCC_j)} =
        \dfrac{\left( j-3 \right)^2 - 8\alpha}{4j-8\beta-12}.
    \]
    In particular, for $\frac{1}{2}\beta^2 < \alpha < -\frac{5}{4}\beta - \frac{3}{4}$, we have
    \[ \nu_{\beta,\alpha}(\CCC_{-1}[1]) < \nu_{\beta,\alpha}(\CCC_{0}[1]) < -\frac{5}{4} < \nu_{\beta,\alpha}(\CCC_{1}) < \nu_{\beta,\alpha}(\CCC_{2}).\]
    Hence $\CCC_{-1}[2]$, $\CCC_0[2]$, $\CCC_1$, $\CCC_2 \in \Coh^{-\frac{5}{4}}_{\beta,\alpha}(\pro^3,\CCC_0)$ as claimed. It is clear that these objects have non-zero central charge. Then $\sigma'_{\beta,\alpha}$ is a weak stability condition on $\Ku(Y)$ by \cref{prop:ind_stab}.
                    
    To show that $\sigma'_{\beta,\alpha}$ is a stability condition, it remains to show that $Z_{\beta,\alpha}(E) \ne 0$ for all non-zero object $E$ in $\Coh^{-\frac{5}{4}}_{\beta,\alpha}(\pro^3,\CCC_0) \cap \varPsi\sigma^*\Ku(Y)$.
    Suppose that $E \in \Coh^{-\frac{5}{4}}_{\beta,\alpha}(\pro^3,\CCC_0)$ is such that $Z_{\beta,\alpha}(E) = 0$. Then the proof of \cite[Proposition 2.18]{BLMS} shows that $\operatorname{Forg}(E)$ is supported in codimension $\geq 3$. But then
    \[ \Hom_{\Dcatb(\pro^3,\CCC_0)}(\CCC_2,E) = \Hom_{\Dcatb(\pro^3,\CCC_0)}(\CCC_0(h),E) = \Hom_{\Dcatb(\pro^3)}(\sheaf{\pro^3},\operatorname{Forg}(E)(-h)) \ne 0.\]
    Hence $E \notin \Ku(Y)$. We conclude that $\cA_{\beta,\alpha}$ has no massless objects, and $\sigma'_{\beta,\alpha}$ is a stability condition on $\Ku(Y)$.
\end{proof}

\begin{remark}[][rmk:Pi_indep]
    Note that the Clifford algebra structure and the embedding $\Ku(Y) \hookrightarrow \Dcatb(\pro^3,\CCC_0)$ depends on the choice of the blown-up plane $\varPi_0 \subset Y$. In particular, for fixed $(\beta,\alpha)$ and as $\varPi_0$ varies in the Hilbert scheme $\mathcal{F}_{2}(Y)$ of planes of $Y$, the stability conditions $(\sigma'_{\beta,\alpha})_{\varPi_0}$ is a family over $\mathcal{F}_2(Y)$ satisfying the openness property (see \cite[Definition 20.1.(2)]{BLMNPS} for precise statements). The argument in \cite[Proposition 2.6]{LPZ23} proves that the $\sigma'_{\beta,\alpha}$-stability of any object is in fact independent of the choice of $\varPi_0$.
\end{remark}

\section{Serre invariance}\label{sec:Serre}
\subsection{Tensor products and tilted hearts}\label{subsec:tensor}

We will use the remaining two subsections to prove that the stability conditions on the Kuznetsov component of a cubic 5-fold constructed in \cref{thm:main} are Serre-invariant (see \cref{def:Serre_inv}), which completes the second part of the main theorem \ref{thm:main_intro}. The argument adapts the proofs for cubic 3-folds (\cite{PY20,FP23}).

In this subsection we study the action of the tensor product functor $(- \otimes_{\CCC_0} \CCC_1)$ on the tilted heart $\Coh^{-\frac{5}{4}}_{\beta,\alpha}(\pro^3,\CCC_0)$ with a fixed $\beta = -\frac{5}{4}$. The result applies more broadly to general tilt stability and therefore carries independent interest. The method is similar to \cite[Lemma 5.2]{PY20} and we hope to clarify the details missing in the original proof.

We will begin with a discussion of how $(- \otimes_{\CCC_0} \CCC_1)$ changes the character and the stability of objects in $\Dcatb(\pro^3,\CCC_0)$. 

\begin{lemma}[][tensor_stab_para]
    An object $E \in \Dcatb(\pro^3,\CCC_0)$ is $\sigma_{\beta,\alpha}$-semistable if and only if $E \otimes_{\CCC_0} \CCC_1$ is $\sigma_{\beta',\alpha'}$-semistable, where 
    \begin{equation}
        \beta' = \beta + \frac{1}{2}, \qquad \alpha' = \alpha+\frac{1}{2}\beta + \frac{1}{8}.
    \end{equation}
    Moreover, if $E$ is a torsion-free sheaf, then $E$ is $\mu_h$-semistable if and only if $E \otimes_{\CCC_0} \CCC_1$ is also $\mu_h$-semistable.
\end{lemma}
\begin{proof}
    For $E \in \Dcatb(\pro^3,\CCC_0)$, we will compute $\ch(E \otimes_{\CCC_0} \CCC_1)$. Since  $\Dcatb(\pro^3,\CCC_0) = \Dperf(\pro^3,\CCC_0)$, we may harmlessly assume that $E = E' \otimes \CCC_0$ for some $E' \in \Dcatb(\pro^3)$. Then we have  
    \begin{align}
        \ch(E \otimes_{\CCC_0} \CCC_1) &= \ch(E' \otimes \CCC_0 \otimes_{\CCC_0} \CCC_1)  = \ch(E' \otimes \CCC_1) = \ch(E')\cdot\ch(\CCC_1) \\ 
        &= \ch(E)\cdot\dfrac{\ch(\CCC_1)}{\ch(\CCC_0)} = \ch(E)\cdot\dfrac{8-8h+7h^2 - \tfrac{13}{3}h^3}{8-12h+12h^2 -9h^3} \\
        &= \ch(E)\left( 1 + \dfrac{1}{2}h + \dfrac{1}{8}h^2 + \dfrac{1}{48}h^3\right) \\
        &= \e^{h/2}\ch(E) \eqqcolon \ch^{-\frac{1}{2}}(E). \label{equ:Chern_tensor}
    \end{align}
    The tilt slope of $E \otimes_{\CCC_0} \CCC_1$ is given by 
    \begin{align*}
        \nu_{\beta,\alpha}(E \otimes_{\CCC_0} \CCC_1) 
    &= \dfrac{\ch_{\CCC_0,2}(E \otimes_{\CCC_0} \CCC_1) - \alpha\ch_{\CCC_0,0}(E \otimes_{\CCC_0} \CCC_1)}{\ch_{\CCC_0,1}(E \otimes_{\CCC_0} \CCC_1) - \beta\ch_{\CCC_0,0}(E \otimes_{\CCC_0} \CCC_1)} 
    =  \dfrac{\ch_{\CCC_0,2}^{-\frac{1}{2}}(E) - \alpha\ch_{\CCC_0,0}^{-\frac{1}{2}}(E)}{\ch_{\CCC_0,1}^{-\frac{1}{2}}(E) - \beta\ch_{\CCC_0,0}^{-\frac{1}{2}}(E)} \\
    &= \dfrac{\ch_{\CCC_0,2}(E) - \ab(\alpha - \frac{1}{2}\beta + \frac{1}{8})\ch_{\CCC_0,0}(E)}{\ch_{\CCC_0,1}(E) - \ab(\beta-\frac{1}{2})\ch_{\CCC_0,0}(E)} + \frac{1}{2} 
    = \nu_{\beta-\frac{1}{2},\, \alpha - \frac{1}{2}\beta + \frac{1}{8}}(E) + \frac{1}{2}.
    \end{align*}
    Therefore, $E$ is $\sigma_{\beta,\alpha}$-semistable if and only if $E \otimes_{\CCC_0} \CCC_1$ is $\sigma_{\beta+\frac{1}{2},\, \alpha + \frac{1}{2}\beta + \frac{1}{8}}$-semistable. By letting $\alpha \to +\infty$ we deduce that $E$ is $\mu_h$-semistable if and only if $E \otimes_{\CCC_0} \CCC_1$ is also $\mu_h$-semistable.
\end{proof}

For the rest of the subsection and for most applications in the remainder of this paper, we can fix $\beta = -\frac{5}{4}$. We need a better parametrization for the tilt stability conditions on $\Dcatb(\pro^3,\CCC_0)$ and for the stability conditions on $\Ku(X)$. The new parameters are given by a shift that centralizes the vertical line $\beta = -\frac{5}{4}$. For this we transform $(\beta,\alpha)$ into new coordinates $(\xi,\eta)$ via
\begin{equation}
    \begin{cases}
        \xi = \beta + \dfrac{5}{4}, \\[2ex]
        \eta = \alpha + \dfrac{5}{4}\beta + \dfrac{25}{32}.
    \end{cases} \label{equ:tilt_new_coor}
\end{equation}
The tilt stability conditions on $\Dcatb(\pro^3, \CCC_0)$ are expressed in the new coordinates $(\xi,\eta)$ as:
\[ \tilde\sigma_{\xi,\eta} = \ab(\Coh^{\xi-\frac{5}{4}}(\pro^3,\CCC_0),\quad \tilde Z_{\xi,\eta} = -\left(\ch_{\CCC_0,2}^{-\frac{5}{4}} - \eta\ch_{\CCC_0,0}^{-\frac{5}{4}}\right) + \ii\ab( \ch_{\CCC_0,1}^{-\frac{5}{4}} - \xi\ch_{\CCC_0,0}^{-\frac{5}{4}} ) ).\]
Here $\ch_{\CCC_0}^{-\frac{5}{4}}$ is the twisted Chern character: 
\begin{equation}
    \ch_{\CCC_0}^{-\frac{5}{4}}(E) = \e^{5h/4}\ch_{\CCC_0}(E).
\end{equation}
The new tilt slope is given by 
\[ \tilde\nu_{\xi,\eta}(E) = -\dfrac{\Re \tilde Z_{\xi,\eta}(E)}{\Im \tilde Z_{\xi,\eta}(E)}  = \dfrac{\ch_{\CCC_0,2}^{-\frac{5}{4}}(E) - \eta\ch_{\CCC_0,0}^{-\frac{5}{4}}(E)}{\ch_{\CCC_0,1}^{-\frac{5}{4}}(E) - \xi\ch_{\CCC_0,0}^{-\frac{5}{4}}(E)}\implies \tilde\nu_{\xi,\eta}(E) = \nu_{\beta,\, \alpha}(E) +  \dfrac{5}{4}. \]
Hence an object is $\tilde\sigma_{\xi,\eta}$-(semi)stable if and only if it is $\sigma_{\beta,\alpha}$-(semi)stable. Define the heart 
\begin{equation}
    \tilde{\Coh}^0_{\xi,\eta} = \ord{ \tilde{\Coh}^{\tilde{\nu} > 0}_{\xi,\eta},\, \tilde{\Coh}^{\tilde{\nu} \leq 0}_{\xi,\eta}[1]}_{\ext}
\end{equation}
to be the tilt of $\Coh^{\xi - \frac{5}{4}}(\pro^3,\CCC_0)$ at $\tilde{\nu}_{\xi,\eta} = 0$, where $0 < \eta < \frac{1}{32}$. Then we can identify $\tilde{\Coh}^0_{\xi,\eta}$ with $\Coh^{-\frac{5}{4}}_{\beta,\alpha}(\pro^3,\CCC_0)$.
In the $(\xi,\eta)$-coordinates, the stability conditions on $\Ku(Y)$ are expressed as 
\begin{equation}
    \ts'_{\xi,\eta} = \ab(
        \tA_{\xi,\eta} \coloneqq \tilde{\Coh}^0_{\xi,\eta} \cap \Ku(\pro^3,\CCC_0),
        \quad
        \tZ^0_{\xi,\eta} = -\ii \tZ_{\xi,\eta} = \ab(\ch_{\CCC_0,1}^{-\frac{5}{4}} - \xi  \ch_{\CCC_0,0}^{-\frac{5}{4}})
        + \ii\ab(\ch_{\CCC_0,2}^{-\frac{5}{4}} - \eta\ch_{\CCC_0,0}^{-\frac{5}{4}}) 
    ).
\end{equation}
The admissible region of the parameters is given simply by 
\begin{equation}\label{equ:para_ab_twisted}
    \tilde{V} \coloneqq \ab\{(\xi,\eta) \in \R^2 \ \bigg|\  |\xi| < \frac{1}{4},\ \frac{1}{2}\xi^2 < \eta < \frac{1}{32}\}.
\end{equation}

The wall-and-chamber structure in our new $(\xi,\eta)$-coordinates can be described as follows: every point $P = (\xi(P),\eta(P)) \in \R^2$ lying above the parabola $\varGamma\colon \eta = \frac{1}{2}\xi^2$ defines a weak stability condition on $\Dcatb(\pro^3,\CCC_0)$. Meanwhile, for every numerical character $[E] \in \Knum(\pro^3,\CCC_0)$, we define a point $\left[\ch_{\CCC_0,0}^{-\frac{5}{4}}(E):\ch_{\CCC_0,1}^{-\frac{5}{4}}(E):\ch_{\CCC_0,2}^{-\frac{5}{4}}(E)\right]$ in the real projective plane $\RP^2$. If $\ch_{0}(E) \ne 0$, we can identify the chart $\left\{ [1:\xi:\eta] \mid (\xi,\eta) \in \R^2\right\}$ of $\RP^2$ with the $(\xi,\eta)$-plane, in which $E$ is represented by the point 
\[  \tv(E) \coloneqq \left( \dfrac{\ch_{\CCC_0,1}^{-\frac{5}{4}}(E)}{\ch_{\CCC_0,0}^{-\frac{5}{4}}(E)},\ \dfrac{\ch_{\CCC_0,2}^{-\frac{5}{4}}(E)}{\ch_{\CCC_0,0}^{-\frac{5}{4}}(E)} \right).\]
By the Bogomolov--Gieseker inequality (\cref{prop:BG_noncomm_Pm}), if $E \in \Dcatb(\pro^3,\CCC_0)$ is tilt semi-stable, then the point $\tv(E)$ lies in the region $\{ (\xi,\eta) \in \R^2 \mid \eta \leq \frac{1}{2}\xi^2 \}$, which is below the parabola $\varGamma$. The phase $\phi_P(E)$ of $E$ with respect to the stability condition $\tilde\sigma_{P} = \tilde{\sigma}_{\xi(P),\eta(P)}$ is determined by 
\begin{equation}
    \phi_P(E) \equiv \dfrac{1}{\pi}\arg(\ell_{P}^-,\ell_{PE}^+)\ \bmod 2\Z,
\end{equation} 
where $\ell_{P}^-$ is the downward-pointing vertical ray starting at $P$, and $\ell_{PE}^+$ is the ray starting at $P$ and passing through $\tv(E)$.
This is shown in \cref{fig:tilt_phase}. Note that this result is still true for the case $\ch_{0}(E) = 0$ where the point $\tv(E)$ is placed at the infinity. In particular, for objects $E$ and $F$ with non-zero central charge, one has $\phi_P(E) = \phi_P(F)$ if and only if the points $P = (\xi,\eta)$, $\tv(E)$, and $\tv(F)$ are collinear in the plane.
\begin{figure}[H]
    \centering 
    \begin{tikzpicture}[line cap=round,line join=round,>=triangle 45,x=0.7cm,y=0.7cm,scale=1]
    \begin{axis}[
            x=0.7cm,y=0.7cm,
            axis lines=middle,
            xmin=-2,
            xmax=5,
            ymin=-1,
            ymax=6,
            xlabel={\scriptsize$\xi$},
            ylabel={\scriptsize$\eta$},
            xtick distance = 10,
            ytick distance = 10,]
        \clip(-2,-1.7158718302765568) rectangle (5,7);
        
        \draw [shift={(-1,2)}] (0,0) -- (-90:0.47047833014623636) arc (-90:33.690067525979806:0.47047833014623636) -- cycle;
        \addplot [samples=500, smooth, line width=0.5pt, domain=-2:5] {x^2/2};
        \draw [line width=1pt] (1,3) -- (1,-1);
        \draw [line width=1pt,domain=1:5] plot(\x,{(--7--2*\x)/3});
        \begin{scriptsize}
            \draw[color=black] (3,5.7) node {$V$};
            \draw [fill=black] (1,3) circle (1.5pt);
            \draw[color=black] (1.2570496928226398,3.4985963288442212) node {$P$};
            \draw [fill=black] (4,5) circle (1.5pt);
            \draw[color=black] (4.5,4.5) node {$\tv(E)$};
            \draw[color=black] (2.4,2.6) node {$\pi\phi_{P}(E)$};
            \draw[color=black] (1.5,-0.5) node {$\ell_{P}^-$};
            \draw[color=black] (4,5.6) node {$\ell_{PE}^+$};
        \end{scriptsize}
    \end{axis}
\end{tikzpicture}
    \caption{The phase $\phi_P(E)$ of an object $E$ with respect to the stability condition $\ts_{\xi(P),\eta(P)}$.}
    \label{fig:tilt_phase}
    \end{figure}

As we will deform the the tilt stability on the $(\xi,\eta)$-plane, the following lemma from \cite{LZ19} allows us to control the range of phase of a stable object under deformation. 

\begin{lemma}[{\cite[Lemma 3]{LZ19}}][lem:LZ19]
    Suppose that $E \in \Dcatb(\pro^3,\CCC_0)$ is $\tilde\sigma_P$-stable. Let $E_1$ and $E_2$ be the intersection points of the line $L_{PE}$ connecting $P$ and $\tv(E)$ and the parabola $\varGamma$, where $\xi(E_1) \geq \xi(E_2)$. We take $\phi_Q(E_i) \coloneqq \frac{1}{\pi}\arg{\left(\ell_{Q}^-,\ell_{QE_i}^+\right)} \in (-1,1]$ for $i=1,2$. Let $F$ be a Harder--Narasimhan factor of $E$ with respect to another tilt stability condition $\tilde\sigma_{Q}$.
    \begin{enumerate}[dense]
        \item If $E$ lies in the heart $\Coh^P(\pro^3,\CCC_0) \coloneqq \Coh^{\xi(P)-\frac{5}{4}}(\pro^3,\CCC_0)$, then the phase $\phi_Q(F)$ is between $\phi_Q(E_1)$ and $\phi_Q(E_2)+1$;
        \item If $E[1] \in \Coh^P(\pro^3,\CCC_0)$, then the phase $\phi_Q(F)$ is between $\phi_Q(E_1)-1$ and $\phi_Q(E_2)$.
    \end{enumerate}
\end{lemma}

The next result computes the action of functor $(-\otimes_{\CCC_0} \CCC_1)$ on the phase of tilt stability conditions. It is an analogue of \cite[Lemma 5.2]{PY20} and also of Li's restriction-1 property in \cite{Li26} for tilt stability.

\begin{proposition}[][prop:tensor_heart]
    Fix the point $Q = \ab(0,\frac{1}{32})$, and let $\tilde\cP_{Q}$ be the slicing on $\Dcatb(\pro^3,\CCC_0)$ induced by the tilt stability $\ts_{0,\frac{1}{32}}$. We have 
    \begin{equation}\label{equ:tensor_phase}
        E \in \tilde\cP_{Q}\ab[\frac{1}{2},\, \frac{3}{2}] \implies  E \otimes_{\CCC_0} \CCC_1 \in \tilde\cP_{Q}\left(\frac{1}{2},\, \frac{5}{2}\right].
    \end{equation}
\end{proposition}

\begin{proof}
    Using \cref{tensor_stab_para} and changing the $(\beta,\alpha)$ coordinates into the $(\xi,\eta)$ coordinates. we note that $E$ is $\tilde\sigma_{0,\eta}$-semistable if and only if $E \otimes_{\CCC_0} \CCC_1$ is $\tilde\sigma_{\frac{1}{2},\ \eta + \frac{1}{8}}$-semistable. Let $\phi_M(E)$ denote the phase of $E$ with respect to the tilt stability condition $\ts_{\xi,\eta}$ represented by $M = (\xi,\eta)$. 
                                                                            
                                                        Let $B$ be the extension of all HN factors of $E$ with $\phi_Q \leq 1$. Then we have the triangle
    \begin{equation}\label{equ:tensor_AB_F_exten}
        \begin{tikzcd}[ampersand replacement=\&,cramped]
            {A[1]} \& E \& {B} \& {}
            \arrow[from=1-1, to=1-2]
            \arrow[from=1-2, to=1-3]
            \arrow["{+1}", from=1-3, to=1-4]
        \end{tikzcd}
    \end{equation}
    where $B \in \tilde\cP_{Q}\ab[\frac{1}{2},\, 1]$ and $A[1] \in \tilde\cP_{Q}\left(1,\, \frac{3}{2}\right]$. 

    \noindent\textbf{Step 1:} We claim that 
    \begin{equation}\label{equ:B_tensor_phase}
        B \otimes_{\CCC_0}\CCC_1 \in \tilde\cP_{Q}\left(\frac{1}{2},\, 2\right].
    \end{equation}
    By passing to a stable factor we may assume that $B$ is $\ts_{Q} = \ts_{0,\frac{1}{32}}$-stable. Then \cref{tensor_stab_para}  shows that $B \otimes_{\CCC_0} \CCC_1$ is $\ts_{P}$-stable, where $P = \ab(\frac{1}{2},\, \frac{5}{32})$. Let $\nu_0 \coloneqq \tilde\nu_Q(B) \geq 0$. Then $\tilde\nu_P(B \otimes_{\CCC_0} \CCC_1) = \nu_0 + \frac{1}{2}$, and the line $\ell_B$ connecting $P$ and $\tv(B \otimes_{\CCC_0} \CCC_1)$ has equation 
    \[ \ell_B\colon \quad \eta = \left( \nu_0 + \dfrac{1}{2} \right)\left( \xi - \dfrac{1}{2} \right) + \dfrac{5}{32}. \]
    Let $B_1$, $B_2$ be the intersection points of $\ell_B$ and the parabola $\varGamma$, with $\xi(B_1) > \xi(B_2)$. Then we have 
    \[\xi(B_1) = \nu_0 + \dfrac{1}{2} + \sqrt{\nu_0^2 + \frac{1}{16}} \geq \frac{1}{4}.\]
    So the point $B_1$ lies above $Q$. By \cref{lem:LZ19}, the phases of the stable factors of $B \otimes_{\CCC_0} \CCC_1$ with respect to $\tilde\sigma_Q$ is between $\phi_Q(B_1)$ and $\phi_Q(B_2) + 1$. This is shown in \cref{fig:B}. We have 
    \[ \frac{1}{2} < \phi_Q(B_1) \leq \phi_Q^-(B \otimes_{\CCC_0} \CCC_1)\leq\phi_Q^+(B \otimes_{\CCC_0} \CCC_1)\leq \phi_Q(B_2)+1 \leq 2. \]
    This proves \eqref{equ:B_tensor_phase}.
    \begin{figure}[h]
        \centering 
        \begin{tikzpicture}[
  line cap=round,
  line join=round,
  >=triangle 45,
  dot/.style={circle, fill=black, inner sep=0pt, minimum size=3pt},
  ptlabel/.style={font=\scriptsize, text=black},
  curve/.style={line width=0.5pt},
  ray/.style={line width=1pt, cyan!50!black},
  guide/.style={line width=0.5pt, dash pattern=on 1pt off 1pt},
  dottedray/.style={line width=1pt, dotted}
]
    \pgfmathsetmacro{\xmin}{-0.5}
  \pgfmathsetmacro{\xmax}{1}
  \pgfmathsetmacro{\ymin}{-0.05}
  \pgfmathsetmacro{\ymax}{0.4}

    \pgfmathsetmacro{\Qx}{0}
  \pgfmathsetmacro{\Qy}{0.03125}

  \pgfmathsetmacro{\BX}{0.4}
  \pgfmathsetmacro{\BY}{0.05}

    \pgfmathsetmacro{\CZeroX}{-1/4}
  \pgfmathsetmacro{\CZeroY}{1/32}
  \pgfmathsetmacro{\COneX}{1/4}
  \pgfmathsetmacro{\COneY}{1/32}

      \pgfmathsetmacro{\PX}{1/2}
  \pgfmathsetmacro{\PY}{\Qy+1/8}

      \pgfmathsetmacro{\GammaX}{\BX+1/2}
  \pgfmathsetmacro{\GammaY}{\BY+(1/2)*\BX+1/8}

    \pgfmathsetmacro{\mEllB}{(\PY-\GammaY)/(\PX-\GammaX)}
  \pgfmathsetmacro{\bEllB}{\PY-\mEllB*\PX}

      \pgfmathsetmacro{\discEllB}{\mEllB*\mEllB+2*\bEllB}
  \pgfmathsetmacro{\rootEllBLeft}{\mEllB-sqrt(\discEllB)}
  \pgfmathsetmacro{\rootEllBRight}{\mEllB+sqrt(\discEllB)}

  \pgfmathsetmacro{\BtwoX}{min(\rootEllBLeft,\rootEllBRight)}
  \pgfmathsetmacro{\BoneX}{max(\rootEllBLeft,\rootEllBRight)}
  \pgfmathsetmacro{\BtwoY}{0.5*\BtwoX*\BtwoX}
  \pgfmathsetmacro{\BoneY}{0.5*\BoneX*\BoneX}

    \pgfmathsetmacro{\mQBOne}{(\BoneY-\Qy)/(\BoneX-\Qx)}
  \pgfmathsetmacro{\mBTwoQ}{(\BtwoY-\Qy)/(\BtwoX-\Qx)}

    \pgfmathsetmacro{\rayQBOneRightY}{\Qy+\mQBOne*(\xmax-\Qx)}
  \pgfmathsetmacro{\rayBTwoQLeftY}{\Qy+\mBTwoQ*(\xmin-\Qx)}

    \pgfmathsetmacro{\rayQBOneEndX}{\xmax}
  \pgfmathsetmacro{\rayQBOneEndY}{\rayQBOneRightY}

    \pgfmathsetmacro{\BtwoLeftCrossX}{\mBTwoQ-sqrt(\mBTwoQ*\mBTwoQ+2*\Qy)}
  \pgfmathsetmacro{\BtwoLeftCrossY}{0.5*\BtwoLeftCrossX*\BtwoLeftCrossX}
  \pgfmathsetmacro{\QBOneRightCrossX}{\mQBOne+sqrt(\mQBOne*\mQBOne+2*\Qy)}
  \pgfmathsetmacro{\QBOneRightCrossY}{0.5*\QBOneRightCrossX*\QBOneRightCrossX}

    \pgfmathsetmacro{\parabolaTopRightX}{sqrt(2*\ymax)}

  \begin{axis}[
    set layers,
    axis on top,
    x=8cm,
    y=12cm,
    axis lines=middle,
    font=\scriptsize,
    xmin=\xmin,
    xmax=\xmax,
    ymin=\ymin,
    ymax=\ymax,
    xlabel={$\xi$},
    ylabel={$\eta$},
    xtick={-1/4,1/4,1/2,3/4},
    xticklabels={$-\frac{1}{4}$,$\frac{1}{4}$,$\frac{1}{2}$,$\frac{3}{4}$},
    ytick={1/32,1/8},
    yticklabels={,$\frac{1}{8}$},
    clip=false
  ]

    \node[ptlabel, anchor=south east] at (axis cs:\Qx,\Qy) {$\frac{1}{32}$};

            \begin{scope}[on layer=axis background]
      \clip (axis cs:\xmin,\ymin) rectangle (axis cs:\xmax,\ymax);

            \fill[cyan!18]
        plot[domain=\xmin:\BtwoLeftCrossX, samples=120]
          (axis cs:\x,{0.5*\x*\x})
        --
        plot[domain=\BtwoLeftCrossX:\xmin, samples=120]
          (axis cs:\x,{\Qy+\mBTwoQ*(\x-\Qx)})
        -- cycle;

            \fill[cyan!18]
        (axis cs:\QBOneRightCrossX,\QBOneRightCrossY)
        -- plot[domain=\QBOneRightCrossX:\parabolaTopRightX, samples=120]
          (axis cs:\x,{0.5*\x*\x})
        -- (axis cs:\rayQBOneEndX,\ymax)
        -- (axis cs:\rayQBOneEndX,\rayQBOneEndY)
        -- plot[domain=\rayQBOneEndX:\QBOneRightCrossX, samples=120]
          (axis cs:\x,{\Qy+\mQBOne*(\x-\Qx)})
        -- cycle;
    \end{scope}

        \begin{scope}
      \clip (axis cs:\xmin,\ymin) rectangle (axis cs:\xmax,\ymax);

      \addplot[curve, samples=500, domain=\xmin:\xmax] {x^2/2};

      \draw[ray]
        (axis cs:\Qx,\Qy) --
        (axis cs:\rayQBOneEndX,\rayQBOneEndY);

      \draw[ray]
        (axis cs:\xmin,\rayBTwoQLeftY) --
        (axis cs:\BtwoX,\BtwoY);

      \draw[guide]
        (axis cs:-0.25,0.03125) --
        (axis cs:\xmax,0.03125);

      \draw[dottedray]
        (axis cs:\BtwoX,\BtwoY) --
        (axis cs:\GammaX,\GammaY);
    \end{scope}

        \node[dot] at (axis cs:\CZeroX,\CZeroY) {};
    \node[dot] at (axis cs:\COneX,\COneY) {};
    \node[dot] at (axis cs:\Qx,\Qy) {};
    \node[dot] at (axis cs:\PX,\PY) {};
    \node[dot] at (axis cs:\BX,\BY) {};
    \node[dot] at (axis cs:\GammaX,\GammaY) {};
    \node[dot] at (axis cs:\BoneX,\BoneY) {};
    \node[dot] at (axis cs:\BtwoX,\BtwoY) {};

        \node[ptlabel, anchor=south east] at (axis cs:\CZeroX,0.036) {$\CCC_0$};
    \node[ptlabel, anchor=north west] at (axis cs:\COneX,0.036) {$\CCC_1$};
    \node[ptlabel, anchor=south west] at (axis cs:\Qx,0.04) {$Q$};
    \node[ptlabel, anchor=south east] at (axis cs:\PX,\PY) {$P$};
    \node[ptlabel, anchor=south west] at (axis cs:\BX,\BY) {$B$};
    \node[ptlabel, anchor=south east] at (axis cs:\BoneX,\BoneY) {$B_1$};
    \node[ptlabel, anchor=south] at (axis cs:\BtwoX,\BtwoY) {$B_2$};
    \node[ptlabel, anchor=south west] at (axis cs:\GammaX,\GammaY) {$B \otimes_{\CCC_0} \CCC_1$};

        \node[ptlabel, anchor=north west] at (axis cs:0.835,0.300) {\color{cyan!50!black}$\ell_{QB_1}^{+}$};
    \node[ptlabel, anchor=south] at (axis cs:-0.125,0.03) {\color{cyan!50!black}$\ell_{B_2Q}^{-}$};
    \node[ptlabel, anchor=south east] at (axis cs:0.865,0.37) {$\varGamma$};

  \end{axis}
\end{tikzpicture}
        \caption{The range of phase of the stable factors of $B \otimes_{\CCC_0} \CCC_1$ with respect to $\tilde\sigma_P$.}
        \label{fig:B}
    \end{figure}

    \noindent\textbf{Step 2:} We claim that 
    \begin{equation}\label{equ:A_tensor_phase}
        A[1] \otimes_{\CCC_0}\CCC_1 \in \tilde\cP_{Q}\left(1,\, \frac{5}{2}\right].
    \end{equation}
     Similarly we may assume that $A[1]$ is $\tilde\sigma_Q$-stable. Let $A_1$, $A_2$ be the intersection points of the line $\ell_A$ joining $P$ and $\tv(A[1] \otimes_{\CCC_0} \CCC_1)$ and the parabola $\varGamma$, with $\xi(A_1) > \xi(A_2)$. 
         Let $\nu_0 \coloneqq \tilde\nu_Q(A[1]) < 0$. By solving equations, we have the coordinate of $A_2$ given by 
        \[\xi(A_2) = \nu_0 + \dfrac{1}{2} - \sqrt{\nu_0^2 + \frac{1}{16}}; \qquad \eta(A_2) = \left( \nu_0 + \dfrac{1}{2} \right)\left( \nu_0 - \sqrt{\nu_0^2 + \frac{1}{16}} \right) + \frac{5}{32}.\]
        We have either $\xi(A_2) \leq 0$ or $\xi(A_2) > 0$ and $\eta(A_2) < \frac{1}{32}$ holds. In other words, the point $A_2$ lies either below $Q$ or to its left. By \cref{lem:LZ19}, the phases of the stable factors of $A[1] \otimes_{\CCC_0} \CCC_1$ with respect to $\tilde\sigma_Q$ is between $\phi_Q(A_1) + 1$ and $\phi_Q(A_2) + 2$.
        It follows that
        \[ 1 < \phi_Q(A_1) + 1 \leq \phi_Q^-(A[1] \otimes_{\CCC_0} \CCC_1)\leq\phi_Q^+(A[1] \otimes_{\CCC_0} \CCC_1)\leq \phi_Q(A_2)+2 \leq \frac{5}{2}. \]
        This proves \eqref{equ:A_tensor_phase}, and
        hence \eqref{equ:tensor_phase} holds. \qedhere
                        \begin{figure}[h]
            \centering 
            \begin{tikzpicture}[
  line cap=round,
  line join=round,
  >=triangle 45,
  dot/.style={circle, fill=black, inner sep=0pt, minimum size=3pt},
  ptlabel/.style={font=\scriptsize, text=black},
  curve/.style={line width=0.5pt},
  ray/.style={line width=1pt, cyan!50!black},
  guide/.style={line width=0.5pt, dash pattern=on 1pt off 1pt},
  dottedray/.style={line width=1pt, dotted}
]
    \pgfmathsetmacro{\xmin}{-0.7}
  \pgfmathsetmacro{\xmax}{0.7}
  \pgfmathsetmacro{\ymin}{-0.13}
  \pgfmathsetmacro{\ymax}{0.32}

    \pgfmathsetmacro{\Qx}{0}
  \pgfmathsetmacro{\Qy}{0.03125}

  \pgfmathsetmacro{\AX}{-0.55}
  \pgfmathsetmacro{\AY}{0.08}

    \pgfmathsetmacro{\CZeroX}{-1/4}
  \pgfmathsetmacro{\CZeroY}{1/32}
  \pgfmathsetmacro{\COneX}{1/4}
  \pgfmathsetmacro{\COneY}{1/32}

      \pgfmathsetmacro{\PX}{1/2}
  \pgfmathsetmacro{\PY}{\Qy+1/8}

      \pgfmathsetmacro{\ATensorX}{\AX+1/2}
  \pgfmathsetmacro{\ATensorY}{\AY+(1/2)*\AX+1/8}

    \pgfmathsetmacro{\mEllA}{(\PY-\ATensorY)/(\PX-\ATensorX)}
  \pgfmathsetmacro{\bEllA}{\PY-\mEllA*\PX}

      \pgfmathsetmacro{\discEllA}{\mEllA*\mEllA+2*\bEllA}
  \pgfmathsetmacro{\rootEllALeft}{\mEllA-sqrt(\discEllA)}
  \pgfmathsetmacro{\rootEllARight}{\mEllA+sqrt(\discEllA)}

  \pgfmathsetmacro{\ATwoX}{min(\rootEllALeft,\rootEllARight)}
  \pgfmathsetmacro{\AOneX}{max(\rootEllALeft,\rootEllARight)}
  \pgfmathsetmacro{\ATwoY}{0.5*\ATwoX*\ATwoX}
  \pgfmathsetmacro{\AOneY}{0.5*\AOneX*\AOneX}

    \pgfmathsetmacro{\mQATwo}{(\ATwoY-\Qy)/(\ATwoX-\Qx)}
  \pgfmathsetmacro{\mAOneQ}{(\AOneY-\Qy)/(\AOneX-\Qx)}

    \pgfmathsetmacro{\dxQATwo}{\ATwoX-\Qx}
  \pgfmathsetmacro{\dyQATwo}{\ATwoY-\Qy}

  \ifdim\dxQATwo pt>0pt
    \pgfmathsetmacro{\txQATwo}{(\xmax-\Qx)/\dxQATwo}
  \else
    \ifdim\dxQATwo pt<0pt
      \pgfmathsetmacro{\txQATwo}{(\xmin-\Qx)/\dxQATwo}
    \else
      \pgfmathsetmacro{\txQATwo}{1000000}
    \fi
  \fi

  \ifdim\dyQATwo pt>0pt
    \pgfmathsetmacro{\tyQATwo}{(\ymax-\Qy)/\dyQATwo}  \else
    \ifdim\dyQATwo pt<0pt
      \pgfmathsetmacro{\tyQATwo}{(\ymin-\Qy)/\dyQATwo}
    \else
      \pgfmathsetmacro{\tyQATwo}{1000000}
    \fi
  \fi

  \pgfmathsetmacro{\tQATwo}{min(\txQATwo,\tyQATwo)}
  \pgfmathsetmacro{\rayQATwoEndX}{\Qx+\tQATwo*\dxQATwo}
  \pgfmathsetmacro{\rayQATwoEndY}{\Qy+\tQATwo*\dyQATwo}

    \pgfmathsetmacro{\dxAOneMinus}{\Qx-\AOneX}
  \pgfmathsetmacro{\dyAOneMinus}{\Qy-\AOneY}

  \ifdim\dxAOneMinus pt>0pt
    \pgfmathsetmacro{\txAOneMinus}{(\xmax-\Qx)/\dxAOneMinus}
  \else
    \ifdim\dxAOneMinus pt<0pt
      \pgfmathsetmacro{\txAOneMinus}{(\xmin-\Qx)/\dxAOneMinus}
    \else
      \pgfmathsetmacro{\txAOneMinus}{1000000}
    \fi
  \fi

  \ifdim\dyAOneMinus pt>0pt
    \pgfmathsetmacro{\tyAOneMinus}{(\ymax-\Qy)/\dyAOneMinus}
  \else
    \ifdim\dyAOneMinus pt<0pt
      \pgfmathsetmacro{\tyAOneMinus}{(\ymin-\Qy)/\dyAOneMinus}
    \else
      \pgfmathsetmacro{\tyAOneMinus}{1000000}
    \fi
  \fi

  \pgfmathsetmacro{\tAOneMinus}{min(\txAOneMinus,\tyAOneMinus)}
  \pgfmathsetmacro{\rayAOneQLeftX}{\Qx+\tAOneMinus*\dxAOneMinus}
  \pgfmathsetmacro{\rayAOneQLeftY}{\Qy+\tAOneMinus*\dyAOneMinus}

    \pgfmathsetmacro{\AOneLeftCrossX}{\mAOneQ-sqrt(\mAOneQ*\mAOneQ+2*\Qy)}
  \pgfmathsetmacro{\AOneLeftCrossY}{0.5*\AOneLeftCrossX*\AOneLeftCrossX}
  \pgfmathsetmacro{\QATwoRightCrossX}{\mQATwo+sqrt(\mQATwo*\mQATwo+2*\Qy)}
  \pgfmathsetmacro{\QATwoRightCrossY}{0.5*\QATwoRightCrossX*\QATwoRightCrossX}

  \begin{axis}[
    set layers,
    axis on top,
    x=8cm,
    y=12cm,
    axis lines=middle,
    xmin=\xmin,
    xmax=\xmax,
    ymin=\ymin,
    ymax=\ymax,
    font=\scriptsize,
    xlabel={$\xi$},
    ylabel={$\eta$},
    xtick={-1/2,-1/4,1/4,1/2},
    xticklabels={$-\frac{1}{2}$,$-\frac{1}{4}$,$\frac{1}{4}$,$\frac{1}{2}$},
    ytick={1/32,1/8},
    yticklabels={,$\frac{1}{8}$},
    clip=false
  ]

    \node[ptlabel, anchor=south east] at (axis cs:\Qx,\Qy) {$\frac{1}{32}$};

        \begin{scope}[on layer=axis background]
      \clip (axis cs:\xmin,\ymin) rectangle (axis cs:\xmax,\ymax);
      \fill[cyan!18]
        (axis cs:\rayAOneQLeftX,\ymin)
        -- plot[domain=\rayAOneQLeftX:\AOneLeftCrossX, samples=80]
          (axis cs:\x,{\Qy+\mAOneQ*(\x-\Qx)})
        -- plot[domain=\AOneLeftCrossX:\QATwoRightCrossX, samples=140, smooth]
          (axis cs:\x,{0.5*\x*\x})
        -- plot[domain=\QATwoRightCrossX:\rayQATwoEndX, samples=100]
          (axis cs:\x,{\Qy+\mQATwo*(\x-\Qx)})
        -- (axis cs:\rayQATwoEndX,\ymin)
        -- cycle;
    \end{scope}

        \begin{scope}
      \clip (axis cs:\xmin,\ymin) rectangle (axis cs:\xmax,\ymax);

      \addplot[curve, samples=500, domain=\xmin:\xmax, smooth] {x^2/2};

      \draw[ray]
        (axis cs:\Qx,\Qy) --
        (axis cs:\rayQATwoEndX,\rayQATwoEndY);

      \draw[ray]
        (axis cs:\rayAOneQLeftX,\rayAOneQLeftY) --
        (axis cs:\AOneX,\AOneY);

      \draw[dottedray]
        (axis cs:\AOneX,\AOneY) --
        (axis cs:\ATensorX,\ATensorY);

      \draw[guide]
        (axis cs:\xmin,0.03125) --
        (axis cs:0.25,0.03125);
    \end{scope}

        \node[dot] at (axis cs:\CZeroX,\CZeroY) {};
    \node[dot] at (axis cs:\COneX,\COneY) {};
    \node[dot] at (axis cs:\Qx,\Qy) {};
    \node[dot] at (axis cs:\PX,\PY) {};
    \node[dot] at (axis cs:\AX,\AY) {};
    \node[dot] at (axis cs:\ATensorX,\ATensorY) {};
    \node[dot] at (axis cs:\ATwoX,\ATwoY) {};
    \node[dot] at (axis cs:\AOneX,\AOneY) {};

        \node[ptlabel, anchor=south east] at (axis cs:-0.25,0.036) {$\CCC_0$};
    \node[ptlabel, anchor=north west] at (axis cs:0.25,0.036) {$\CCC_1$};
    \node[ptlabel, anchor=south west] at (axis cs:0,0.04) {$Q$};
    \node[ptlabel, anchor=north east] at (axis cs:0.5,0.145) {$P$};
    \node[ptlabel, anchor=south] at (axis cs:\AX,\AY) {$A[1]$};
    \node[ptlabel, anchor=south] at (axis cs:-0.43,0.036) {$\ell_0$};
    \node[ptlabel, anchor=north east] at (axis cs:\ATensorX,\ATensorY) {$A[1] \otimes_{\CCC_0} \CCC_1$}; 
    \node[ptlabel, anchor=north] at (axis cs:\ATwoX,\ATwoY) {$A_2$};
    \node[ptlabel, anchor=south east] at (axis cs:0.66,0.22) {$A_1$};

        \node[ptlabel, anchor=south west] at (axis cs:-0.585,0.176) {$\varGamma$};
    \node[ptlabel, anchor=north west] at (axis cs:-0.535,-0.055) {\color{cyan!50!black}$\ell_{A_1Q}^{-}$};
    \node[ptlabel, anchor=west] at (axis cs:0.618,-0.044) {\color{cyan!50!black}$\ell_{QA_2}^{+}$};

  \end{axis}
\end{tikzpicture}
            \caption{The range of phase of the stable factors of $A[1] \otimes_{\CCC_0} \CCC_1$ with respect to $\tilde\sigma_P$. }
            \label{fig:A}
        \end{figure}

\end{proof}

\subsection{Proof of Serre invariance}\label{subsec:Serre_inv}

\begin{definition}[][def:Serre_inv]
    Let $\sigma = (\AAA,Z)$ be a stability condition on the triangulated category $\TTT$ with a Serre functor $\mathsf S$. The stability condition $\sigma$ is called Serre-invariant, if there exists $\tilde g \in \tilde{\operatorname{GL}_2^+}(\R)$ such that $\mathsf S \cdot \sigma = \sigma \cdot \tilde g$. 
    
    Explicitly, $\sigma$ is Serre-invariant if and only if we can find $\tilde g = (M,g) \in \tilde{\operatorname{GL}_2^+}(\R)$ such that $\mathcal{P}_{\sigma}(g(0),g(1)] = \mathsf S(\AAA)$ and $M^{-1}\circ Z = Z \circ \mathsf{S}^{-1}_*$. 
\end{definition}

\begin{theorem}[][thm:Serre_inv]
    The stability conditions $\sigma'_{\beta,\alpha}$ are Serre-invariant, and they lie in the same orbit with respect to the standard $\tilde{\operatorname{GL}^+_2}(\R)$-action.
\end{theorem}
\begin{remark}
    It would be interesting to investigate if there is a unique Serre-invariant $\tilde{\operatorname{GL}^+_2}(\R)$-orbit for the stability conditions on $\Ku(Y)$, as is the case for cubic 3-folds and very general cubic 4-folds (\cite{FP23}). Unfortunately, the techniques used in those cases may not directly apply here, as the fractional Calabi--Yau dimension of $\Ku(Y)$ for a cubic 5-fold is $7/3$, which does not satisfy the condition (C1) in \cite[Definition 3.1]{FP23}.
\end{remark}

The proof of \cref{thm:Serre_inv} will take up the rest of this subsection.

\begin{lemma}[][lem:heart_alpha]
    For any $(\xi,\eta) \in \tilde{V}$ (see \eqref{equ:para_ab_twisted}) and $0 < \eta' < \frac{1}{32}$, we have 
    \begin{enumerate}[dense]
        \item If $\xi = 0$, then $\tA_{0,\eta} = \tA_{0,\eta'}$;
        \item If $\xi < 0$, then $\tA_{\xi,\eta} \subset \ord{\tA_{0,\eta'}[-1],\ \tA_{0,\eta'}}$;
        \item If $\xi > 0$, then $\tA_{\xi,\eta} \subset \ord{\tA_{0,\eta'},\ \tA_{0,\eta'}[1]}$.
    \end{enumerate}
    In particular, the stability conditions in \cref{thm:main} lies in the same orbit of the $\widetilde{\GL_2^+}(\R)$-action.
        \end{lemma}
\begin{proof}
    The comparison for hearts is a standard result for tilt stability. See for example \cite[Proposition 3.6]{PY20} for a complete proof. For the last statement, it would follow from the fact the set $\left\{\tZ^0_{\xi,\eta}(\bar\kappa_1),\ \tZ^0_{\xi,\eta}(\bar\kappa_2)\right\}$ remains an oriented $\R$-basis of $\C$ as $(\xi,\eta)$ varies in $\tilde V$. Indeed, by the computation in \cref{prop:Knum_Ku_C0} we have 
    \begin{equation}
        \tZ^0_{\xi,\eta}(\bar\kappa_1) = 4;
                \qquad 
        \tZ^0_{\xi,\eta}(\bar\kappa_2) = 8\xi + 2 + \ii \ab(8\eta + \frac{7}{4}) \label{equ:Z0_basis}
    \end{equation}
    The determinant is equal to $7 + 32\eta$, which is strictly positive for $0 < \eta < \frac{1}{32}$.
\end{proof}

For the following lemma, we fix $\beta = -\frac{5}{4}$, or $\xi = 0$.  

\begin{lemma}[][lem:heart_ev]
    Let $E \in {\displaystyle{\bigcap}}\ab\{\tilde{\Coh}^0_{0,\eta} \ \big|\  0 < \eta < \frac{1}{32}\}$. Then there exists $\eta_1$ with $0 < \eta_1 < \frac{1}{32}$ such that for all $\eta$ satisfying $\eta_1 < \eta < \frac{1}{32}$, the evaluation morphism 
    \begin{equation}
        \ev\colon \Hom_{\CCC_0}(\CCC_0[2],E) \otimes \CCC_0[2] \to E \label{equ:ev_map}
    \end{equation}
    is a monomorphism in the heart $\tilde{\Coh}^0_{0,\eta}$.
                                        \end{lemma}
\begin{proof}
    If $\Hom_{\CCC_0}(\CCC_0[2],E) = 0$, the statement is vacuously true. Thus, we may assume that $\Hom_{\CCC_0}(\CCC_0[2],E) \ne 0$. By the wall-and-chamber structure of tilt stability, there exists $\varepsilon > 0$ and $\eta_1 \coloneqq \frac{1}{32} - \varepsilon$ such that the object $E$ has the Harder--Narasimhan filtration with respect to every $\tilde\sigma_{0,\eta}$ with $\eta_1 < \eta < \frac{1}{32}$:
    \[0 = E_0 \subset E_1 \subset \cdots \subset E_m \subset \cdots \subset E_n = E.\]
    For $1 \leq i \leq n$, let $F_i \coloneqq E_i/E_{i-1}$ be the associated Harder--Narasimhan factors of $E$. Denote by $\ell_0 \subset \R^2$ the ray $\{(\xi,\eta) \mid \xi \leq -\frac{1}{4},\ \eta = \frac{1}{32} \}$. Note that if $G$ is an object with $\tv(G) \in \ell_0$, then  
    \[  \lim_{\eta \to \frac{1}{32}}\tilde\nu^0_{0,\eta}(G) = +\infty, \]
    whereas $G$ with $\tv(G)$ above $\ell_0$ has finite limit tilt as $\eta \to \frac{1}{32}$. The object $\CCC_0[2]$ is $\tilde\sigma_{0,\eta}$-semistable with $\tv(\CCC_0[2]) \in \ell_0$. There exists some $\eta'$ such that $\tilde\nu_{0,\eta'}(\CCC_0[2]) > \tilde\nu_{0,\eta'}(F_i)$ for all $F_i$ for which $\tv(F_i) \notin \ell_0$. So the assumption that $\Hom_{\CCC_0}(\CCC_0[2],E) \ne 0$ implies that there is $1 \leq m \leq n$ such that $\tv(F_i) \in \ell_0$ for $1 \leq i \leq m$ and $\tv(F_i) \notin \ell_0$ for $m+1 \leq i \leq n$. This is shown in \cref{fig:HN_ell_0}.
    \begin{figure}[h]
        \centering
        \begin{tikzpicture}[line cap=round,line join=round,>=triangle 45,x=12cm,y=26cm]
    \begin{axis}[
            x=12cm,y=26cm,
            axis lines=middle,
            xmin=-0.63,
            xmax=0.4,
            ymin=-0.02,
            ymax=0.15,
            xlabel={$\xi$},
            ylabel={$\eta$},
            xtick={-1/4, 1/4},
            xticklabels={$-\frac{1}{4}$, $\frac{1}{4}$},
            ytick={0.018250781449224086, 1/32},
            yticklabels={{\scriptsize${{}_{\eta'}}$}, {\scriptsize$\frac{1}{32}$}}
            ]
        \clip(-0.63,-0.02) rectangle (0.4,0.15);
        \addplot [samples=500, line width=0.5pt, domain=-0.63:0.39] {x^2/2};
        \draw [line width=0.5pt,dotted] (-0.25,0.03125)-- (0.25,0.03125);
        \draw [line width=0.8pt,domain=-0.6046307961249218:-0.25] plot(\x,{(-0.012392855081443725-0*\x)/-0.3965713626061992});
        \draw [line width=0.5pt,dash pattern=on 1pt off 1pt,domain=-0.6046307961249218:0.29394573444796884] plot(\x,{(-0.004562695362306021--0.012999218550775914*\x)/-0.25});
        \begin{scriptsize}
            \draw [fill=black] (-0.25,0.03125) circle (1.5pt);
            \draw[color=black] (-0.25,1/64) node {$F_m,\CCC_0$};
            \draw [fill=black] (0.25,0.03125) circle (1.5pt);
            \draw[color=black] (0.25,1/64) node {$\CCC_1$};
            \draw [fill=black] (0,0.018250781449224086) circle (1.5pt);
                        \draw[color=black] (-0.58,1/64) node {$\ell_0$};
            \draw [fill=black] (-0.4842239732969621,0.03125) circle (1.5pt);
            \draw[color=black] (-0.4842239732969621,1/64) node {$F_1$};
            \draw [fill=black] (-0.4514392000750622,0.03125) circle (1.5pt);
            \draw[color=black] (-0.4514392000750622,1/64) node {$F_2$};
            \draw [fill=black] (-0.3498706869562353,0.03125) circle (1.5pt);
            \draw[color=black] (-0.3498706869562353,1/64) node {$F_{m-1}$};
            \draw [fill=black] (-0.46834687207524356,0.054111884351054476) circle (1.5pt);
            \draw[color=black] (-0.5234175055933598,0.064111884351054476) node {$F_{m+1}$};
            \draw [fill=black] (-0.4982557782412829-0.03,0.09361421324959729) circle (1.5pt);
            \draw[color=black] (-0.5447026111578547-0.03,0.10361421324959729) node {$F_{n-1}$};
            \draw [fill=black] (-0.4932237149657189,0.11044278697417491) circle (1.5pt);
            \draw[color=black] (-0.5180020797289359,0.12044278697417491) node {$F_n$};
            \draw[color=black] (-0.4914839504301042,0.076) node {$\ddots$};
            \draw[color=black] (-0.4,01/64) node {$\cdots$};
            \draw[color=black] (-0.5,0.14) node {$\varGamma$};
        \end{scriptsize}
    \end{axis}
\end{tikzpicture} 
        \caption[The Harder--Narasimhan factors of $E$ with respect to $\tilde\sigma_{0,\eta}$.]{The Harder--Narasimhan factors of $E$ with respect to $\tilde\sigma_{0,\eta}$. \\ 
        Note that $\tv(F_m) = \tv(\CCC_0)$ which follows from a simple tilt comparison.}
        \label{fig:HN_ell_0}
    \end{figure}
    
    Note that $\Hom_{\CCC_0}(\CCC_0[2],E/E_m) = 0$ and hence $\Hom_{\CCC_0}(\CCC_0[2],E) \cong \Hom_{\CCC_0}(\CCC_0[2],E_m)$. Consider the commutative diagram 
    \[\begin{tikzcd}[ampersand replacement=\&,cramped]
	{\Hom_{\CCC_0}(\CCC_0[2],E_m) \otimes \CCC_0[2]} \& {E_m} \\
	{\Hom_{\CCC_0}(\CCC_0[2],E) \otimes \CCC_0[2]} \& E
	\arrow["{\ev}", from=1-1, to=1-2]
	\arrow["\cong", from=1-1, to=2-1]
	\arrow[hook, from=1-2, to=2-2]
	\arrow["{\ev}", from=2-1, to=2-2]
    \end{tikzcd}\]
    Since $E_m \subset E$, the bottom morphism is injective if the top morphism is. Hence we may replace $E$ by $E_m$ in \eqref{equ:ev_map}. In particular, we may assume that $E \in \tilde\Coh^0_{0,\frac{1}{32}}$ is $\tilde\sigma_{0,\frac{1}{32}}$-semistable of phase $1$. 

    \medskip

    Let $G\coloneqq\cone(\ev)$, $G_1 \coloneqq \ker(\ev) = \mathcal H^{-1}(G)$, and $G_2 \coloneqq \coker(\ev) = \mathcal H^{0}(G)$, where the cohomology is taken with respect to the heart $\tilde\Coh^0_{0,\frac{1}{32}}$. We have the exact sequence in the heart $\tilde\Coh^0_{0,\frac{1}{32}}$:
    \begin{equation}
        \begin{tikzcd}[ampersand replacement=\&,cramped]
            0 \& G_1 \& \CCC_0^{\oplus k_0}[2] \& E \& G_2 \& 0.
            \arrow[from=1-1, to=1-2]
            \arrow[from=1-2, to=1-3]
            \arrow["\ev", from=1-3, to=1-4]
            \arrow[from=1-4, to=1-5]
            \arrow[from=1-5, to=1-6]
        \end{tikzcd}
    \end{equation}
    where $k_0 \coloneqq \dim\Hom_{\CCC_0}(\CCC_0[2],E)$. Since both $\CCC_0[2]$ and $E$ has the maximal phase $1$ in the heart, it follows that $G_1$ and $G_2$ are also $\tilde\sigma_{0,\frac{1}{32}}$-semistable of phase $1$. Using $G_1 \subset \CCC_0^{\oplus k_0}[2]$ we will show that $G_1 \cong \CCC_0^{\oplus r}[2]$ for some $r \in \Z_{\geq 0}$. Indeed, for $k_0 = 1$, we have a short exact sequence in the heart $\tilde\Coh^0_{0,\frac{1}{32}}$:
    \begin{equation}\begin{tikzcd}[ampersand replacement=\&,cramped]
    0 \& G_1 \& \CCC_0[2] \& K \& 0.
    \arrow[from=1-1, to=1-2]
    \arrow[from=1-2, to=1-3]
    \arrow[from=1-3, to=1-4]
    \arrow[from=1-4, to=1-5]
    \end{tikzcd}
    \end{equation}
    Since $G_1$ and $\CCC_0[2]$ have the same phase with respect to the heart $\tilde\Coh^0_{0,\frac{1}{32}}$, and $\CCC_0[2]$ has primitive character, it follows that $\ch_{\CCC_0, \leq 2}^{-\frac{5}{4}}(K) = 0$. Therefore the quotient object $K$ has zero-dimensional support. Since its phase is $1$, it is a zero-dimensional sheaf. Moreover, $\CCC_0[2]$ and $K$ are $\tilde\sigma_{0,\frac{1}{32}}$-semistable with $\phi_{\tilde\sigma_{0,\frac{1}{32}}}(\CCC_0[2]) > \phi_{\tilde\sigma_{0,\frac{1}{32}}}(K)$. Therefore we must have $K = 0$ and thus $G_1 \cong \CCC_0[2]$. For general $k_0 \geq 1$, inductively $G_1$ is the iterated extension of $\CCC_0[2]$ by itself. Furthermore, since $\CCC_0[2]$ is exceptional, we have $G_1 \cong \CCC_0^{\oplus r}[2]$ for some $r \in \Z_{\geq 0}$. 

    \medskip

    Now consider the long exact sequence 
    \[\begin{tikzcd}[ampersand replacement=\&,cramped,row sep=0, column sep = small]
	\& {0 = \Hom_{\CCC_0}(\CCC_0[2],E[-1])} \& {\Hom_{\CCC_0}(\CCC_0[2],G[-1])} \& {\Hom_{\CCC_0}(\CCC_0[2],\CCC_0^{\oplus k_0}[2])} \\
	{} \& {\Hom_{\CCC_0}(\CCC_0[2],E)} \& {\Hom_{\CCC_0}(\CCC_0[2],G)=0.}
	\arrow[from=1-2, to=1-3]
	\arrow[from=1-3, to=1-4]
	\arrow["\sim", from=2-1, to=2-2]
	\arrow[from=2-2, to=2-3]
        \end{tikzcd}\]
    It follows that $\Hom_{\CCC_0}(\CCC_0[2],G[-1]) = 0$. Applying $\Hom_{\CCC_0}(\CCC_0[3],-)$ to the sequence $0 \to G_1[1] \to G \to G_2 \to 0$, we obtain that 
    \[ \Hom_{\CCC_0}(\CCC_0[2],\CCC_0^{\oplus r}[2]) = \Hom_{\CCC_0}(\CCC_0[2],G_1) = \Hom_{\CCC_0}(\CCC_0[2],G[-1]) = 0.\]
    We must have $r = 0$ and thus $G_1 = 0$. This concludes the proof that \eqref{equ:ev_map} is injective.
\end{proof}

\begin{lemma}[][lem:rot_func_tilt]
    For $0 < \eta < \frac{1}{32}$, the category $\lmut{\CCC_0}\tA_{0,\eta} \otimes \CCC_1$ is a tilt of $\tA_{0,\eta}$.
\end{lemma}
\begin{proof}
    Let $E$ be an object of $\tA_{0,\eta} \coloneqq \tilde{\Coh}^0_{0,\eta} \cap \Ku(\pro^3,\CCC_0)$. Note that $\lmut{\CCC_0}E$ fits into the distinguished triangle:
    \[\begin{tikzcd}[ampersand replacement=\&,cramped]
            {\displaystyle\bigoplus_{k \in \Z} \Hom_{\CCC_0}(\CCC_0,E[k]) \otimes \CCC_0[-k]} \& E \& {\lmut{\CCC_0}E} \& {}
            \arrow[from=1-1, to=1-2]
            \arrow[from=1-2, to=1-3]
            \arrow["{+1}", from=1-3, to=1-4]
        \end{tikzcd}\]
    Since $\Dcatb(\pro^3,\CCC_0)$ has the Serre functor $\mathsf S_{\CCC_0} = - \otimes_{\CCC_0} \CCC_{-2}[3]$, and $\CCC_0[2], \CCC_{-2}[2] \in \tilde{\Coh}^0_{0,\eta}$, we have
    \[\Hom(\CCC_0[2], E[k]) = \Hom(E, \mathsf S_{\CCC_0}(\CCC_0)[-k+2]) = \Hom(E, \CCC_{-2}[-k+5]) = 0,\]
    for $k < 0$ or $-k+5<2$. Then we have the distinguished triangle:
    \[\begin{tikzcd}[ampersand replacement=\&,cramped]
            {\CCC_0^{\oplus k_0}[2] \oplus \CCC_0^{\oplus k_1}[1] \oplus \CCC_0^{\oplus k_2} \oplus \CCC_0^{\oplus k_3}[-1]} \& E \& {\lmut{\CCC_0}E} \& {}
            \arrow[from=1-1, to=1-2]
            \arrow[from=1-2, to=1-3]
            \arrow["{+1}", from=1-3, to=1-4]
        \end{tikzcd}\]
    Note that $\CCC_0[2]$ a stable object with respect to $\ts'_{0,\eta}$. By \cref{lem:heart_ev}, there exists $\varepsilon_1 > 0$ such that $F \coloneqq \cone\left( \CCC_0[2]^{\oplus k_0} \to E \right) \in \tilde{\Coh}^0_{0,\eta}$ for every $\frac{1}{32}-\varepsilon_1 < \eta < \frac{1}{32}$. By the octahedral axiom we have a commutative diagram, in which each row and each column are distinguished triangles:
    \[\begin{tikzcd}[ampersand replacement=\&,cramped]
            {\CCC_0^{\oplus k_0}[2]} \& E \& F \\
            {\CCC_0^{\oplus k_0}[2] \oplus \CCC_0^{\oplus k_1}[1] \oplus \CCC_0^{\oplus k_2} \oplus \CCC_0^{\oplus k_3}[-1]} \& E \& {\lmut{\CCC_0}E} \\
            {\CCC_0^{\oplus k_1}[1] \oplus \CCC_0^{\oplus k_2} \oplus \CCC_0^{\oplus k_3}[-1]} \& 0 \& {\CCC_0^{\oplus k_1}[2] \oplus \CCC_0^{\oplus k_2}[1] \oplus \CCC_0^{\oplus k_3}}
            \arrow["f", from=1-1, to=1-2]
            \arrow[from=1-1, to=2-1]
            \arrow[from=1-2, to=1-3]
            \arrow["\id", from=1-2, to=2-2]
            \arrow[from=1-3, to=2-3]
            \arrow[from=2-1, to=2-2]
            \arrow[from=2-1, to=3-1]
            \arrow[from=2-2, to=2-3]
            \arrow[from=2-2, to=3-2]
            \arrow[from=2-3, to=3-3]
            \arrow[from=3-1, to=3-2]
            \arrow[from=3-2, to=3-3]
        \end{tikzcd}\]
    Tensoring the rightmost column with $\CCC_1$, we obtain the triangle:
    \begin{equation}
        \begin{tikzcd}[ampersand replacement=\&,cramped]
        	{F \otimes_{\CCC_0} \CCC_1} \& {\lmut{\CCC_0}E  \otimes_{\CCC_0} \CCC_1} \& {\CCC_1^{\oplus k_1}[2] \oplus \CCC_1^{\oplus k_2}[1] \oplus \CCC_1^{\oplus k_3}} \& {}
        	\arrow[from=1-1, to=1-2]
        	\arrow[from=1-2, to=1-3]
        	\arrow["{+1}", from=1-3, to=1-4]
        \end{tikzcd}
    \end{equation}
    Since $F \in \tilde{\Coh}^0_{0,\eta}$ for any $\frac{1}{32} - \varepsilon_1 < \eta < \frac{1}{32}$, by deforming $\eta \to \frac{1}{32}$, we observe that $F \in \tilde{\cP}_Q\ab[\frac{1}{2},\, \frac{3}{2}]$. By \cref{prop:tensor_heart}, we have 
    \begin{equation}
        F \otimes_{\CCC_0} \CCC_1 \in \tilde{\cP}_Q\left(\frac{1}{2},\, \frac{5}{2}\right].
    \end{equation}
    On the other hand, since $\CCC_1^{\oplus k_1}[2] \oplus \CCC_1^{\oplus k_2}[1] \oplus \CCC_1^{\oplus k_3} \in  \tilde{\cP}_Q\left[\frac{1}{2},\, \frac{5}{2}\right]$, we have that 
    \begin{equation}
        \lmut{\CCC_0}E  \otimes_{\CCC_0} \CCC_1 \in \tilde{\cP}_Q\left[\frac{1}{2},\, \frac{5}{2}\right].
    \end{equation}

    For $\varepsilon > 0$ sufficiently small, let $Q_{\varepsilon} = \ab(0,\frac{1}{32}-\varepsilon)$ and let $\tilde{\cP}_{Q_{\varepsilon}}$ be the associated slicing on $\Dcatb(\pro^3,\CCC_0)$ induced by the tilt stability $\ts_{0,\frac{1}{32}-\varepsilon}$. Then there exists $\delta > 0$ such that 
    \begin{equation}
        \lmut{\CCC_0}E  \otimes_{\CCC_0} \CCC_1 \in \tilde{\cP}_{Q_{\varepsilon}}\left(\frac{1}{2},\, \frac{5}{2}+\delta\right).
    \end{equation}
    Moreover, $\delta \to 0$ as $\varepsilon \to 0$. Now using that $\lmut{\CCC_0}E  \otimes_{\CCC_0} \CCC_1 \in \Ku(\pro^3,\CCC_0)$ and the fact that the t-structure on $\Dcatb(\pro^3,\CCC_0)$ restricts to $\Ku(\pro^3,\CCC_0)$, we have 
    \begin{equation}
        \lmut{\CCC_0}E  \otimes_{\CCC_0} \CCC_1 \in \ord{\tA_{0,\frac{1}{32}-\varepsilon},\, \tA_{0,\frac{1}{32}-\varepsilon}[1],\, \tA_{0,\frac{1}{32}-\varepsilon}[2]}.
    \end{equation}
    Let $T \in \Ku(\pro^3,\CCC_0)$ be the cohomological object $\cH^{-2}(\lmut{\CCC_0}E  \otimes_{\CCC_0} \CCC_1)$ with respect to the heart $\tA_{0,\frac{1}{32}-\varepsilon}$. Note that $T$ is independent of $\varepsilon$ by \cref{lem:heart_alpha}. Moreover, we have $T \in \tilde{\cP}_{Q_{\varepsilon}}\left[\frac{5}{2},\frac{5}{2}+\delta\right)$. Taking the limit $\varepsilon \to 0$, we have that $T \in \tilde{\cP}_{Q}\ab(\frac{5}{2})$. By the geometry on the $(\xi,\eta)$-plane, this implies
    \begin{equation}\label{equ:phase_2.5_boundary}
        \ch_{\CCC_0,2}^{-\frac{5}{4}}(T) = \frac{1}{32}\ch_{\CCC_0,0}^{-\frac{5}{4}}(T).
    \end{equation}
    Meanwhile, since $T \in \Ku(\pro^3,\CCC_0)$, we can write $[T] = a\bar\kappa_1 + b\bar\kappa_2$. From \cref{prop:Knum_Ku_C0} we have 
    \begin{equation}
        \ch_{\CCC_0,\leq 2}^{-\frac{5}{4}}(T) = \ab(-8b,\, 4a+2b,\, \frac{7}{4}b). 
    \end{equation}
    Comparing with \eqref{equ:phase_2.5_boundary} we have $b = 0$. If $a \ne 0$, then $T$ has phase $\frac{5}{2}$ with respect to any $Q_{\varepsilon}$, which contradicts the construction that $T \in \tA_{0,\frac{1}{32}-\varepsilon}[2]$. Hence $ \ch_{\CCC_0,\leq 2}^{-\frac{5}{4}}(T) = 0$. In particular $T$ has 0-dimensional support. But $\Ku(\pro^3,\CCC_0) \cap  \Coh(\pro^3,\CCC_0)_0 = 0$. We must have $T = 0$ and hence 
    \[
    \lmut{\CCC_0}E  \otimes_{\CCC_0} \CCC_1 \in \ord{\tA_{0,\frac{1}{32}-\varepsilon},\, \tA_{0,\frac{1}{32}-\varepsilon}[1]}. \qedhere
    \]
                                                                                                                                                                        
    \end{proof}
\medskip

\begin{proof}[Proof of \cref{thm:Serre_inv}.] 
    Change to the $(\xi,\eta)$-coordinates as before. By \cref{lem:heart_alpha}, $\ts'_{\xi,\eta}$ lies in the same orbit for any $(\xi,\eta) \in \tilde{V}$. Therefore it suffices to prove that $\ts'_{\xi,\eta}$ is Serre-invariant for $\xi = 0$. By \cref{cor:Serre}, the Serre functor of $\Ku(Y) \simeq \Ku(\pro^3,\CCC_0)$ satisfies $\mathsf S_{\Ku(\pro^3,\CCC_0)}^{-1} = \mathsf{O}_{\Ku(\pro^3,\CCC_0)}^{2}[-3]$. So it suffices to find $\tilde g \in \tilde{\operatorname{GL}_2^+}(\R)$ such that 
    \[ \mathsf{O}_{\Ku(\pro^3,\CCC_0)}\cdot \ts'_{0,\eta} = \ts'_{0,\eta} \cdot \tilde g. \]
    By \cref{prop:Knum_Ku_C0}, we know that $\Knum(\Ku(\pro^3,\CCC_0))$ is generated by the classes $\bar\kappa_1$ and $\bar\kappa_2$, with $\mathsf{O}_{\Ku(\pro^3,\CCC_0)}^{-1}$ acting by  
    \[ \mathsf{O}_{\Ku(\pro^3,\CCC_0)}^{-1}(\bar\kappa_1) = \bar\kappa_1 - \bar\kappa_2; \qquad \mathsf{O}_{\Ku(\pro^3,\CCC_0)}^{-1}(\bar\kappa_2) = \bar\kappa_1. \]
    Therefore we can find $M \in \operatorname{GL}_2^+(\R)$ such that $M^{-1}\circ \tZ^0_{0,\eta} = \tZ^0_{0,\eta} \circ \mathsf{O}_{\Ku(\pro^3,\CCC_0)*}^{-1}$. 
                                    We can lift $M$ to $\tilde g = (M,f) \in \tilde{\operatorname{GL}_2^+}(\R)$ with $f(0,1] \subset (0,2]$, so that the heart of $\ts'_{0,\eta}$ is a tilt of $\tA_{0,\eta}$. By \cref{lem:rot_func_tilt}, the heart $\mathsf{O}_{\Ku(\pro^3,\CCC_0)}(\tA_{0,\eta})$ is also a tilt of $\tA_{0,\eta}$. Now the stability conditions $\mathsf{O}_{\Ku(\pro^3,\CCC_0)}\cdot \ts'_{0,\eta}$ and $\ts'_{0,\eta} \cdot \tilde g$ have the same central charge, and their hearts are all contained in the tilt $\ord{\tA_{0,\eta},\ \tA_{0,\eta}[1]}$. By \cite[Lemma 8.11]{bay16}, we have $\mathsf{O}_{\Ku(\pro^3,\CCC_0)}\cdot \ts'_{0,\eta} = \ts'_{0,\eta} \cdot \tilde g$.
\end{proof}

\subsection{A Gepner-type stability condition}

\begin{corollary}[Global dimension of the Serre-invariant stability condition][cor:gl_dim]
    There is some $\tilde g \in \tilde{\operatorname{GL}^+_2}(\R)$ acting on the stability condition $\ts'_{0,\eta}$ such that $\sigma_Y \coloneqq \ts'_{0,\eta} \cdot \tilde g$ has the following property. For any $E \in \mathcal P_{\sigma_Y}(\phi)$, it holds that $\mathsf S_{\Ku(Y)}E \in \mathcal P_{\sigma_Y}\left( \phi + \frac{7}{3} \right)$. In particular, the global dimension of the stability condition $\sigma_Y$ is given by
    \begin{equation}
        \operatorname{gl.dim}(\sigma_Y) \coloneqq \sup\left\{ \phi_2 - \phi_1 \mid \exists\, E_1 \in \mathcal P_{\sigma_Y}(\phi_1),\ \exists\, E_2 \in \mathcal P_{\sigma_Y}(\phi_2),\ \Hom(E_1,E_2) \ne 0 \right\} = \dfrac{7}{3}. \label{def:gl_dim}
    \end{equation}
\end{corollary}
\begin{proof}
    By \cref{prop:Knum_Ku_C0}, the lattice $\Ku(Y) \cong \Ku(\pro^3,\CCC_0)$ is spanned by the classes $\bar \kappa_1$ and $\bar \kappa_2$. Their central charges are given in the equation \eqref{equ:Z0_basis}. Hence the rescaling matrix   
    \[M = \begin{pmatrix}
        4 & 0 \\ 0 & \frac{32\eta + 7}{2\sqrt{3}}
    \end{pmatrix} \in \GL_2^+(\R)\]
    acts on $\tZ^0_{0,\eta}$ by $Z'' \coloneqq M^{-1}\circ \tZ^0_{0,\eta}$ which satisfies that $Z''(\bar \kappa_1) = 1$ and $Z''(\bar \kappa_2) = \e^{\pi\ii/3}$. We can lift $M$ to $\tilde g = (M,g) \in \tilde{\operatorname{GL}_2^+}(\R)$ acting on $\ts'_{0,\eta}$ and denote $\sigma_Y \coloneqq \ts'_{0,\eta} \cdot \tilde g$. For any $\sigma_Y$-semistable $E \in \Ku(Y)$ with character $a\kappa_1 + b\kappa_2$, we have 
    \[ Z''(\mathsf S_{\Ku(Y)}E) = aZ''(\bar\kappa_2) + bZ''(\bar\kappa_2 - \bar\kappa_1) = \e^{\pi\ii/3}aZ''(\bar\kappa_1) + \e^{\pi\ii/3}bZ''(\bar\kappa_2) = \e^{\pi\ii/3}Z''(E).\]
    Hence the phase satisfies that $\phi_{\sigma_Y}(\mathsf S_{\Ku(Y)}E) - \phi_{\sigma_Y}(E) = \frac{1}{3} + 2k(E)$ for some $k(E) \in \Z$. Combining with the fact that $\mathsf S_{\Ku(Y)}^3 = [7]$, we have 
    \begin{equation}
        k(E) + k(\mathsf S_{\Ku(Y)}E) + k(\mathsf S_{\Ku(Y)}^2E) = 3. \label{equ:k+k+k}
    \end{equation} 
    Now, since $\sigma_Y$ is Serre-invariant, $\mathsf S_{\Ku(Y)} \cdot \sigma_Y = \sigma_Y\cdot \tilde g'$ for some $\tilde g' = (M',g') \in \tilde{\operatorname{GL}_2^+}(\R)$. We have 
    \begin{equation}
        \phi_{\sigma_Y}(\mathsf S_{\Ku(Y)}E) = g'(\phi_{\sigma_Y}(E)) = \phi_{\sigma_Y}(E) + \frac{1}{3} + 2k(E).
    \end{equation}
    Since $g'$ is an increasing function, we have 
    \begin{align*}
        7 + \phi_{\sigma_Y}(E) = \phi_{\sigma_Y}(\mathsf S_{\Ku(Y)}^3E) &= g'\circ g' \circ g'(\phi_{\sigma_Y}(E)) \geq g' \circ g'(\phi_{\sigma_Y}(E)) + 2k(E) \\ 
        &\geq g'(\phi_{\sigma_Y}(E)) + 4k(E) = \phi_{\sigma_Y}(E) + 6k(E) + \dfrac{1}{3}.
    \end{align*}
    Hence $k(E)\leq 1$ for all $\sigma_Y$-semistable $E$. Since $\mathsf S_{\Ku(Y)}E$ and $\mathsf S_{\Ku(Y)}^2E$ are also $\sigma_Y$-semistable, combining with equation \eqref{equ:k+k+k} we obtain that $k(E) = 1$ for all $\sigma_Y$-semistable $E$. We deduce that $\phi_{\sigma_Y}(\mathsf S_{\Ku(Y)}E) - \phi_{\sigma_Y}(E) = \frac{7}{3}$. 
    
    For the global dimension, for $E_1 \in \mathcal P_{\sigma_Y}(\phi_1)$ and $E_2 \in \mathcal P_{\sigma_Y}(\phi_2)$, if $\phi_2 - \phi_1 > \frac{7}{3}$, then $\phi_{\sigma_Y}(E_2) > \phi_{\sigma_Y}(\mathsf S_{\Ku(Y)}E_1) = \phi_{\sigma_Y}(E_1) + \frac{7}{3}$ and hence $\Hom(E_1,E_2) \cong \Hom(E_2,\mathsf S_{\Ku(Y)}E_1)^\dual = 0$. Therefore $\operatorname{gl.dim}(\sigma_Y) \leq \frac{7}{3}$. The equality holds because for any $\sigma_Y$-stable $E_1 \in \mathcal P_{\sigma_Y}(\phi)$, let $E_2 \coloneqq \mathsf S_{\Ku(Y)}E_1 \in \mathcal P_{\sigma_Y}(\phi + \frac{7}{3})$ and we have $\Hom(E_1,E_2) \cong \Hom(E_1,E_1)^\dual \cong \C$.
\end{proof}

\begin{figure}[h]
    \centering
    \begin{tikzpicture}[
    line cap=round,
    line join=round,
    >=latex,
    x=2.8cm,
    y=2.8cm,
    point/.style={circle, fill=black, inner sep=1.55pt},
    axis/.style={line width=0.55pt, black!70},
    boundary/.style={line width=1.35pt, cyan!50!black}
]
    \def\rt{0.8660254037844386}

    \coordinate (O) at (0,0);
    \coordinate (kone) at (1,0);
    \coordinate (ktwo) at (0.5,\rt);
    \coordinate (ktwominuskone) at (-0.5,\rt);
    \coordinate (mkone) at (-1,0);
    \coordinate (mktwo) at (-0.5,-\rt);
    \coordinate (koneminusktwo) at (0.5,-\rt);

        \fill[cyan!18] (-1.35,0) rectangle (1.35,1.25);

            \draw[axis,->] (0,-1.08) -- (0,1.28) node[above] {$\alpha$};
    \node[below right] at (1.35,0) {$\beta$};

        \draw[boundary] (-1.35,0) -- (O);
    \draw[boundary,densely dashed] (O) -- (1.35,0);

        \node[cyan!50!black] at (1,1.12) {$\tA_{0,\eta}$};

        \node[point] at (O) {};
    \node[below left=2pt] at (O) {\scriptsize$0$};

    \node[point] at (kone) {};
    \node[below=4pt] at (kone) {$\bar\kappa_1$};

    \node[point] at (ktwo) {};
    \node[above right=3pt] at (ktwo) {$\bar\kappa_2$};

    \node[point] at (ktwominuskone) {};
    \node[above left=3pt] at (ktwominuskone) {$\bar\kappa_2-\bar\kappa_1$};

    \node[point] at (mkone) {};
    \node[below=4pt] at (mkone) {$-\bar\kappa_1$};

    \node[point] at (mktwo) {};
    \node[below left=3pt] at (mktwo) {$-\bar\kappa_2$};

    \node[point] at (koneminusktwo) {};
    \node[below right=3pt] at (koneminusktwo) {$\bar\kappa_1-\bar\kappa_2$};

        \def\ticklen{0.035}

        \foreach \x/\lab in {-1/, -0.5/{-\frac12}, 0.5/{\frac12}, 1/} {
        \draw[axis] (\x,-\ticklen) -- (\x,\ticklen);
        \ifx\lab\empty\else
            \node[below=3pt] at (\x,0) {\scriptsize$\lab$};
        \fi
    }

        \foreach \y/\lab in {-1/{-1}, -0.5/{-\frac12}, 0.5/{\frac12}, 1/{1}} {
        \draw[axis] (-\ticklen,\y) -- (\ticklen,\y);
        \node[left=3pt] at (0,\y) {\scriptsize$\lab$};
    }
\end{tikzpicture}
    \caption[The central charges of $\Knum(\Ku(\pro^3,\CCC_0))$ under the stability condition $\sigma_Y$.]{The central charges of the lattice points in $\Knum(\Ku(\pro^3,\CCC_0))$ under the Gepner-type stability condition $\sigma_Y$.}
    \label{fig:hexa_central_charge}
\end{figure}

As an application of \cref{cor:gl_dim}, the stability condition $\sigma_Y$ confirms the existence of the Gepner-type stability condition on the Kuznetsov component of cubic 5-folds, as a special case of \cite[Conjecture 1.1]{Tod13} by Toda. Indeed, since $\mathsf{O}_{\Ku(Y)} = \mathsf{S}^{-1}_{\Ku(Y)}[3]$, we have that $\mathsf O_{\Ku(Y)}E \in \mathcal P_{\sigma_Y}\left( \phi + \frac{2}{3} \right)$ for any $E \in \mathcal P_{\sigma_Y}(\phi)$. 

\begin{definition}
    Let $\sigma$ be a stability condition on the triangulated category $\DDD$. If there is an auto-equivalence $\varGamma\in \Aut(\DDD)$ and a rational number $\lambda \in \Q$ such that $\varGamma \cdot \sigma = \sigma \cdot \lambda$, 
    then $\sigma$ is called of \textbf{Gepner-type} with respect to the pair $(\varGamma,\lambda)$.
\end{definition}

\begin{corollary}[][]
    \cite[Conjecture 1.1]{Tod13} holds for $n = 7$ and $d = 3$. In particular, the stability condition $\sigma_Y$ on $\Ku(Y)$ is of Gepner-type with respect to the pair $(\mathsf{O}_{\Ku(Y)},\frac{2}{3})$.
\end{corollary}

\section{The moduli of planes}\label{sec:moduli}
In this section we study moduli spaces of objects stable with respect to the Serre-invariant stability conditions constructed in \cref{sec:stab}. For any object $E \in \Ku(Y)$, its $\sigma'_{\beta,\alpha}$-(semi)stability is independent of the parameters $(\beta,\alpha)$ and of the blown-up plane $\varPi_0$; see \cref{lem:heart_alpha,rmk:Pi_indep}. The same $\widetilde{\GL_2^+}(\R)$-orbit contains the Gepner-type stability condition $\sigma_Y$ from \cref{cor:gl_dim}. Henceforth, whenever no confusion can arise, we call $\sigma_Y$ \emph{the} stability condition on $\Ku(Y)$ and simply speak of $\sigma_Y$-(semi)stability. We continue to use the $(\xi,\eta)$-coordinates introduced in \cref{subsec:tensor} for wall-crossing computations.

\begin{definition}
    Fix $v \in \Knum(\Ku(Y))$. 
        Let $\Modss(\Ku(Y),v)$ be the good moduli space parametrizing S-equivalence classes of $\sigma_Y$-semistable objects in $\Ku(Y)$ with numerical character $v$; it is an algebraic space proper over $\C$ by \cite[Theorem 21.24]{BLMNPS}. Let $\Mods(\Ku(Y),v) \subset \Modss(\Ku(Y),v)$ be the open subspace parametrizing $\sigma_Y$-stable objects. If $v$ is primitive, then every semistable object of class $v$ is stable, so $\Mods(\Ku(Y),v)=\Modss(\Ku(Y),v)$.
\end{definition}

\subsection[Stability of the projected ideal sheaf of a plane]{Stability of the projected ideal sheaf of a plane}\label{subsec:P_stab}

In this subsection we will show the non-emptiness of a particular moduli space on $\Ku(Y)$ by showing that the object $\cF_{\varPi}$ defined in \cref{def:Ku_objs} is $\sigma_Y$-stable. By \cref{rmk:Pi_indep}, the $\sigma_Y$-stability of $\cF_{\varPi}$ is independent of the choice of $\varPi_0$, so we may assume that $\varPi \cap \varPi_0 = \varnothing$ by \cref{lem:disjoint_planes}. From \eqref{equ:P_Pi_Psi} we have $\varPsi\sigma^*\cF_{\varPi} = \pi_*\ab(\EEE(H+h)|_{\sigma^{-1}\varPi})$.

Fix $\beta = -\frac{5}{4}$, equivalently $\xi = 0$. We identify the twisted modified Chern characters, truncated to the second term, $\Chern{\CCC_0, \leq 2}^{-\frac{5}{4}}$, with vectors in $\Q^3$ via $a + bh + ch^2 + \mathcal{O}(h^3) \longmapsto (a,b,c)$. Then by \cref{prop:Knum_Ku_C0}, the object $\pi_*(\EEE(H+h)|_{\sigma^{-1}\varPi}) \in \Dcatb(\pro^3,\CCC_0)$ has the character
\[ \Chern{\CCC_0, \leq 2}^{-\frac{5}{4}}\ab(\pi_*\ab(\EEE(H+h)|_{\sigma^{-1}\varPi})) = (0,4,0). \]

\begin{proposition}[][prop:P_Pi_stable]
    The object $E\coloneqq\pi_*\ab(\EEE(H+h)|_{\sigma^{-1}\varPi}) \in \Dcatb(\pro^3,\CCC_0)$ lies in the tilted heart $\Coh^{-\frac{5}{4}}(\pro^3,\CCC_0)$ and is $\ts_{0,\eta}$-stable for any $\eta > 0$.
\end{proposition}
    Since $E$ is a torsion sheaf, it has infinite slope and hence is in $\Coh^{-\frac{5}{4}}(\pro^3,\CCC_0)$. We employ the standard wall-crossing technique for tilt stability. The first step is the large volume limit:
    \begin{lemma}[Large-volume stability][]
        The sheaf $E \in \Dcatb(\pro^3,\CCC_0)$ is Gieseker stable. In particular, it is $\ts_{0,\eta}$-stable for $\eta \gg 0$.
    \end{lemma}
    \begin{proof}
        Suppose that $E$ is not Gieseker stable. Then it has a destabilizing subsheaf $F$. Since $E$ is pure and locally free of rank $4$ on its support, $F$ has the same support and generic rank $0<b<4$. Thus $\Chern{\CCC_0, \leq 2}^{-\frac{5}{4}}(F) = (0,b,c)$ for some $c \in \Q$. Write $[F] = \sum_{i=-1}^2 c_i[\CCC_i]$ in $\Knum(\pro^3,\CCC_0)$, with $c_i \in \Z$. Then
        \[8\sum_{i=-1}^2 c_i = \Chern{\CCC_0, 0}^{-\frac{5}{4}}(F) = 0; \qquad -2\sum_{i=-1}^2 c_i + 4\sum_{i=-1}^2 ic_i = \Chern{\CCC_0, 1}^{-\frac{5}{4}}(F) = b.\]
        We find that $b$ is divisible by $4$. This is an immediate contradiction.
    \end{proof}
    Next we investigate the numerical walls for the object $E$. We will show that there is at most one wall, and the stability of $E$ is unchanged across the wall, which will finish the proof.
    \begin{lemma}[][lem:P_Pi_1st_wall]
        Suppose that $E \in \Coh^{-\frac{5}{4}}(\pro^3,\CCC_0)$ is strictly $\ts_{0,\eta_1}$-semistable for some $\eta_1>0$. If $F$ is a stable Jordan--H\"older factor of $E$, then $\Chern{\CCC_0, \leq 2}^{-\frac{5}{4}}(F) = \Chern{\CCC_0, \leq 2}^{-\frac{5}{4}}(\CCC_1)$ and $\Chern{\CCC_0, \leq 2}^{-\frac{5}{4}}(E/F) = \Chern{\CCC_0, \leq 2}^{-\frac{5}{4}}(\CCC_0[1])$.
    \end{lemma}
    \begin{proof}
        By assumption there is a destabilizing sequence in the heart $\Coh^{-\frac{5}{4}}(\pro^3,\CCC_0)$:
        \[\begin{tikzcd}[ampersand replacement=\&,cramped]
                0 \& F \& E \& G \& 0
                \arrow[from=1-1, to=1-2]
                \arrow[from=1-2, to=1-3]
                \arrow[from=1-3, to=1-4]
                \arrow[from=1-4, to=1-5]
            \end{tikzcd}\]
        such that $F$ is $\ts_{0,\eta_1}$-semistable, $G$ is $\ts_{0,\eta_1}$-stable, and
        \begin{equation}
            \tilde{\nu}_{0,\eta_1}(F) = \tilde{\nu}_{0,\eta_1}(E) = \tilde{\nu}_{0,\eta_1}(G) = 0. \label{equ:compare_tilt}
        \end{equation}
        Consider the long exact sequence with respect to the original heart $\Coh(\pro^3,\CCC_0)$. Since $E$ is a sheaf, $\mathcal H^{-1}(E) = 0$ and hence $\mathcal H^{-1}(F) = 0$. That is, $F \simeq \mathcal H^0(F)$ is also a sheaf. The long exact sequence is given by
        \[\begin{tikzcd}[ampersand replacement=\&,cramped,row sep=0]
                \& 0 \& {\mathcal H^{-1}(G)} \& F \& E \& {\mathcal H^0(G)} \& 0. \\
                {\Chern{0}} \&\& 8a \& 8a \& 0 \& 0 \\
                {\Chern{\CCC_0, 1}^{-\frac{5}{4}}} \&\& {2b+4b'-4} \& 2b \& 4 \& {4b'} \\
                {\Chern{\CCC_0, 2}^{-\frac{5}{4}}} \&\& {\frac{1}{4}({c+c'})} \& {\frac{1}{4}c} \& 0 \& {\frac{1}{4}c'}
                \arrow[from=1-2, to=1-3]
                \arrow[from=1-3, to=1-4]
                \arrow[from=1-4, to=1-5]
                \arrow[from=1-5, to=1-6]
                \arrow[from=1-6, to=1-7]
            \end{tikzcd}\]
        We write $\Chern{\CCC_0, \leq 2}^{-\frac{5}{4}}(F) = (8a,2b,\frac{1}{4}c)$ for some $a,b,c \in \Z$. Then $\Chern{\CCC_0, \leq 2}^{-\frac{5}{4}}(G) = (-8a,4-2b,-\frac{1}{4}c)$. Since $\mathcal H^{0}(G)$ is a quotient sheaf of $E$, we have $\Chern{0}(\mathcal H^{0}(G)) = 0$ and hence $\Chern{\CCC_0, \leq 2}^{-\frac{5}{4}}(\mathcal H^{0}(G)) = (0,4b',\frac{1}{4}c')$ for some $b',c' \in \Z$. It follows from the long exact sequence that $\Chern{\CCC_0, \leq 2}^{-\frac{5}{4}}\ab(\mathcal H^{-1}(G)) = (8a,2b+4b'-4,\frac{1}{4}(c+c'))$, as indicated in the diagram above.

        Next suppose that $a = 0$. Then \eqref{equ:compare_tilt} implies that $c=0$, and in particular, $\tilde{\nu}_{0,\eta}(F) = \tilde{\nu}_{0,\eta}(E) = \tilde{\nu}_{0,\eta}(G) = 0$ for all $\eta > 0$. This contradicts the fact that $E$ is $\ts_{0,\eta}$-stable for $\eta \gg 0$. Therefore we must have $a>0$. The condition \eqref{equ:compare_tilt} gives $c = 32\eta_1 a$. Hence $c>0$ as well.

        By construction of the heart $\Coh^{-\frac{5}{4}}(\pro^3,\CCC_0)$, we have
        \[
            F \in \Coh^{> -\frac{5}{4}}(\pro^3,\CCC_0),
            \qquad
            \mathcal H^{-1}(G) \in \Coh^{\leq -\frac{5}{4}}(\pro^3,\CCC_0).
        \]
        It follows that
        \[\dfrac{2b}{8a} > 0, \qquad \dfrac{2b+4b'-4}{8a} \leq 0.\]
        Since $\mathcal H^0(G)$ is a quotient of the pure sheaf $E$, which has generic rank $4$ on its support, we have $0\leq b'\leq 1$. If $b'=1$, the two displayed inequalities would give simultaneously $b>0$ and $b\leq 0$. Thus $b'=0$, and the same inequalities give $b=1$ or $2$.
                                                If $b=2$, then we have $\Chern{\CCC_0, \leq 2}^{-\frac{5}{4}}(G) = (-8a,0,-\frac{1}{4}c)$ and hence $\tilde{\nu}_{0,\eta_1}(G) = +\infty$, which is a contradiction. Hence we must have $b = 1$. Finally, since $F$ is $\ts_{0,\eta_1}$-semistable, the Bogomolov--Gieseker inequality gives
        \[\Delta_{\CCC_0}(F) = 4 - 4ac \geq 0,\]
        which forces $a=c=1$ as $a,c>0$. Now $\Chern{\CCC_0, \leq 2}^{-\frac{5}{4}}(F) = \ab(8,2,\frac{1}{4}) = \Chern{\CCC_0, \leq 2}^{-\frac{5}{4}}(\CCC_1)$ and $\Chern{\CCC_0, \leq 2}^{-\frac{5}{4}}(G) = \ab(-8,2,-\frac{1}{4}) = \Chern{\CCC_0, \leq 2}^{-\frac{5}{4}}(\CCC_0[1])$.
    \end{proof}
    \begin{lemma}[][lem:disc_zero]
        Let $G \in \Coh^{-\frac{5}{4}}(\pro^3,\CCC_0)$ be a $\ts_{\xi_0,\eta_0}$-semistable object for some $\eta_0>0$ and $\xi_0 > \mu_h(G) + \frac{5}{4} = -\frac{1}{4}$. If $\Chern{\CCC_0, \leq 2}^{-\frac{5}{4}}(G) = \Chern{\CCC_0, \leq 2}^{-\frac{5}{4}}(\CCC_0[1])$, then $G \cong \CCC_0[1]$.
    \end{lemma}
    \begin{proof}
        First, $\CCC_j$ is slope stable for any $j \in \Z$, since $\rk(\CCC_j) = 8$ and, by \cref{lem:Chern_noncomm_P3}, any sheaf in $\Coh(\pro^3,\CCC_0)$ has rank divisible by $8$. By \cite[Corollary 3.11.(a)]{bay16}, for any $\eta > \frac{1}{2}\xi^2$,  $\CCC_j[1] \in \Coh^{\xi - \frac{5}{4}}(\pro^3,\CCC_0)$ is $\ts_{\xi,\eta}$-stable if $\xi\geq\mu_h(\CCC_j)+\frac{5}{4}$; and $\CCC_j \in \Coh^{\xi - \frac{5}{4}}(\pro^3,\CCC_0)$ is $\ts_{\xi,\eta}$-stable if $\xi<\mu_h(\CCC_j)+\frac{5}{4}$.
        \medskip
                                                                                                                                                
        Second, since the character $\Chern{\CCC_0, \leq 2}^{-\frac{5}{4}}(G) = (8,-2,\frac{1}{4})$ is primitive, $G$ is in fact $\ts_{\xi_0,\eta_0}$-stable. Using $\Delta_{\CCC_0}(G) = \Delta_{\CCC_0}(\CCC_0[1]) = 0$, by \cite[Corollary 3.11(c)]{bay16}, we have that $G' \coloneqq \mathcal H^{-1}(G)$ is a slope-stable sheaf and that $G'' \coloneqq \mathcal H^0(G)$ has zero-dimensional support. Note that $\Chern{\CCC_0, \leq 2}^{-\frac{5}{4}}(G') = (8,-2,\frac{1}{4})$ and $\Chern{\CCC_0}^{-\frac{5}{4}}(G'') = (0,0,0,\ell)$, where $\ell \geq 0$. We will show that $G' \cong \CCC_0$.

        For $-\frac{5}{4}<\xi<-\frac{1}{4}$, $G' \in \Coh^{\xi - \frac{5}{4}}(\pro^3,\CCC_0)$. Since $G'$ is slope-stable and $\Delta_{\CCC_0}(G') = 0$, by \cite[Lemma 2.7, Corollary 3.11]{bay16}, $G'$ is $\ts_{\xi,\eta}$-stable, and we have $\tilde{\nu}_{\xi,\eta}(\CCC_{-2}[1]) < \tilde{\nu}_{\xi,\eta}(\CCC_{-1}[1]) < \tilde{\nu}_{\xi,\eta}(G')$. Therefore tilt comparison and Serre duality give
        \begin{align}
             & \Hom_{\CCC_0}(G', \CCC_{-1}[1-j]) \cong \Hom_{\CCC_0}(\CCC_1, G'[2+j])^\dual = 0, \qquad j \geq 0; \label{equ:Hom_van_G'_1} \\
             & \Hom_{\CCC_0}(G', \CCC_{-2}[1-j]) \cong \Hom_{\CCC_0}(\CCC_0, G'[2+j])^\dual = 0, \qquad j \geq 0. \label{equ:Hom_van_G'_2}
        \end{align}
        For $-\frac{1}{4}<\xi<\frac{3}{4}$, $G'[1] \in \Coh^{\xi-\frac{5}{4}}(\pro^3,\CCC_0)$. We claim that $G'[1]$ is $\ts_{\xi,\eta}$-stable; if not, then it could only be destabilized along the vertical wall $\xi = -\frac{1}{4}$. Consider the destabilizing sequence in $\Coh^{-\frac{3}{2}}(\pro^3,\CCC_0)$:
        \[\begin{tikzcd}[ampersand replacement=\&,cramped]
                0 \& G'_1 \& G'[1] \& G'_2 \& 0
                \arrow[from=1-1, to=1-2]
                \arrow[from=1-2, to=1-3]
                \arrow[from=1-3, to=1-4]
                \arrow[from=1-4, to=1-5]
            \end{tikzcd}.\]
        Since $\Delta_{\CCC_0}(G') = 0$, the characters of $G'_1$ and $G'_2$ are proportional to that of $G'[1]$. Moreover, $\Chern{\CCC_0, \leq 2}^{-\frac{5}{4}}(G'[1])$ is primitive. We must therefore have $\Chern{\CCC_0, \leq 2}^{-\frac{5}{4}}(G'_1) = 0$. Thus $G'_1$ is a zero-dimensional sheaf. But then $G'_1$ destabilizes $G'$ for $\xi<-\frac{1}{4}$, a contradiction. Hence $G'[1]$ is $\ts_{\xi,\eta}$-stable for $\xi \geq -\frac{1}{4}$. The tilt comparison $\tilde{\nu}_{\xi,\eta}(G'[1]) < \tilde{\nu}_{\xi,\eta}(\CCC_1)$ implies that
        \begin{equation}
            \Hom_{\CCC_0}(\CCC_1, G'[1-j]) = 0, \qquad j \geq 0. \label{equ:Hom_van_G'_3}
        \end{equation}
        Combining \eqref{equ:Hom_van_G'_1}, \eqref{equ:Hom_van_G'_2}, and \eqref{equ:Hom_van_G'_3}, we have $\chi_{\CCC_0}(\CCC_1, G') = 0$. Since we also have $\chi_{\CCC_0}(\CCC_1, \CCC_0) = 0$ and $\Chern{\CCC_0, \leq 2}^{-\frac{5}{4}}(G') = \Chern{\CCC_0, \leq 2}^{-\frac{5}{4}}(\CCC_0)$, we deduce that $[G'] = [\CCC_0]$ in $\Knum(\pro^3,\CCC_0)$. Finally, we have
        \[1 = \chi_{\CCC_0}(\CCC_0, \CCC_0) = \chi_{\CCC_0}(\CCC_0, G') = \hom_{\CCC_0}(\CCC_0, G') - \hom_{\CCC_0}(\CCC_0, G'[1]). \]
        In particular $\Hom_{\CCC_0}(\CCC_0, G') \ne 0$. Since $\mu_h(\CCC_0) = \mu_h(G')$ and $\CCC_0$ is slope stable and locally free, we have $G' \cong \CCC_0$.

        \medskip

        Finally, since $G' \cong \CCC_0 \in \Coh^{\leq (\xi_0-\frac{5}{4})}(\pro^3,\CCC_0)$ and $G'' \in \Coh^{> (\xi_0-\frac{5}{4})}(\pro^3,\CCC_0)$, there is a short exact sequence in the heart $\Coh^{\xi_0-\frac{5}{4}}(\pro^3,\CCC_0)$:
        \[\begin{tikzcd}[ampersand replacement=\&,cramped]
                0 \& \CCC_0[1] \& G \& G'' \& 0.
                \arrow[from=1-1, to=1-2]
                \arrow[from=1-2, to=1-3]
                \arrow[from=1-3, to=1-4]
                \arrow[from=1-4, to=1-5]
            \end{tikzcd}\]
        Note that $\Hom_{\CCC_0}(G'', \CCC_0[2]) = 0$ because $G''$ is $0$-dimensional and $\CCC_0$ is locally free. In particular, the sequence above must split. But this contradicts the stability of $G$ unless $G'' = 0$. We conclude that $G \cong \CCC_0[1]$.
        \qedhere
                
                                            \end{proof}
\begin{proof}[Proof of \cref{prop:P_Pi_stable}]
    By \cref{lem:P_Pi_1st_wall,lem:disc_zero}, if $E$ is strictly $\ts_{0,\eta_1}$-semistable, then there exists a destabilizing sequence
    \[\begin{tikzcd}[ampersand replacement=\&,cramped]
            0 \& F \& E \& \CCC_0[1] \& 0.
            \arrow[from=1-1, to=1-2]
            \arrow[from=1-2, to=1-3]
            \arrow[from=1-3, to=1-4]
            \arrow[from=1-4, to=1-5]
        \end{tikzcd}\]
    In particular $\Hom(E,\CCC_0[1]) = \Hom(\CCC_2,E[2]) \ne 0$ by Serre duality. This contradicts the fact that $E$ is an object of $\Ku(\pro^3,\CCC_0) \subset \ord{\CCC_2}^\perp$. We conclude that $E$ is $\ts_{0,\eta}$-stable for all $\eta > 0$.
\end{proof}

The equality $\tilde{\nu}_{0,\eta}(\pi_*(\EEE(H+h)|_{\sigma^{-1}\varPi})) = 0$ implies that $\pi_*(\EEE(H+h)|_{\sigma^{-1}\varPi}) \in \tilde{\Coh}^{\tilde{\nu}\leq 0}_{0,\eta}$. Hence $\varPsi\sigma^*\cF_{\varPi}[1]$ lies in the heart $\tA_{0,\eta} = \tilde{\Coh}^0_{0,\eta} \cap \Ku(\pro^3,\CCC_0)$.

\begin{corollary}[][cor:P_Pi_stable]
    The object $\varPsi\sigma^*\cF_{\varPi}[1]$ is $\ts^0_{0,\eta}$-stable for all $\eta > 0$. In particular, $\cF_{\varPi}[1]$ is $\sigma_Y$-stable.
\end{corollary}
\begin{proof}
    Suppose that $\varPsi\sigma^*\cF_{\varPi}[1]$ is not $\ts^0_{0,\eta}$-stable. Then there is a non-trivial proper quotient object $G$ of $\varPsi\sigma^*\cF_{\varPi}[1]$ in the heart $\tA_{0,\eta}$ such that $\tilde{\nu}_{0,\eta}(G) \leq \tilde{\nu}_{0,\eta}(\varPsi\sigma^*\cF_{\varPi}[1])$. Consider the short exact sequence in $\tA_{0,\eta}$:
    \[\begin{tikzcd}[ampersand replacement=\&,cramped]
            0 \& F \& {\varPsi\sigma^*\cF_{\varPi}[1]=\pi_*(\EEE(H+h)|_{\sigma^{-1}\varPi})[1]} \& G \& 0.
            \arrow[from=1-1, to=1-2]
            \arrow[from=1-2, to=1-3]
            \arrow[from=1-3, to=1-4]
            \arrow[from=1-4, to=1-5]
        \end{tikzcd}\]
    Consider the long exact sequence in $\Coh^{-\frac{5}{4}}(\pro^3,\CCC_0)$:
    \[\begin{tikzcd}[ampersand replacement=\&,cramped]
            0 \& {\mathcal{H}_{\sharp}^{-1}(F)} \& {\pi_*(\EEE(H+h)|_{\sigma^{-1}\varPi})} \& {\mathcal{H}_{\sharp}^{-1}(G)} \& {\mathcal{H}_{\sharp}^{0}(F)} \& 0.
            \arrow[from=1-1, to=1-2]
            \arrow[from=1-2, to=1-3]
            \arrow[from=1-3, to=1-4]
            \arrow[from=1-4, to=1-5]
            \arrow[from=1-5, to=1-6]
        \end{tikzcd}\]
    By assumption we have
    \begin{equation}
        \tilde{\nu}_{0,\eta}(\pi_*(\EEE(H+h)|_{\sigma^{-1}\varPi})) \geq \tilde{\nu}_{0,\eta}(\mathcal{H}_{\sharp}^{-1}(G)). \label{equ:Coh0_slope_compare}
    \end{equation}
    Moreover, from the long exact sequence and the $\ts_{0,\eta}$-stability of $\pi_*(\EEE(H+h)|_{\sigma^{-1}\varPi})$, we also have
    \[\tilde{\nu}_{0,\eta}(\mathcal{H}_{\sharp}^{-1}(F)) < \tilde{\nu}_{0,\eta}(\pi_*(\EEE(H+h)|_{\sigma^{-1}\varPi})) \leq \tilde{\nu}_{0,\eta}(\mathcal{H}_{\sharp}^{-1}(G)) \leq 0 < \tilde{\nu}_{0,\eta}(\mathcal{H}_{\sharp}^{0}(F)) .\]
    Combining these two facts, the equality in \eqref{equ:Coh0_slope_compare} must hold. Then we have $\mathcal{H}_{\sharp}^{-1}(F) = 0$ and $\tZ_{0,\eta}(\mathcal{H}_{\sharp}^{0}(F)) = 0$.
    The latter equation implies that $F$ is a sheaf in $\Coh(\pro^3,\CCC_0)$ with $0$-dimensional support. Note that the sheaf $\pi_*(\EEE(H+h)|_{\sigma^{-1}\varPi})$, which has $2$-dimensional support, has depth $2$. It follows that
    \[\Hom(F,\pi_*(\EEE(H+h)|_{\sigma^{-1}\varPi})[1]) = 0.\]
    In particular $F = 0$ and $G \cong \pi_*\ab(\EEE(H+h)|_{\sigma^{-1}\varPi})[1]$. This is a contradiction.
\end{proof}

By \cref{rmk:S&O:PFK}, the objects $\cF_{\varPi}$, $\PPP_{\varPi}$, and $\KKK_{\varPi}$ are related by the Serre functor. Since the stability condition $\sigma_Y$ is Serre-invariant, we have the following corollary:
\begin{corollary}[][cor:F&K_stable]
    The objects $\PPP_{\varPi}$ and $\KKK_{\varPi}$ are $\sigma_Y$-stable. Moreover, for any numerical character $v$ with $\chi(v,v) = -1$, the moduli space $\Modss(\Ku(Y),v)$ is non-empty.
\end{corollary}

\subsection{The Fano surface of planes on smooth cubic 5-fold} \label{subsec:Fano}

As an example, we will study the lowest-dimensional moduli spaces, of dimension $2$, and relate them to a classical moduli space.

\begin{definition}
    Let $Y$ be a smooth cubic 5-fold. The Hilbert scheme $\mathcal F_2(Y)$ parametrizing $2$-planes in $Y$ is a two-dimensional projective variety \cite[Proposition 1.8]{col86}, called the \textbf{Fano surface of planes} of $Y$.
\end{definition}

\begin{remark}[][rmk:Fano_surf]
    By \cite[Proposition 1.10]{col86}, the Fano surface $\mathcal F_2(Y)$ is smooth and irreducible if $Y$ is a general member of a general Lefschetz pencil of cubic 5-folds.
\end{remark}

\begin{proposition}[][prop:moduli_dim2]
    Let $Y$ be a general cubic 5-fold in the sense of \cref{rmk:Fano_surf}. We have isomorphisms of moduli spaces:
    \[\begin{tikzcd}[ampersand replacement=\&,cramped,column sep=huge,row sep=0]
	{\Mods(\Ku(Y),\kappa_2)} \& {\Mods(\Ku(Y),\kappa_2-\kappa_1)} \& {\Mods(\Ku(Y),-\kappa_1)} \\ 
    {[\PPP_{\varPi}]} \& {[\KKK_{\varPi}[1]]} \& {[\FFF_{\varPi}[1]]}
				\arrow["{\mathsf S_{\Ku(Y)}[-2]}", "\cong"', from=1-1, to=1-2]
	\arrow["{\mathsf S_{\Ku(Y)}[-2]}", "\cong"', from=1-2, to=1-3]
    \arrow[mapsto, from=2-1, to=2-2]
	\arrow[mapsto, from=2-2, to=2-3]
\end{tikzcd}\]
        Moreover, each has a connected component that is a smooth projective surface isomorphic to the Fano surface $\mathcal F_2(Y)$.
\end{proposition}

\begin{remark}
    We conjecture that these moduli spaces are in fact irreducible, and hence each is isomorphic to $\cF_2(Y)$. A verification would likely involve a subtle reverse wall-crossing argument, in the same spirit as \cite{BMMS} or \cite{FP23}. Alternatively, there is another piece of evidence for the conjectural irreducibility: one could consider the degeneration of the family of stability conditions to a stability condition over a nodal cubic 5-fold $Y^{\mathsf{nod}}$. In the forthcoming work \cite{LPZ_nodalcubic5}, we show that the moduli space $\Mods(\Ku(Y^{\mathsf{nod}}),\kappa_1)$ is already isomorphic to $\cF_2(Y^{\mathsf{nod}})$.  
\end{remark}

\begin{lemma}[][lem:Ext_F_Pi]
    Let $Y$ be a smooth cubic 5-fold, and $\varPi \subset Y$ a $2$-plane. Denote by $\Nor{\varPi}{Y}$ the normal bundle of $\varPi$ in $Y$. Then the sheaf $\cF_{\varPi}$ defined in \eqref{seq:F_Pi_def} has Ext groups:
                            \[ \Ext^\bullet(\cF_{\varPi}, \cF_{\varPi}) = \C[0] \oplus \Hlg^0(\varPi,\Nor{\varPi}{Y})[-1] \oplus \Hlg^1(\varPi,\Nor{\varPi}{Y})[-2]. \]
    In particular, if $Y$ is general, then $\Ext^\bullet(\cF_{\varPi}, \cF_{\varPi}) = \C[0] \oplus \C^{2}[-1]$.
\end{lemma}
\begin{proof}
    In the first step we compute $\Ext^\bullet(j_*\sheaf{\varPi},j_*\sheaf{\varPi})$, where $j\colon \varPi \to Y$ is the inclusion map. We have the spectral sequence
    \begin{equation}
        E_2^{p,q} \coloneqq \Hlg^p(Y, \SExt^q(j_*\sheaf{\varPi}, j_*\sheaf{\varPi})) = \Hlg^p{\left( \varPi, {\textstyle \bigwedge}^q \Nor{\varPi}{Y} \right)} \implies \Ext^{p+q}(j_*\sheaf{\varPi}, j_*\sheaf{\varPi}).
    \end{equation}
    From the inclusions $\varPi \subset Y \subset \pro^6$ of smooth projective varieties, we find that $\Nor{\varPi}{Y}$
    is a rank $3$ vector bundle that fits into the relative normal sequence:
    \begin{equation}
        \begin{tikzcd}[ampersand replacement=\&,cramped]
            0 \& \Nor{\varPi}{Y} \& \Nor{\varPi}{\pro^6} = \sheaf{\varPi}(1)^{\oplus 4} \& \Nor{Y}{\pro^6}\big|_{\varPi} = \sheaf{\varPi}(3) \& 0.
            \arrow[from=1-1, to=1-2]
            \arrow[from=1-2, to=1-3]
            \arrow[from=1-3, to=1-4]
            \arrow[from=1-4, to=1-5]
        \end{tikzcd}\label{seq:normal}
    \end{equation}
    It follows that $ \Hlg^i(\Nor{\varPi}{Y})= 0$ for $i \ne 0,1$ and $\det \Nor{\varPi}{Y} = \sheaf{\varPi}(1)$. In particular $\bigwedge^2 \Nor{\varPi}{Y} \cong \Nor{\varPi}{Y}^{\dual}(1)$. Dualizing \eqref{seq:normal} and twisting by $\sheaf{}(1)$, we obtain the relative conormal sequence
    \begin{equation}
        \begin{tikzcd}[ampersand replacement=\&,cramped]
            0 \& \sheaf{\varPi}(-2) \& \sheaf{\varPi}^{\oplus 4} \& \bigwedge^2 \Nor{\varPi}{Y} \& 0.
            \arrow[from=1-1, to=1-2]
            \arrow[from=1-2, to=1-3]
            \arrow[from=1-3, to=1-4]
            \arrow[from=1-4, to=1-5]
        \end{tikzcd}\label{seq:conormal}
    \end{equation}
    In particular we have $\Hlg^0(\bigwedge^2 \Nor{\varPi}{Y}) \cong \Hlg^0(\sheaf{\varPi}^{\oplus 4}) = \C^4$ and $\Hlg^i(\bigwedge^2 \Nor{\varPi}{Y}) = 0$ for $i \ne 0$. The $E_2$-page of the spectral sequence is given by
    \[
        \begin{tabular}{|ccc}
            {$\C^3$}                    & 0                           & 0 \\
            {$\C^4$}                    & 0                           & 0 \\
            {$\Hlg^0(\Nor{\varPi}{Y})$} & {$\Hlg^1(\Nor{\varPi}{Y})$} & 0 \\
            {$\C$}                      & 0                           & 0 \\ \hline
        \end{tabular}\]
    Note that the sequence degenerates at the $E_2$-page because all differentials are zero. It follows that
    \[ \Ext^\bullet(j_*\sheaf{\varPi}, j_*\sheaf{\varPi}) = \C[0] \oplus \Hlg^0(\Nor{\varPi}{Y})[-1] \oplus \left(\Hlg^1(\Nor{\varPi}{Y}) \oplus \C^4\right)[-2] \oplus \C^3[-3]. \]

    \medskip
    The next step is to compute $\Ext^\bullet(\ideal_{\varPi},\ideal_{\varPi})$. We have the double complex of long exact sequences:
    \begin{equation}
        \begin{tikzcd}[ampersand replacement=\&,cramped,row sep=small]
                \&\& {\Ext^i(\sheaf{Y},\ideal_{\varPi})=0} \\
                \&\& {\Ext^i(\ideal_{\varPi},\ideal_{\varPi})} \\
                {\Ext^{i}(j_*\sheaf{\varPi},\sheaf{Y})} \& {\Ext^{i}(j_*\sheaf{\varPi},j_*\sheaf{\varPi})} \& {\Ext^{i+1}(j_*\sheaf{\varPi},\ideal_{\varPi})} \& {\Ext^{i+1}(j_*\sheaf{\varPi},\sheaf{Y})} \\
                \&\& {\Ext^{i+1}(\sheaf{Y},\ideal_{\varPi})=0}
                \arrow[from=1-3, to=2-3]
                \arrow["\sim", from=2-3, to=3-3]
                \arrow[from=3-1, to=3-2]
                \arrow[from=3-2, to=3-3]
                \arrow[from=3-3, to=3-4]
                \arrow[from=3-3, to=4-3]
            \end{tikzcd} \label{les:Ext_I_Pi}
    \end{equation}
    Since $\Ext^\bullet(\sheaf{Y},\ideal_{\varPi}) = 0$, we have $\Ext^i(\ideal_{\varPi},\ideal_{\varPi}) \cong \Ext^{i+1}(j_*\sheaf{\varPi},\ideal_{\varPi})$ for all $i$. By Grothendieck--Verdier duality, we have
    \[ \Ext^{\bullet}(j_*\sheaf{\varPi},\sheaf{Y}) = \Ext^{\bullet}(\sheaf{\varPi},j^!\sheaf{Y}) = \Ext^{\bullet}(\sheaf{\varPi},\sheaf{\varPi}(1))[-3] = \C^3[-3].\]
    Therefore $\Ext^{i}(j_*\sheaf{\varPi},j_*\sheaf{\varPi}) \cong \Ext^i(\ideal_{\varPi},\ideal_{\varPi})$ for $i \ne 2,3$. For $i = 2,3$, consider the horizontal long exact sequence:
    \[\begin{tikzcd}[ampersand replacement=\&,cramped,row sep=0]
            0 \& {\Ext^{2}(j_*\sheaf{\varPi},j_*\sheaf{\varPi})} \& {\Ext^{2}(\ideal_{\varPi},\ideal_{\varPi})} \& {\Ext^{3}(j_*\sheaf{\varPi},\sheaf{Y})} \\
            {} \& {\Ext^{3}(j_*\sheaf{\varPi},j_*\sheaf{\varPi})} \& {\Ext^{3}(\ideal_{\varPi},\ideal_{\varPi})} \& 0.
            \arrow[from=1-1, to=1-2]
            \arrow[from=1-2, to=1-3]
            \arrow[from=1-3, to=1-4]
            \arrow["\eta", from=2-1, to=2-2]
            \arrow[from=2-2, to=2-3]
            \arrow[from=2-3, to=2-4]
        \end{tikzcd}\]
    Note that by Grothendieck--Verdier duality, the map $\eta\colon {\Ext^{3}(j_*\sheaf{\varPi},\sheaf{Y})} \to {\Ext^{3}(j_*\sheaf{\varPi},j_*\sheaf{\varPi})}$ is functorially isomorphic to the map $\Hlg^0(\sheaf{\varPi}(1)) \to \Hlg^0(j^*j_*\sheaf{\varPi}(1))$, which is an isomorphism. Hence we deduce that
    \[ \Ext^\bullet(\ideal_{\varPi}, \ideal_{\varPi}) = \C[0] \oplus \Hlg^0(\Nor{\varPi}{Y})[-1] \oplus \left(\Hlg^1(\Nor{\varPi}{Y}) \oplus \C^4\right)[-2]. \]

    \medskip
    Finally, we compute $\Ext^\bullet(\cF_{\varPi},\cF_{\varPi})$. Using \eqref{seq:F_Pi_def}, we have the double complex of long exact sequences:
    \begin{equation}
        \begin{tikzcd}[ampersand replacement=\&,cramped,row sep=small]
                \&\& {\Ext^i(\sheaf{Y}^{\oplus 4},\cF_{\varPi})=0} \\
                \&\& {\Ext^i(\cF_{\varPi},\cF_{\varPi})} \\
                {\Ext^{i}(\ideal_{\varPi}(1),\sheaf{Y}^{\oplus 4})} \& {\Ext^{i}(\ideal_{\varPi}(1),\ideal_{\varPi}(1))} \& {\Ext^{i+1}(\ideal_{\varPi}(1),\cF_{\varPi})} \& {\Ext^{i+1}(\ideal_{\varPi}(1),\sheaf{Y}^{\oplus 4})} \\
                \&\& {\Ext^{i+1}(\sheaf{Y}^{\oplus 4},\cF_{\varPi})=0}
                \arrow[from=1-3, to=2-3]
                \arrow["\sim", from=2-3, to=3-3]
                \arrow[from=3-1, to=3-2]
                \arrow[from=3-2, to=3-3]
                \arrow[from=3-3, to=3-4]
                \arrow[from=3-3, to=4-3]
            \end{tikzcd}   \label{les:Ext_F_Pi}
    \end{equation}
    Since $\cF_{\varPi}$ is acyclic, we have $\Ext^i(\cF_{\varPi},\cF_{\varPi}) \cong \Ext^{i+1}(\ideal_{\varPi}(1),\cF_{\varPi})$ for all $i$. By the proof of \cref{prop:Ku_objs}.(2), we know that $\Ext^{\bullet}(\ideal_{\varPi}(1),\sheaf{Y}) = \C[-2]$. Hence $\Ext^{i}(\ideal_{\varPi}(1),\ideal_{\varPi}(1)) \cong \Ext^i(\cF_{\varPi},\cF_{\varPi})$ for $i \ne 1,2$. For $i = 1,2$, consider the horizontal long exact sequence:
    \begin{equation}
        \begin{tikzcd}[ampersand replacement=\&,cramped,row sep=0]
            0 \& {\Ext^{1}(\ideal_{\varPi}(1),\ideal_{\varPi}(1))} \& {\Ext^{1}(\cF_{\varPi},\cF_{\varPi})} \& {\Ext^{2}(\ideal_{\varPi}(1),\sheaf{Y}^{\oplus 4})} \\
            {} \& {\Ext^{2}(\ideal_{\varPi}(1),\ideal_{\varPi}(1))} \& {\Ext^{2}(\cF_{\varPi},\cF_{\varPi})} \& {0.}
            \arrow[from=1-1, to=1-2]
            \arrow[from=1-2, to=1-3]
            \arrow[from=1-3, to=1-4]
            \arrow["\delta", from=2-1, to=2-2]
            \arrow[from=2-2, to=2-3]
            \arrow[from=2-3, to=2-4]
        \end{tikzcd}\label{les:F_Pi}
    \end{equation}
    The map $\delta\colon \Ext^{2}(\ideal_{\varPi}(1),\sheaf{Y}^{\oplus 4}) \to \Ext^{2}(\ideal_{\varPi},\ideal_{\varPi})$ is the same as $\Ext^{3}(j_*\sheaf{\varPi}(1),\sheaf{Y}^{\oplus 4}) \to \Ext^{2}(j_*\sheaf{\varPi},j_*\sheaf{\varPi})$, which by Grothendieck--Verdier duality is isomorphic to the map $\Hlg^0(\sheaf{\varPi}^{\oplus 4}) \to \Hlg^{-1}(j^*j_*\sheaf{\varPi}(1))$. Note that
    \[ \Hlg^{-1}(j^*j_*\sheaf{\varPi}(1)) \cong \Hlg^0(\Nor{\varPi}{Y}^\dual(1)) \oplus \Hlg^1(\textstyle\bigwedge^2\Nor{\varPi}{Y}^\dual(1)) \cong  \Hlg^0(\Nor{\varPi}{Y}^\dual(1)) \oplus \Hlg^1(\Nor{\varPi}{Y}).\]
    The image of the map $\delta\colon \Hlg^0(\sheaf{\varPi}^{\oplus 4}) \to \Hlg^0(\Nor{\varPi}{Y}^\dual(1)) \oplus \Hlg^1(\Nor{\varPi}{Y})$ lands in the first summand; in fact it is the isomorphism on global sections induced by the conormal sequence \eqref{seq:conormal}. We conclude from the long exact sequence that
    \[ \Ext^\bullet(\cF_{\varPi}, \cF_{\varPi}) = \C[0] \oplus \Hlg^0(\Nor{\varPi}{Y})[-1] \oplus \Hlg^1(\Nor{\varPi}{Y})[-2]. \]
    Finally, for general $Y$, we have $\Hlg^0(\Nor{\varPi}{Y}) \cong \C^2$ and $\Hlg^1(\Nor{\varPi}{Y}) = 0$ for all $\varPi \in \cF_2(Y)$. Hence the result follows.
\end{proof}

\begin{proof}[Proof of \cref{prop:moduli_dim2}.]
    First assume that $Y$ is a smooth cubic 5-fold. The Serre functor places the
    three classes in the same orbit, and hence
    \[
        \Mods(\Ku(Y),\kappa_2)
        \cong \Mods(\Ku(Y),\kappa_2-\kappa_1)
        \cong \Mods(\Ku(Y),-\kappa_1).
    \]

    Let $M$ be the Gieseker moduli space of sheaves on $Y$ parametrizing twisted ideal sheaves $\ideal_{\varPi}(1)$ for 2-planes $\varPi \subset Y$. It is known that $M$ is isomorphic to the Fano surface $\mathcal F_2(Y)$. Let $\mathcal{L} \in \Dcatb(Y \times M)$ be the universal sheaf on $Y \times M$. By \cite[Theorem 5.8]{kuz11}, we have an SOD of the form
    \[\Dcatb(Y \times M) = \ord{\sheaf{Y}(-2) \boxtimes \Dcatb(M),\ \sheaf{Y}(-1) \boxtimes \Dcatb(M),\ \Ku(Y \times M),\ \sheaf{Y} \boxtimes \Dcatb(M),\ \sheaf{Y}(1) \boxtimes \Dcatb(M)}.\]
    Let $\mathcal{L}'$ be the projection of $\mathcal{L}$ to $\Ku(Y \times M)$. Since $\pr_Y(\ideal_{\varPi}(1)) = \cF_{\varPi}[1]$, the object $\mathcal{L}'$ is a family of $\sigma_Y$-stable objects in $\Ku(Y)$ of character $-\kappa_1$. It induces a morphism $M \cong \cF_2(Y) \to \Modss(\Ku(Y),-\kappa_1)$ whose image consists of the objects $\cF_{\varPi}[1]$. The morphism is injective, since for distinct planes $\varPi$, $\varPi' \subset Y$, we have
    \begin{equation}\label{equ:F_Pi_distinct}
        \Hom(\cF_{\varPi}[1],\cF_{\varPi'}[1]) \cong \Hom(\ideal_{\varPi},\ideal_{\varPi'}) \cong \Hom(j_*\sheaf{\varPi},j_*\sheaf{\varPi'}) = 0. 
    \end{equation}
    In the proof of \cref{lem:Ext_F_Pi}, the sequence \eqref{les:F_Pi} shows that the morphism $\cF_2(Y) \to \Modss(\Ku(Y),-\kappa_1)$ induces isomorphisms on the Zariski tangent spaces: $\Ext^1(\ideal_{\varPi}(1), \ideal_{\varPi}(1)) \cong \Ext^1(\cF_{\varPi}[1], \cF_{\varPi}[1])$. Since the source is proper and the target is separated, this injective morphism is proper; together with the tangent-space isomorphisms, it is a closed immersion.

    Moreover, for general $Y$, by \cref{lem:Ext_F_Pi} and \cite[Corollary 1.4]{col86}, we have, for any plane $\varPi \subset Y$,
    \[\Ext^2(\cF_{\varPi}[1], \cF_{\varPi}[1]) \cong \Hlg^1(\varPi,\Nor{\varPi}{Y}) = 0, \]
    which implies that $\Modss(\Ku(Y),-\kappa_1)$ is smooth at $[\cF_{\varPi}[1]]$ (\cite[Corollary 4.5.2]{huybMS}). It follows that the morphism $\cF_2(Y) \to \Modss(\Ku(Y),-\kappa_1)$ is also an open immersion, which identifies $\cF_2(Y)$ with a connected component of $\Modss(\Ku(Y),-\kappa_1)$. \qedhere
\end{proof}

As a corollary, we give a new proof of a classical result \cite[\S 2.2.2]{IM08} which relates the Fano surface $\cF_2(Y)$ to the Fano variety of lines on a hyperplane section of $Y$.

\begin{definition}
    Let $X$ be a general cubic 4-fold not containing a plane, and let $Y$ be a general cubic 5-fold containing $X$ as a hyperplane section. Let $i\colon X \hookrightarrow Y$ denote the inclusion. For any plane $\varPi \subset Y$,  $\ell \coloneqq X \cap \varPi$ is a line in $X$. It is known that the Fano variety $\cF_1(X)$ of lines on $X$ is a smooth projective hyper-K\"ahler manifold of dimension $4$. This gives a map:
    \[\begin{tikzcd}[ampersand replacement=\&,cramped,row sep=0]
            {r_1 \colon \cF_{2}(Y)} \& {\cF_1(X)} \\
            {[\varPi]} \& {[X \cap \varPi]}
            \arrow[from=1-1, to=1-2]
            \arrow[maps to, from=2-1, to=2-2]
        \end{tikzcd}\]
\end{definition}

\begin{corollary}[][cor:Lag_imm]
    For $X$ and $Y$ as above, the morphism $r_1 \colon \cF_2(Y) \to \cF_1(X)$ is generically injective and unramified; its image is a (singular) Lagrangian surface in the hyper-K\"ahler 4-fold $\cF_1(X)$.
\end{corollary}
\begin{proof}
    The fact that $r_1$ is generically injective is proved in \cite[Proposition 7]{IM08}. 
                Here we give a different proof of the unramified and Lagrangian properties using a modular interpretation of $\cF_2(Y)$ and $\cF_1(X)$.

    Recall from \cite{BLMS,LPZ23} that for a line $\ell \subset X$, we can associate the objects $\FFF_{\ell}, \PPP_{\ell} \in \Ku(X)$, given in distinguished triangles
    \begin{equation}
        \begin{tikzcd}[ampersand replacement=\&,cramped,row sep = 0]
            {\sheaf{X}(-1)[1]} \& {\PPP_{\ell}} \& {\ideal_{\ell \mid X}} \& {} \\
            {\FFF_{\ell}} \& {\sheaf{X}^{\oplus 4}} \& {\ideal_{\ell \mid X}(1)} \& {}
            \arrow[from=1-1, to=1-2]
            \arrow[from=1-2, to=1-3]
            \arrow["{+1}", from=1-3, to=1-4]
            \arrow[from=2-1, to=2-2]
            \arrow[from=2-2, to=2-3]
            \arrow["{+1}", from=2-3, to=2-4]
        \end{tikzcd} \label{seq:P_ell}
    \end{equation}
    They are stable with respect to the stability condition $\sigma_X$ on $\Ku(X)$ as constructed in \cite{BLMS}. By \cite{LPZ23}, the Bridgeland moduli space $M_1(X) \coloneqq \mathrm{M}_{\sigma_X}(\Ku(X),i^*\kappa_2)$ parametrizing the objects $\PPP_{\ell}$, is isomorphic to the Fano variety $\cF_1(X)$. The holomorphic symplectic form $\omega$ on $M_1(X)$ is given by the Yoneda pairing (\cite{KM09}):
    \[\begin{tikzcd}[ampersand replacement=\&,cramped,row sep=0]
	{\omega\colon\Ext^1_X(F,F) \times \Ext^1_X(F,F)} \& {\Ext^2_X(F,F)} \& \C \\
	{(F \xrightarrow{\alpha} F[1],\ F \xrightarrow{\beta} F[1])} \& {(F \xrightarrow{\beta[1]\circ\alpha} F[2])}
	\arrow[from=1-1, to=1-2]
	\arrow["\sim", from=1-2, to=1-3]
	\arrow[maps to, from=2-1, to=2-2]
    \end{tikzcd}\]
    Similarly, by \cref{prop:moduli_dim2}, $\cF_2(Y)$ is isomorphic to a connected component $M_2(Y)$ of the moduli space $\Modss(\Ku(Y),\kappa_2)$ parametrizing the objects $\PPP_{\varPi}$. In fact the morphism $r_1\colon M_2(Y) \to M_1(X)$ is induced by the functor $i^*\colon \Ku(Y) \to \Ku(X)$. 
    For $[F] = [i^*E] \in M_1(X)$, and $\alpha = i^*\alpha'$, $\beta = i^*\beta' \in i^*(\Ext^1_Y(E,E))$, by functoriality of $i^*$ we have 
    \[\omega(\alpha,\beta) = \beta[1] \circ \alpha = i^*\beta'[1] \circ i^*\alpha' = i^*(\beta'[1] \circ \alpha') = 0,\]
    where the last equality follows from the fact that $\Ext^2_Y(E,E) = 0$. Hence $\omega|_{i^*M_2(Y)} = 0$, which proves that $\img r_1$ is Lagrangian.

    To show that $r_1$ is unramified, it suffices to show that it induces injective maps on tangent spaces. Consider the distinguished triangle on $Y$:
    \begin{equation}
        \begin{tikzcd}[ampersand replacement=\&,cramped]
            {\PPP_{\varPi}(-1)} \& {\PPP_{\varPi}} \& {i_*i^*\PPP_{\varPi}} \& {}
            \arrow[from=1-1, to=1-2]
            \arrow[from=1-2, to=1-3]
            \arrow["{+1}", from=1-3, to=1-4]
        \end{tikzcd}
    \end{equation}
    Applying $\Hom(\PPP_{\varPi},-)$ we obtain the long exact sequence:
    \begin{equation}
        \begin{tikzcd}[ampersand replacement=\&,cramped]
            {\Ext^1(\PPP_{\varPi}(1),\PPP_{\varPi})} \& {\Ext^1(\PPP_{\varPi},\PPP_{\varPi})} \& {\Ext^1(i^*\PPP_{\varPi},i^*\PPP_{\varPi})} \& {}
            \arrow[from=1-1, to=1-2]
            \arrow[from=1-2, to=1-3]
            \arrow[from=1-3, to=1-4]
        \end{tikzcd}
    \end{equation}
    Using Serre duality (\cref{rmk:S&O:PFK}) and \cref{lem:Ext_F_Pi}, we have 
    \begin{equation}
        \begin{aligned}
            \Ext^1(\PPP_{\varPi}(1),\PPP_{\varPi}) &\cong \Ext^1(\lmut{\sheaf{Y}}\PPP_{\varPi}(1),\PPP_{\varPi}) \cong \Ext^1(\cF_{\varPi}[1],\PPP_{\varPi})\\ 
            &\cong \Hom(\mathsf{S}_{\Ku(Y)}^{-1}\PPP_{\varPi}, \cF_{\varPi})^\dual \cong \Ext^2(\cF_{\varPi},\cF_{\varPi})^\dual = 0.
        \end{aligned}
    \end{equation}
    Hence $\Ext^1(\PPP_{\varPi},\PPP_{\varPi}) \to \Ext^1(i^*\PPP_{\varPi},i^*\PPP_{\varPi})$ is injective and the morphism $r_1$ is indeed unramified.
\end{proof}

\subsection{Plane incidence in the cubic 5-fold}\label{subsec:plane_incid}

In this subsection we return to the classical geometry of the Fano surface $\cF_2(Y)$ of planes, and discuss how two planes $\varPi$ and $\varPi'$ intersect. This serves as a preparation of various computations in the next section. Define the incidence locus 
\begin{equation}
    \sI_0 \coloneqq \ab\{ (\varPi,\, \varPi') \mid \varPi \cap \varPi' \ne \varnothing \} \subset \cF_2(Y) \times \cF_2(Y) .
\end{equation}
For $k = 1,2$, define also its strata
\begin{equation}
    \sI_k \coloneqq \ab\{ (\varPi,\, \varPi') \mid \dim(\varPi \cap \varPi') \geq k \}  \subset \cF_2(Y) \times \cF_2(Y) .
\end{equation}
In particular $\sI_2 = \Delta_{\cF}$ is just the diagonal in $\cF_2(Y)$. For any $\varPi_0 \in \cF_2(Y)$, let $\sI_k(\varPi_0)$ denotes the locus $\{\varPi \in \cF_2(Y) \mid (\varPi,\, \varPi_0) \in \sI_k\} \subset \cF_2(Y)$.

The following result shows that $\sI_0(\varPi) \ne \cF_2(Y)$ for any $\varPi \subset Y$.
\begin{lemma}[][lem:disjoint_planes]
Let $Y$ be a smooth cubic $5$-fold, and let
$\varPi\subset Y$ be a plane. Then there exists a plane
$\varPi'\subset Y$ such that $\varPi\cap\varPi'=\varnothing.$
\end{lemma}
\begin{proof}
First we claim that $\cF$ is reduced. By \cite[(1.1)]{col86}, the equation of $Y$ containing a plane $\varPi = \bV(x_0,...,x_3)$ can be written as 
\begin{equation}
    f = \sum_{i=0}^3 x_iq_i(x_4,x_5,x_6) + \sum_{i,j=0}^3x_ix_j\ell_{ij}(x_4,x_5,x_6) + c(x_0,x_1,x_2,x_3),
\end{equation}
where $q_i$ are quadratic forms, $\ell_{ij}$ lienar forms and $c$ a cubic form. If $q_0,...,q_3$ are linearly independent, then $\cF$ is smooth at $[\varPi]$ (\cite[Corollary 1.4]{col86}) and hence is reduced at $[\varPi]$. Along the special plane $\varPi$ where $q_0,...,q_3$ are linearly dependent, since $Y$ is smooth along $\varPi$, $q_0,...,q_3$ spans a base-point-free linear system (\cite[p.\ 6]{Mbo23}). After a change of coordinates we may assume that $q_0 = 0$ and that $q_1,q_2,q_3$ are independent. Let $R = \C[x_4,x_5,x_6]$ and $I = \ord{q_1,q_2,q_3}$, and pick an isomorphism $\lambda\colon R_3/I_3 \xrightarrow{\sim}\C$. A plane near $\varPi$ is given by the linear forms $x_i = L_i(x_4,x_5,x_6) = \sum_{j=1}^3 a_{ij}x_{j+3}$ for $i=0,1,2,3$. The parameters $a_{ij}$ for $1 \leq i,j \leq 3$ can be formally inverted as power series of $a_{01},a_{02},a_{03}$. Then the complete local ring of $\cF$ at $[\varPi]$ can be expressed as 
\begin{equation}
    \widehat{\sO}_{\cF,[\varPi]} = \C[\![a_{01},a_{02},a_{03}]\!]/\ord{g},
\end{equation}
where the lowest term of $g$ is either $\lambda(\ell_{00}L_0^2)$ or $\lambda(c_{000}L_0^3)$, which are square-free. Hence $\cF$ is reduced at $[\varPi]$, proving the claim.

Write $Y \subset \pro(V) \cong \pro^6$ and set $\cG=\Gr(3,V)$ and $\cF=\cF_2(Y)\subset \cG$. The Fano surface $\cF$ is locally a reduced hypersurface of dimension $2$, and in particular it is a local complete intersection: twisting its Koszul resolution
(\cite[(2.1)]{Mbo23}) by $\sheaf{\cG}(1)$ and applying the
Borel--Weil--Bott theorem \cite[Theorem 3.2]{Mbo23} gives
    \begin{equation}\label{equ:Fano_plane_linear_forms}
        \Hlg^0(\cG,\sheaf{\cG}(1))
        \cong
        \Hlg^0(\cF,\sheaf{\cF}(1)).
    \end{equation}
Thus $\cF$ is contained in no Pl\"ucker hyperplane.

Write $\varPi=\pro(A)$. The Schubert locus $ \Omega_{\varPi} \coloneqq \ab\{U\in \cG\mid U\cap A\neq 0\}$ is the common zero locus of the 4-dimensional subspace
\[
    \textstyle\bigwedge^3(V/A)^\dual
    \hookrightarrow
    \textstyle\bigwedge^3V^\dual
    =
    \Hlg^0(\cG,\sheaf{\cG}(1)).
\]
If every plane contained in $Y$ met $\varPi$, then
$\cF\subset\Omega_{\varPi}$ set-theoretically. Since $\cF$ is reduced,
every element of $\bigwedge^3(V/A)^\dual$ would vanish on $\cF$,
contradicting the isomorphism \eqref{equ:Fano_plane_linear_forms}. Hence some plane
$\varPi'\subset Y$ is disjoint from $\varPi$.
\end{proof}

For the following results we need $Y$ to be a general cubic 5-fold so that $\cF_2(Y)$ is a smooth irreducible projective surface (\cref{rmk:Fano_surf}). 
\begin{lemma}[][lem:plane_incid]
Let $Y$ be a general cubic 5-fold.
\begin{enumerate}[dense]
    \item The incidence locus $\sI_0$ is a closed subset of
    dimension $3$ in $\cF_2(Y) \times \cF_2(Y)$ such that the projections $\sI_0 \setminus \Delta_{\cF_2(Y)} \to \cF_2(Y)$ to each component is surjective. 
    \item Fix any $\varPi \in \cF_2(Y)$; then $\sI_0(\varPi)$ is a closed subset of dimension $1$ in $\cF_2(Y)$.
\end{enumerate}
\end{lemma}
\begin{proof}
    \begin{enumerate}
        \item It is clear that $\sI_0$ is a proper closed subscheme of the 4-dimensional smooth projective variety $\cF_2(Y) \times \cF_2(Y)$. For the first claim it suffices to show that $\sI_0$ is 3-dimensional. This is a standard dimension counting argument by considering the universal family. First we note that the set of intersecting planes in $\pro^6$, 
        \begin{equation}
            \ab\{(\varPi,\, \varPi') \in \Gr(3,7)^2 \mid \varPi \cap \varPi' \ne \varnothing\}
        \end{equation}
        has dimension $\alpha = \dim \Gr(1,7) + 2\dim\Gr(2,6) = 22$. Indeed, the dimension is counted by first choosing a point in $\pro^6$ and then choosing two 2-planes containing that point. On the other hand, requiring that a smooth cubic 5-fold $Y$ contain a fixed pair of planes $(\varPi, \varPi')$ intersecting at a point $p$ imposes $\beta$ independent linear conditions on its equation, where 
        \begin{equation}
            \beta = \hlg^0(\sheaf{\varPi}(3)) + \hlg^0(\sheaf{\varPi'}(3)) - \hlg^0(\sheaf{p}(3)) = 19.
        \end{equation} 
        Therefore the universal incidence correspondence has relative dimension $\alpha-\beta=3$ over the space of cubic 5-folds. We deduce that $\dim \sI_0 = 3$. The same count with the first plane fixed shows that the projection $\operatorname{pr}_2\colon \sI_0 \to \cF_2(Y)$ is dominant. Since the projection is proper, it is surjective. Since the projection has positive-dimensional fibres, whereas $\Delta_{\cF}$ intersects a fibre at a single point, $\sI_0 \setminus \Delta_{\cF} \to \cF_2(Y)$ is still surjective.

        \item By a dimension count, for a general $\varPi \in \cF_2(Y)$, $\sI_0(\varPi) \cong \operatorname{pr}_2^{-1}(\varPi)$ is a fibre of dimension 1; then by upper semi-continuity we have $\dim \sI_0(\varPi) \geq 1$ for all $\varPi \in \cF_2(Y)$. But since every plane $\varPi$ has another plane $\varPi'$ disjoint from it, $\sI_0(\varPi)$ is a proper closed subset of $\cF_2(Y)$, which is an irreducible surface. Therefore every $\sI_0(\varPi)$ is of dimension 1. \qedhere
    \end{enumerate}
    
\end{proof}
\begin{corollary}[Avoiding finitely many planes][cor:avoid_planes]
    Let $Y$ be a general cubic 5-fold. Let $S = (\varPi_1,...,\varPi_m) \subset \cF_2(Y)$ be a finite subset and let $U \subset \cF_2(Y)$ be a dense subset. There exists some $\varPi \in U \setminus S$ such that $\varPi \cap \varPi_i = \varnothing$ for $i \in \{1,...,m\}$. 
\end{corollary}
\begin{proof}
    Simply take $\varPi \in U \setminus \bigcup_{i=1}^m \sI_0(\varPi_i)$, which is non-empty by \cref{lem:plane_incid}.
\end{proof}

Next we consider the stratum $\sI_1 \subset \sI_0$. Suppose that $\varPi, \varPi' \subset Y$ are two planes such that $\varPi \cap \varPi' = Z \cong \pro^1$. We may choose coordinates on $\pro^6$ such that
\begin{equation}
    \varPi = \bV(x_0,x_1,x_2,x_3), \qquad \varPi' = \bV(x_0,x_1,x_2,x_4).
\end{equation}
Then the equation of $Y$ can be written as
\begin{equation}\label{equ:plane_incid_first_type}
    f = \sum_{i=0}^2 x_i\bar{q}_i(x_5,x_6) + \sum_{i=0}^2\sum_{j=i}^4 x_ix_j\ell_{ij}(x_5,x_6) + x_3x_4\ell_{34}(x_5,x_6) + c(x_0,...,x_4),
\end{equation}
where $\bar{q}_i$ are quadratic forms, $\ell_{ij}$ linear forms and $c$ a cubic form. 
We say that\footnote{This terminology is used only in this paper.} the intersection of $(\varPi,\ \varPi')$ is 
\begin{itemize}[dense]
    \item \textbf{of the first type}, if the quadrics $\bar q_0,\bar q_1,\bar q_2$ are linearly dependent;
    \item \textbf{of the second type}, if the quadrics $\bar q_0,\bar q_1,\bar q_2$ are linearly independent.
\end{itemize}
Denote these loci in $\cF_2(Y) \times \cF_2(Y)$ by $\sI_1^{(1)}$ and $\sI_1^{(2)}$ respectively. Then $\sI_1 = \sI_1^{(1)} \sqcup \sI_1^{(2)} \sqcup \Delta_{\cF}$. In the proof of \cref{lem:ext1_P_F} we will also characterize these two types of line-intersection of planes by the excess normal bundle $\cM_Z$.

By the discussion in \cite[Proposition 7]{IM08}, for a general $\varPi \in \cF_2(Y)$, $\sI_1(\varPi) \subset \cF_2(Y)$ is finite and has exactly $\delta = 47061$ elements, corresponding to $\delta$ nodes in the image $r_1(\cF_2(Y)) \subset \cF_1(X)$ in \cref{cor:Lag_imm}. The following fact is interesting to note but we will not attempt to prove here\footnote{In fact, this can be deduced from \cref{lem:ext1_P_F} and the upper semi-continuity of the dimension of Ext groups.}: $\sI_1^{(1)}(\varPi)$ lies in the closure of $\sI_0(\varPi) \setminus \sI_1(\varPi)$, which is a (union of) reduced curve(s) in $\cF_2(Y)$, whereas $\sI_1^{(2)}(\varPi)$ does not. In other words, $\sI_0(\varPi) \subset \cF_2(Y)$ consists of the union of reduced curves and finitely many isolated points. At any of these isolated points, any infinitesimal deformation of two planes $\varPi$ and $\varPi'$ intersecting in a line (of the second type) pulls the two planes apart directly.

\section{Higher-dimensional moduli spaces}\label{sec:high_moduli}
Readers are reminded that we will only work with the stability condition $\sigma_Y$ constructed in \cref{cor:gl_dim}, which satisfies that $\mathsf{S}_{\Ku(Y)}\cdot \sigma_Y = \sigma_Y \cdot \left[\frac{7}{3}\right]$.
Note that the central charge $Z_{\sigma_Y}$ of $\sigma_Y$ maps the lattice $\Knum(\Ku(Y))$ to the standard hexagonal lattice $\varLambda \subset \C$, as shown in \cref{fig:hexa_central_charge}. In particular, we have $Z_{\sigma_Y}(\kappa_1) = 1$ and $Z_{\sigma_Y}(\kappa_2) = \e^{\ii\pi/3}$. We have $\phi_{\sigma_Y}(\PPP_{\varPi}) - \phi_{\sigma_Y}(\FFF_{\varPi'}) = \frac{1}{3}$.

For $v \in \Knum(\Ku(Y))$, let $\phi(v) \in \R/2\Z$ denote the phase of an object in $\Modss(\Ku(Y),v)$. For bookkeeping, we write $v[n]$ for the formal pair consisting of $v$ and an integer $n$, and set $\phi(v[n])\coloneqq\phi(v)+n\in\R$ after choosing a lift of $\phi(v)$. 

For $v = a\kappa_1 + b\kappa_2$ and $w = c\kappa_1 + d\kappa_2$ in $\varLambda$, we define the wedge product $v \wedge w \coloneqq ad - bc \in \Z$. By Pick's theorem, for primitive vectors $v , w \in \varLambda$, $|v \wedge w| = 1$ if and only if the parallelogram spanned by $v$ and $w$ contains no other lattice points in its interior.

\subsection{An extension lemma}

We present a technical computation of the extension groups between two objects of adjacent $(-1)$-classes. This will be used in \cref{subsec:nonempty,subsec:Hom_van}.

\begin{lemma}[Extensions for $(-1)$-classes][lem:ext1_P_F]
    Let $j\colon \varPi \hookrightarrow Y$ and $j'\colon \varPi' \hookrightarrow Y$ be two planes. Consider the objects $\FFF_{\varPi'},\  \PPP_{\varPi} \in \Ku(Y)$ associated with $\varPi$ and $\varPi'$ as defined in \cref{def:Ku_objs}. Then: 
    \[ \Ext^\bullet(\PPP_{\varPi}, \FFF_{\varPi'}) \cong \begin{cases}
        \C^3[-1] \oplus \C^3[-2], & \varPi = \varPi'; \\ 
        \C\phantom{{}^3}[-1] \oplus \C\phantom{{}^3}[-2], & \varPi \cap \varPi' = \{p\}, \text{ or }\varPi \cap \varPi' \cong \pro^1 \text{ is of the first type}; \\
        0, & \text{otherwise.}
    \end{cases} \]
\end{lemma}
To compute this we need the following observation:
\begin{lemma}[][lem:ext1_aux]
    $\Ext^\bullet(\PPP_{\varPi}, \FFF_{\varPi'}) \cong \Ext^\bullet(j_*\sheaf{\varPi}, j'_*\sheaf{\varPi'}(-2)[1])$.
\end{lemma}
\begin{proof}
    Since $\PPP_{\varPi} = \mathsf{S}_{\Ku(Y)}\FFF_{\varPi}[-2]$ and $\FFF_{\varPi} = \mathsf{S}_{\Ku(Y)}\KKK_{\varPi}[-2]$ (see \cref{rmk:S&O:PFK}), we have that 
    \[ \Ext^\bullet(\PPP_{\varPi}, \FFF_{\varPi'}) \cong \Ext^\bullet(\FFF_{\varPi}, \KKK_{\varPi'}). \]
    Using the definition $\KKK_{\varPi'} \coloneqq \rmut{\sheaf{Y}(-1)}\rmut{\sheaf{Y}(-2)}(\ideal_{\varPi'}(-1))$ and the fact that $\FFF_{\varPi} \in {}^{\perp}\!\ord{\sheaf{Y}(-2), \sheaf{Y}(-1)}$, we obtain
    \[ \Ext^\bullet(\FFF_{\varPi}, \KKK_{\varPi'}) \cong \Ext^\bullet(\FFF_{\varPi}, \ideal_{\varPi'}(-1)). \]
    The defining sequence \eqref{seq:F_Pi_def}, together with $\Ext^\bullet(\sheaf{Y},\ideal_{\varPi'}(-1)) = 0$, gives
    \[ \Ext^\bullet(\FFF_{\varPi}, \ideal_{\varPi'}(-1)) \cong \Ext^\bullet(\ideal_{\varPi}(1)[-1], \ideal_{\varPi'}(-1)) \cong \Ext^\bullet(\ideal_{\varPi}, \ideal_{\varPi'}(-2)[1]). \]
    Similarly, since $\Ext^\bullet(\sheaf{Y},\ideal_{\varPi'}(-2)) = 0$, it is isomorphic to $\Ext^\bullet(j_*\sheaf{\varPi}, \ideal_{\varPi'}(-2)[2])$. Finally, Grothendieck--Verdier duality gives $j^!\sheaf{Y}(-2)\cong\sheaf{\varPi}(-1)[-3]$, because $\det\Nor{\varPi}{Y}\cong\sheaf{\varPi}(1)$. Hence
    \[ \Ext^\bullet(j_*\sheaf{\varPi},\sheaf{Y}(-2)) \cong \Ext^\bullet(\sheaf{\varPi},\sheaf{\varPi}(-1)[-3])=0. \]
    We conclude that $\Ext^\bullet(\PPP_{\varPi}, \FFF_{\varPi'}) \cong \Ext^\bullet(j_*\sheaf{\varPi}, \ideal_{\varPi'}(-2)[2]) \cong \Ext^\bullet(j_*\sheaf{\varPi}, j'_*\sheaf{\varPi'}(-2)[1]).$
\end{proof}

\begin{proof}[Proof of \cref{lem:ext1_P_F}]
        Note that $\FFF_{\varPi}$ and $\PPP_{\varPi}$ are $\sigma_Y$-stable and $\phi_{\sigma_Y}(\PPP_{\varPi}) - \phi_{\sigma_Y}(\FFF_{\varPi'}) = \frac{1}{3}$. Hence if $i  \leq 0$ then $\phi_{\sigma_Y}(\PPP_{\varPi}) > \phi_{\sigma_Y}(\FFF_{\varPi'}[i])$ and $\Hom(\PPP_{\varPi}, \FFF_{\varPi'}[i]) = 0$. Moreover, we have $\operatorname{gl.dim}(\sigma_Y) = \frac{7}{3}$. Hence $\Hom(\PPP_{\varPi},\FFF_{\varPi'}[i]) = 0$ for $i \geq 3$. We deduce that the only possible non-zero terms $\Ext^i(\PPP_{\varPi}, \FFF_{\varPi'})$ are $i = 1,2$. Moreover, since $\chi(\PPP_{\varPi}, \FFF_{\varPi'}) = \chi(\kappa_2,\kappa_1) = 0$, we have $\Ext^1(\PPP_{\varPi}, \FFF_{\varPi'}) = \Ext^2(\PPP_{\varPi}, \FFF_{\varPi'})$.

    By \cref{lem:ext1_aux}, $\Ext^\bullet(\PPP_{\varPi}, \FFF_{\varPi'}) \cong \Ext^\bullet(j_*\sheaf{\varPi}, j'_*\sheaf{\varPi'}(-2)[1])$. Consequently, $\Ext^i(j_*\sheaf{\varPi}, j'_*\sheaf{\varPi'}(-2)) = 0$ for $i \ne 2,3$, and the dimensions of the groups in degrees $2$ and $3$ agree. The proof is thus reduced to case-by-case local computations.
    \begin{enumerate}
        \item Assume that $\varPi \cap \varPi' = \varnothing$. Then $\Ext^\bullet(j_*\sheaf{\varPi}, j'_*\sheaf{\varPi'}(-2)) = 0$, since these groups are supported on $\varPi \cap \varPi'$.

        \item Assume that $\varPi = \varPi'$. To compute $\Ext^2(j_*\sheaf{\varPi}, j_*\sheaf{\varPi}(-2))$ we use the local-to-global spectral sequence 
        \begin{equation}
        E_2^{p,q} \coloneqq \Hlg^p{\left( \varPi, {\textstyle \bigwedge}^q \Nor{\varPi}{Y}(-2) \right)} \implies \Ext^{p+q}(j_*\sheaf{\varPi}, j_*\sheaf{\varPi}(-2)).
        \end{equation}
        To compute $\Hlg^p{\left( \varPi, {\textstyle \bigwedge}^q \Nor{\varPi}{Y}(-2) \right)}$ we use the sequences \eqref{seq:normal} and \eqref{seq:conormal}. The $E_2$-page is given by 
        \[
        \begin{tabular}{|ccc}
            {0}  & 0     & 0 \\
            {0}   & {$\C^3$}    & 0 \\
            {0} & {$\C^3$} & 0 \\
            {0}    & 0    & 0 \\ \hline
        \end{tabular}\]
        The spectral sequence degenerates at this page and hence we have $\Ext^\bullet(j_*\sheaf{\varPi}, j'_*\sheaf{\varPi'}(-2)) \cong \C^3[-2] \oplus \C^3[-3]$.

        \item Assume that $0 \leq \dim(\varPi \cap \varPi') \leq 1$. By Grothendieck--Verdier duality, we have that 
        \begin{align}
            \Ext^\bullet(j_*\sheaf{\varPi}, j'_*\sheaf{\varPi'}(-2)) 
            &\cong \Ext^\bullet(\sheaf{\varPi}, j^!j'_*\sheaf{\varPi'}(-2)) \\
            &\cong \Ext^\bullet(\sheaf{\varPi}, j^*j'_*\sheaf{\varPi'}(-1)[-3]) \\
            &\cong \mathbb{H}^\bullet(\sheaf{\varPi} \otimes^{\mathsf{L}} \sheaf{\varPi'}(-1)[-3]).
        \end{align}
                We use the Tor hyper-cohomology sequence:
        \[ F_2^{p,q} \coloneqq \Hlg^p(\STor_q(\sheaf{\varPi}, \sheaf{\varPi'}(-1))) \implies \mathbb{H}^{p-q}(\sheaf{\varPi} \otimes^{\mathsf{L}} \sheaf{\varPi'}(-1)). \]

        Let $\delta \coloneqq \dim(\varPi \cap \varPi') - \dim \varPi - \dim \varPi' + \dim Y$ be the excess dimension of the intersection. We have that $\STor_q(\sheaf{\varPi}, \sheaf{\varPi'}(-1)) = 0$ for $q < 0$ or $q > \delta$.

        Let $Z \coloneqq \varPi \cap \varPi'$. If $Z$ is a point, then $\delta = 1$. Since $\STor_q(\sheaf{\varPi}, \sheaf{\varPi'}(-1))$ is supported on $Z$, we have that $F_2^{p,q} = 0$ for $p \ne 0$. Hence the only contribution to $\mathbb{H}^{0}(\sheaf{\varPi} \otimes^{\mathsf{L}} \sheaf{\varPi'}(-1))$ is $F^{0,0}_2$, which is $\Hlg^{0}(\STor_0(\sheaf{\varPi},\sheaf{\varPi'}(-1))) \cong \Hlg^0(\sheaf{Z}(-1)) \cong \C$. 

        If $Z$ is a line, then $\delta = 2$. Although we will not be using the case in subsequent computations, we include it here for completeness. The term $F^{p,q}_2$ can be non-zero only if $0 \leq p \leq 1$ and $0 \leq q \leq 2$. It follows that the spectral sequence degenerates at the $F_2$-page and we have that 
        \[\mathbb{H}^0(\sheaf{\varPi} \otimes^{\mathsf{L}} \sheaf{\varPi'}(-1)) \cong \Hlg^0(\sheaf{Z}(-1)) \oplus \Hlg^1(\STor_1(\sheaf{\varPi},\sheaf{\varPi'}(-1))).\]
        It is clear that $\Hlg^0(\sheaf{Z}(-1)) = 0$. For the second term, $\STor_1(\sheaf{\varPi},\sheaf{\varPi'}(-1)) \cong \cM_Z^{\dual}(-1)$, where $\cM_Z$ is the excess normal bundle (see \cite[\S 6.3]{FultonInt}) defined by the short exact sequence:
        \[\begin{tikzcd}[ampersand replacement=\&,cramped]
            0 \& {\Nor{Z}{\varPi}} \& {\Nor{\varPi'}{Y}\big|_{Z}} \& {\MMM_Z} \& 0.
            \arrow[from=1-1, to=1-2]
            \arrow[from=1-2, to=1-3]
            \arrow[from=1-3, to=1-4]
            \arrow[from=1-4, to=1-5]
        \end{tikzcd}\]
        We have $\Nor{Z}{\varPi} \cong \sheaf{Z}(1)$ and $\Nor{\varPi'}{Y}\big|_{Z} = \ker{\left(\sheaf{Z}(1)^{\oplus 4} \twoheadrightarrow \sheaf{Z}(3)\right)}$ (see \eqref{seq:normal}). Hence $\cM_Z$ fits into the short exact sequence:
        \[\begin{tikzcd}[ampersand replacement=\&,cramped]
            0 \& {\cM_Z} \& {\sheaf{Z}(1)^{\oplus 3}} \&\& {\sheaf{Z}(3)} \& 0,
            \arrow[from=1-1, to=1-2]
            \arrow[from=1-2, to=1-3]
            \arrow["{(\bar q_0,\bar q_1,\bar q_2)}", from=1-3, to=1-5]
            \arrow[from=1-5, to=1-6]
        \end{tikzcd}\]
        where $\bar q_0,\bar q_1,\bar q_2$ are the restricted quadratic forms on $Z$ defined after \eqref{equ:plane_incid_first_type}.
        Since $Z \cong \pro^1$, the bundle $\cM_Z$ splits as a direct sum of line bundles. If $\varPi \cap \varPi' = Z$ is of the second type, that is, $\bar q_0,\bar q_1,\bar q_2$ are linearly independent, then $\cM_Z \cong \sheaf{Z}^{\oplus 2}$. Hence
        \[ \mathbb{H}^0(\sheaf{\varPi} \otimes^{\mathsf{L}} \sheaf{\varPi'}(-1)) \cong \Hlg^1(\sheaf{Z}(-1)^{\oplus 2}) = 0. \]
        On the other hand, if $\varPi \cap \varPi' = Z$ is of the first type, then a change of coordinates kills $\bar q_0$ and we have $\cM_Z \cong \sheaf{Z}(1) \oplus \sheaf{Z}(-1)$. In this case we have 
        \[ \mathbb{H}^0(\sheaf{\varPi} \otimes^{\mathsf{L}} \sheaf{\varPi'}(-1)) \cong \Hlg^1(\sheaf{Z} \oplus \sheaf{Z}(-2)) \cong \C, \]
        which finishes the proof. \qedhere
                                    \end{enumerate}
\end{proof}

\subsection{Non-emptiness of moduli spaces}\label{subsec:nonempty}

In this part we establish the non-emptiness of all moduli spaces on $\Ku(Y)$. 

\begin{lemma}[{\cite[Lemma 2.12]{LLPZ}}][lem:Ext1_nonzero]
    Let $v,w \in \Knum(\Ku(Y))$ be primitive characters with $v \wedge w = 1$. Let $E \in \Mods(\Ku(Y),v)$ and $F \in \Mods(\Ku(Y),w)$ such that $0 < \phi_{\sigma_Y}(F) - \phi_{\sigma_Y}(E) < 1$ and $\Ext^1(F,E) \ne 0$. Then a non-trivial extension $f \in \Ext^1(F,E)$ defines a $\sigma_Y$-stable object $G$. In particular, we have $\Mods(\Ku(Y),v+w) \ne \varnothing$.
\end{lemma}
\begin{proof}
    This is in fact a special case of \cite[Lemma 2.12]{LLPZ}. We adapt the proof here for completeness. First note that $v+w$ is a primitive character. Hence it suffices to show that there is some $\sigma_Y$-semistable $G \in \Ku(Y)$ with character $v+w$. Since $\Ext^1(F,E) \ne 0$, we pick a non-trivial extension $G$ corresponding to $f \in \Hom(F,E[1]) \setminus \{0\}$:
    \begin{equation}
        \begin{tikzcd}[ampersand replacement=\&,cramped]
            E \& G \& F \& {E[1]}
            \arrow[from=1-1, to=1-2]
            \arrow[from=1-2, to=1-3]
            \arrow["f", from=1-3, to=1-4]
        \end{tikzcd} \label{seq:ext_F_E}
    \end{equation}
    Suppose that $G$ is not $\sigma_Y$-semistable. Since $E$, $F$ are $\sigma_Y$-stable, we have
    \[\phi_{\sigma_Y}(E)\leq\phi^-_{\sigma_Y}(G)<\phi^+_{\sigma_Y}(G)\leq\phi_{\sigma_Y}(F).\]
    In particular, every Harder--Narasimhan factor $G_i$ of $G$ has character $a_iv+b_iw$ for $a_i,b_i \geq 0$ and $a_i+b_i \leq 1$. Since $v \wedge w = 1$, it follows that $G$ has exactly two HN factors $G_+$ and $G_-$ with characters $w$ and $v$ respectively. Since $v,w$ are primitive, $G_+$ and $G_-$ are $\sigma_Y$-stable. In fact, $G_+ \cong F$ and $G_- \cong E$. It follows that the Harder--Narasimhan filtration of $G$ is
    \[\begin{tikzcd}[ampersand replacement=\&,cramped]
            0 \& {G_+} \& G \& {G/G_+ = G_-} \& 0.
            \arrow[from=1-1, to=1-2]
            \arrow[from=1-2, to=1-3]
            \arrow[from=1-3, to=1-4]
            \arrow[from=1-4, to=1-5]
        \end{tikzcd}\]
    provides a splitting of the sequence \eqref{seq:ext_F_E}, which implies that $f = 0$. This is a contradiction.
\end{proof}

Building on the existence of $(-1)$-classes (\cref{cor:F&K_stable}), we next construct explicit non-trivial extensions realizing the $(-3)$-classes, and then apply the general lattice-theoretic criterion of \cite{LLPZ}. Note that $|\kappa_1 + \kappa_2|^2 = 3$ and every $(-3)$-class lies in the same Serre orbit as $\kappa_1 + \kappa_2$. It is enough to show the following:

\begin{proposition}[][prop:moduli_v+w]
    Consider the character $\kappa_1 + \kappa_2 \in \Knum(\Ku(Y))$ with Chern character $\Chern{}(\kappa_1+\kappa_2) = 3 - H^2 + \frac{1}{4}H^4$ (see \cref{prop:Knum_Ku}). The moduli space of $\sigma_Y$-stable objects $\Mods(\Ku(Y), \kappa_1+\kappa_2)$ is non-empty.
\end{proposition}
\begin{proof}
    Let $\varPi \subset Y$ be a plane. By \cref{prop:Knum_Ku} and \cref{cor:F&K_stable}, we have $\FFF_{\varPi} \in \Mods(\Ku(Y),\kappa_1)$ and $\PPP_{\varPi} \in \Mods(\Ku(Y),\kappa_2)$, and by \cref{lem:ext1_P_F}, $\Ext^1(\PPP_{\varPi}, \FFF_{\varPi}) \ne 0$. Pick a non-trivial extension $f \in \Ext^1(\PPP_{\varPi}, \FFF_{\varPi})$. Then \cref{lem:Ext1_nonzero} shows that $f$ defines a $\sigma_Y$-stable object of character $\kappa_1+\kappa_2$.
\end{proof}

We are ready to prove the general non-emptiness result:

\begin{theorem}[][thm:moduli_nonempty]
    Let $Y$ be a smooth cubic 5-fold. Let $\sigma_Y$ be a stability condition on $\Ku(Y)$ constructed in \cref{cor:gl_dim}. For any non-zero numerical character $v \in \Knum(\Ku(Y))$, the moduli space $\Modss(\Ku(Y),v)$ of $\sigma_Y$-semistable objects of character $v$ is non-empty.
\end{theorem}
\begin{proof}
    The strategy of proof is similar to \cite[Proposition 2.14]{LLPZ}. It suffices to show that $\Mods(\Ku(Y),v)$ is non-empty for all primitive character $v \in \Knum(\Ku(Y))$. Recall that $\varLambda \coloneqq Z_{\sigma_Y}(\Knum(\Ku(Y))) \subset \C$ is the standard hexagonal lattice.
            We denote by $\arg(v,w)$ the angle between $Z_{\sigma_Y}(v)$ and $Z_{\sigma_Y}(w)$ in $\C$; that is, $\arg(v,w) = \pi(\phi(w) - \phi(v))$. In addition, we define the norm $|v|^2 \coloneqq -\chi(v,v)$. By Pick's theorem, for every primitive $v \in \varLambda$ with $|v|^2 > 1$ there exists a unique pair of primitive vectors $v_{\pm} \in \varLambda$ such that $v = v_+ + v_-$, $|v|^2 > |v_{\pm}|^2$, and $v_- \wedge v_+ = 1$.
    
    We prove $\Mods(\Ku(Y),v) \ne \varnothing$ for primitive $v \in \Knum(\Ku(Y))$ by induction on the norm $|v|^2$. For $|v|^2 = 1$, the non-emptiness is given by \cref{cor:F&K_stable}, and for $|v|^2 = 3$ given by \cref{prop:moduli_v+w}. For $|v|^2 > 3$, we choose the pair $(v_+,v_-)$ of primitive vectors according to Pick's theorem. Note that we have $0 < \arg(v_-,v_+) < \pi/3$. By induction hypothesis we take stable objects $E_{\pm} \in \Mods(\Ku(Y),v_{\pm})$; then $0 < \phi_{\sigma_Y}(E_+) - \phi_{\sigma_Y}(E_-) < 1/3$. Since $\phi_{\sigma_Y}(\mathsf{S}_{\Ku(Y)}E_-) = \phi_{\sigma_Y}(E_-) + \frac{7}{3}$, we have $\Ext^i(E_+,E_-) = 0$ for $i \leq 0$ or $i \geq 3$. Meanwhile, for the given Euler form
    \[
        \chi = \begin{pmatrix}
            -1 & -1 \\
            0 & -1
        \end{pmatrix}
    \]
    on the lattice $\Knum(\Ku(Y))$, from the classification of non-degenerate bilinear forms on a rank 2 lattice given in \cite[Lemma A.1]{LLPZ}, for $|v|^2 > 3$, we have $\chi(v_+,v_-) < 0$, which implies that $\Ext^1(E_+,E_-) \ne 0$. By \cref{lem:Ext1_nonzero}, we deduce that $\Mods(\Ku(Y),v) \ne \varnothing$. Finally, for non-primitive $v \in \Knum(\Ku(Y))$, write $v = mw$ where $w$ is primitive and $m \in \Z_{>0}$. Take $E \in \Mods(\Ku(Y),w)$, and then $E^{\oplus m} \in \Modss(\Ku(Y),v)$. The non-emptiness is proved.
\end{proof}

\begin{remark}[][rmk:pick_decomp]
    For convenience, we will call the pair $(v_+,v_-)$ in the proof of \cref{thm:moduli_nonempty} the \textbf{Pick's decomposition} of the vector $v \in \Z^2$. If $v = mw$ is not primitive, then any $(m_1w, m_2w)$ with $m_1,m_2 \in \Z_{>0}$ and $m_1+m_2 = m$ is also referred to as \emph{a} Pick's decomposition of $v$.
\end{remark}

\subsection{Generic Hom-vanishing results}\label{subsec:Hom_van}

For subsequent discussion, we introduce the following notions.

\begin{definition}
    For $v,w \in \Knum(\Ku(Y))$ we consider the \textbf{non-Brill--Noether locus}, consisting of pairs of objects with vanishing Hom:
    \[\ZZZs(v,w) \coloneqq \left\{ (E_v,E_w)  \mid \Hom(E_v,E_w) = 0 \right\} \subset \Mods(\Ku(Y),v) \times \Mods(\Ku(Y),w).\]
    By upper semi-continuity, $\ZZZs(v,w)$ is open in $\Mods(\Ku(Y),v) \times \Mods(\Ku(Y),w)$. Let $\pi_1\colon \ZZZs(v,w) \to \Mods(\Ku(Y),v)$ and $\pi_2\colon \ZZZs(v,w) \to \Mods(\Ku(Y),w)$ be the corresponding projection maps.
\end{definition}

For convenience we will also use the following notions: for $E_v \in \Mods(\Ku(Y),v)$ and $E_w \in \Mods(\Ku(Y),w)$,
\begin{equation}
    \begin{aligned}
         \ZZZs(v,E_w) &\coloneqq \left\{ E_v \in \Mods(\Ku(Y),v) \mid \Hom(E_v,E_w) = 0 \right\}; \\ 
        \ZZZs(E_v,w) &\coloneqq \left\{ E_w \in \Mods(\Ku(Y),w) \mid \Hom(E_v,E_w) = 0 \right\}.
    \end{aligned}
\end{equation}

\newcommand{\good}{good}
\begin{definition}
    Let $v \in \Knum(\Ku(Y))$ and $E \in \Mods(\Ku(Y),v)$. 
    \begin{itemize}[nosep, before = \vspace*{-\parskip}]
        \item If $\chi(v,v) = -1$, then we say that $E$ is \textbf{\good} if $E$ is isomorphic to any of $\FFF_{\varPi}$, $\PPP_{\varPi}$, $\KKK_{\varPi}$ or their shifts for some 2-plane $\varPi \subset Y$;
        \item If $\chi(v,v) < -1$, then we say that $E$ is \textbf{\good} if there exists an extension $0 \to E_- \to E \to E_+ \to 0$, where $E_{\pm} \in \Mods(\Ku(Y), v_{\pm})$ are {\good} objects, and $(v_+,v_-)$ is the Pick's decomposition of $v$. The locus of {\good} objects in $\Mods(\Ku(Y),v)$ is denoted by $\Modss^{\circ}(\Ku(Y),v)$.
    \end{itemize}
\end{definition}

\begin{definition}
    For a good object $E \in \Modss^{\circ}(\Ku(Y),v)$, we associate a finite subset $\operatorname{pl}(E) \subset \cF_2(Y)$ of planes by the following rules:
    \begin{itemize}[nosep, before = \vspace*{-\parskip}]
        \item If $E$ is isomorphic to any of $\FFF_{\varPi}$, $\PPP_{\varPi}$, $\KKK_{\varPi}$ or their shifts for some 2-plane $\varPi \subset Y$, then $\operatorname{pl}(E) \coloneqq \{\varPi\}$;
        \item If $E$ is obtained by the extension $0 \to E_- \to E \to E_+ \to 0$, where $E_{\pm} \in \Modss^{\circ}(\Ku(Y), v_{\pm})$, and $(v_+,v_-)$ is the Pick's decomposition of $v$, then $\operatorname{pl}(E) \coloneqq \operatorname{pl}(E_+) \cup \operatorname{pl}(E_-)$.
    \end{itemize}
\end{definition}
\begin{remark}
    Note that the set of planes $\operatorname{pl}(E)$ associated to $E$ \emph{may not be unique}. In the following statements, when we write $\operatorname{pl}(E)$ we mean any choice of such a set.
\end{remark}
\begin{remark}
    It is clear from the action of the Serre functor $\mathsf{S}_{\Ku(Y)}$ that $E$ is {\good} if and only if $\mathsf{S}_{\Ku(Y)}E$ is {\good}, and that $\operatorname{pl}(E) = \operatorname{pl}(\mathsf{S}_{\Ku(Y)}(E))$.
\end{remark}

We state the first result concerning the non-emptiness of the non-Brill--Noether loci.
\begin{proposition}[Generic Hom-vanishing in the same or adjacent sextants][prop:phase_Hom_c]
    For a numerical class $v \in \Knum(\Ku(Y))$, we say that $v$ lies in the \textbf{$\bm n$-th sextant} if $\phi(\kappa_1) + \frac{n-1}{3} \leq \phi(v) \leq \phi(\kappa_1) + \frac{n}{3}$, where $1 \leq n \leq 6$.\footnote{By convention, the boundary ``axes'' $\phi_n = n/3$ between two sextants are included in both of them. This is only to simplify the statements and proofs in this subsection.} 
     
    Let $v,w \in \Knum(\Ku(Y))$ be primitive characters where $v$ lies in the $n$-th sextant, $w$ in the $m$-th sextant, and $0 \leq m-n \leq 1$. For $E_v \in \Modss^{\circ}(\Ku(Y),v)$ and $E_w \in \Modss^{\circ}(\Ku(Y),w)$, we have $\Hom(E_v,E_w) = 0$ if $\varPi_v \cap \varPi_w =  \varnothing$ for every $\varPi_v \in \operatorname{pl}(E_v)$ and $\varPi_w \in \operatorname{pl}(E_w)$. In particular, for $Y$ general, $\ZZZs(E_v,w)$ and $\ZZZs(v,E_w)$
    are non-empty open subsets of $\Modss^{\circ}(\Ku(Y),w)$ and $\Modss^{\circ}(\Ku(Y),v)$ respectively.
\end{proposition}

The proof will be induction arguments on the lattice of $\Knum(\Ku(Y))$. As the base case, we first establish the result for $(-1)$-classes.

\begin{lemma}[][lem:phase_Hom_a]
     Let $j\colon \varPi \hookrightarrow Y$ and $j'\colon \varPi' \hookrightarrow Y$ be two planes such that $\varPi \cap \varPi' = \varnothing$. We have the following Hom-vanishing results:
     \begin{enumerate}[nosep, before = \vspace*{-\parskip}]
        \item $\Hom(\FFF_{\varPi}, \FFF_{\varPi'}) = 0$.
        \item $\Hom(\FFF_{\varPi}, \PPP_{\varPi'}) = 0$.
        \item $\Hom(\FFF_{\varPi}, \KKK_{\varPi'}[1]) = 0$.
     \end{enumerate}
\end{lemma}
\begin{proof}
    \begin{enumerate}
        \item Note that $\FFF_{\varPi}$ and $\FFF_{\varPi'}$ are $\sigma_Y$-stable with the same phase. Since $\FFF_{\varPi} \not\cong \FFF_{\varPi'}$, then $\Hom(\FFF_{\varPi}, \FFF_{\varPi'}) = 0$.
        \item Using the fact that $\PPP_{\varPi'} = \SKu \FFF_{\varPi'}[-2]$, we have $\Hom(\FFF_{\varPi}, \PPP_{\varPi'}) \cong \Ext^2(\FFF_{\varPi}, \FFF_{\varPi'})^\dual$. Then we can proceed as in the proof of \cref{lem:Ext_F_Pi}. Since $\varPi \cap \varPi' = \varnothing$, we have $\Ext^\bullet(j_*\sheaf{\varPi}, j'_*\sheaf{\varPi'}) = 0$. Using the double complex of long exact sequences \eqref{les:Ext_I_Pi} (with the target $\varPi$ replaced by $\varPi'$) and that $\Ext^\bullet(j_*\sheaf{\varPi}, \sheaf{Y}) = \C^3[-3]$, we obtain that 
        \[\Ext^\bullet(\ideal_{\varPi}, \ideal_{\varPi'}) = \C^3[-2].\]
        Next, using the double complex of long exact sequences \eqref{les:Ext_F_Pi} (with the target $\varPi$ replaced by $\varPi'$) and that $\Ext^\bullet(\ideal_{\varPi}(1), \sheaf{Y}) = \C[-2]$, we have that $\Ext^i(\FFF_{\varPi}, \FFF_{\varPi'}) = 0$ for $i \ne 1,2$ and the long exact sequence
        \[\begin{tikzcd}[ampersand replacement=\&,cramped, column sep = small]
            0 \& {\Ext^{1}(\FFF_{\varPi},\FFF_{\varPi'})} \& {\Ext^{2}(\ideal_{\varPi}(1),\sheaf{Y}^{\oplus 4})} \& {\Ext^{2}(\ideal_{\varPi}(1),\ideal_{\varPi'}(1))} \& {\Ext^{2}(\FFF_{\varPi},\FFF_{\varPi'})} \& {0.}
            \arrow[from=1-1, to=1-2]
            \arrow[from=1-2, to=1-3]
            \arrow["\delta", from=1-3, to=1-4]
            \arrow[from=1-4, to=1-5]
            \arrow[from=1-5, to=1-6]
        \end{tikzcd}\]
        The map $\delta\colon \Ext^{2}(\ideal_{\varPi}(1),\sheaf{Y}^{\oplus 4}) \to \Ext^{2}(\ideal_{\varPi}(1),\ideal_{\varPi'}(1))$ is given by post-composition with the evaluation map $\sheaf{Y}^{\oplus 4} \to \ideal_{\varPi'}(1)$. Note that $\delta$ is the same as $\Ext^{3}(\sheaf{\varPi}(1),\sheaf{Y}^{\oplus 4}) \to \Ext^{3}(\sheaf{\varPi}(1),\ideal_{\varPi'}(1))$, which by Grothendieck--Verdier duality is isomorphic to the map $\Hlg^0(\sheaf{\varPi}^{\oplus 4}) \to \Hlg^0(\ideal_{\varPi'}(1)|_{\varPi})$. Since $\varPi \cap \varPi' = \varnothing$, then $\ideal_{\varPi'}(1)|_{\varPi} = \sheaf{\varPi}(1)$, and $\Hlg^0(\sheaf{\varPi}^{\oplus 4}) \to \Hlg^0(\sheaf{\varPi}(1))$ is simply given by $\operatorname{ev} \oplus \mathop{0}$, which comes from the restricting the global sections of $\ideal_{\varPi'}(1)$ to $\varPi$. In particular $\delta$ is surjective with one-dimensional kernel. Hence we have 
        \[ \Ext^\bullet(\FFF_{\varPi}, \FFF_{\varPi'}) = \C[-1], \]
        which proves the claim.

        \item By \cref{lem:ext1_P_F} we have 
        \[\Hom(\FFF_{\varPi}, \KKK_{\varPi'}[1]) = \Hom(\mathsf{S}_{\Ku(Y)}\FFF_{\varPi}, \mathsf{S}_{\Ku(Y)}\KKK_{\varPi'}[1]) =\Ext^1(\PPP_{\varPi}, \FFF_{\varPi'}) = 0. \qedhere\]
    \end{enumerate}
\end{proof}
\begin{proof}[Proof of \cref{prop:phase_Hom_c}.]
    We prove the statement by induction first on $|w|^2 = -\chi(w,w)$ and then on $|v|^2 = -\chi(v,v)$. For the base case, suppose that $v,w$ are both $(-1)$-classes. 
    
    Then the statement is given by \cref{lem:phase_Hom_a}. After applying the rotation functor, we may assume that $v = \kappa_1$. Then by assumption $w$ is one of $\kappa_1$, $\kappa_2$, and $\kappa_2 - \kappa_1$. The {\good} objects in $\Modss^{\circ}(\Ku(Y),\kappa_1)$, $\Modss^{\circ}(\Ku(Y),\kappa_2)$ and $\Modss^{\circ}(\Ku(Y),\kappa_2 - \kappa_1)$ are of the form $\FFF_{\varPi}$, $\PPP_{\varPi}$ and $\KKK_{\varPi}[1]$ respectively. The Hom-vanishing results then follow from \cref{lem:phase_Hom_a}.

    Next we fix $v = \kappa_1$, $E_v = \FFF_{\varPi} \in \Modss^{\circ}(\Ku(Y), v)$, and prove the statement for general $w$ by induction on $|w|^2$. For $|w|^2 > 1$, we pick a Pick's decomposition $(w_+,w_-)$ of $w$, and $w_{\pm}$ lie in the same sextant as $w$. For {\good} objects $E_w \in \Modss^{\circ}(\Ku(Y),w)$, take a corresponding extension $0 \to E_- \to E_w \to E_+ \to 0$, where $E_{\pm} \in \Modss^{\circ}(\Ku(Y),w_{\pm})$, and $\operatorname{pl}(E_w) = \operatorname{pl}(E_+) \cup \operatorname{pl}(E_-)$. By assumption we have $\varPi \cap \varPi' = \varnothing$ for any $\varPi' \in \operatorname{pl}(E_\pm)$. The induction hypothesis implies that $\Hom(E_v, E_\pm) = 0$. Applying $\Hom(E_v,-)$ to the extension sequence, we have $\Hom(E_v, E_w) = 0$, which proves the claim.

    Finally, we fix any $w$ and prove the statement for general $v$ by induction on $|v|^2$. The proof is similar to above.

    For the non-emptiness of $\ZZZs(E_v,w)$ and $\ZZZs(v,E_w)$, since $Y$ is general, by \cref{cor:avoid_planes}, the base planes used in the inductive construction of a good object $E_w$ can be chosen in $\cF_2(Y)$ and disjoint from every plane in $\operatorname{pl}(E_v)$. The vanishing just proved then shows that $E_w\in\ZZZs(E_v,w)$. The proof for $\ZZZs(v,E_w)$ is symmetric.
\end{proof}

\subsection{Jumping loci in 4-dimensional moduli spaces}

\begin{lemma}[][lem:k1+k2_ext2]
    Let $Y$ be a general cubic 5-fold. Let $\varPi, \varPi' \subset Y$ be general 2-planes in $Y$ such that $\varPi \cap \varPi' = \{p\}$ is a point. Let $\EEE_{\eta} \in \Modss(\Ku(Y), \kappa_1 + \kappa_2)$ be the object associated to a non-trivial extension $\eta \in \Ext^1(\PPP_{\varPi},\FFF_{\varPi'})$. Then the map $(-\circ\eta)\colon \Ext^1(\PPP_{\varPi}, \PPP_{\varPi}) \to \Ext^2(\PPP_{\varPi}, \FFF_{\varPi'})$ is surjective. In particular, $\Ext^2 (\EEE_{\eta}, \EEE_{\eta}) = 0$.
\end{lemma}
\begin{proof}
    The object $\EEE_{\eta}$ is defined by the short exact sequence $0 \to \FFF_{\varPi'} \to \EEE_{\eta} \to \PPP_{\varPi} \to 0$, and we have the following double complex of long exact sequences:
    \begin{equation}
        \begin{tikzcd}[ampersand replacement=\&,cramped]
    	{\Ext^1(\PPP_{\varPi},\PPP_{\varPi})} \& {\Ext^2(\PPP_{\varPi},\FFF_{\varPi'})} \& {\Ext^2 (\PPP_{\varPi}, \EEE_{\eta})} \& {\Ext^2(\PPP_{\varPi},\PPP_{\varPi})} \\
    	\& {\Ext^2 (\EEE_{\eta}, \FFF_{\varPi'})} \& {\Ext^2 (\EEE_{\eta}, \EEE_{\eta})} \& {\Ext^2 (\EEE_{\eta}, \PPP_{\varPi})} \& 0 \\
    	\& {\Ext^2(\FFF_{\varPi'},\FFF_{\varPi'})} \& {\Ext^2 (\FFF_{\varPi'}, \EEE_{\eta})} \& {\Ext^2(\FFF_{\varPi'},\PPP_{\varPi})} \& 0 \\
    	\&\&\& 0
    	\arrow["f", from=1-1, to=1-2]
    	\arrow["g", from=1-2, to=1-3]
    	\arrow[from=1-2, to=2-2]
    	\arrow[from=1-3, to=1-4]
    	\arrow["h", from=1-3, to=2-3]
    	\arrow[from=1-4, to=2-4]
    	\arrow[from=2-2, to=2-3]
    	\arrow[from=2-2, to=3-2]
    	\arrow[from=2-3, to=2-4]
    	\arrow[from=2-3, to=3-3]
    	\arrow[from=2-4, to=2-5]
    	\arrow[from=2-4, to=3-4]
    	\arrow[from=3-2, to=3-3]
    	\arrow[from=3-3, to=3-4]
    	\arrow[from=3-4, to=3-5]
    	\arrow[from=3-4, to=4-4]
    \end{tikzcd} \label{seq:E_eta}
    \end{equation}
By \cref{lem:Ext_F_Pi}, we have $\Ext^2(\FFF_{\varPi'}, \FFF_{\varPi'}) = 0$ and $\Ext^2(\PPP_{\varPi}, \PPP_{\varPi}) = \Ext^2(\FFF_{\varPi}, \FFF_{\varPi}) = 0$. By Serre duality we also have $\Ext^2(\FFF_{\varPi'}, \PPP_{\varPi}) \cong \Hom(\PPP_{\varPi}, \PPP_{\varPi'})^{\dual} = 0$ since $\PPP_{\varPi} \not\cong \PPP_{\varPi'}$. Therefore the maps $g$ and $h$ are surjective. By diagram chase, we observe that $\Ext^{2}(\EEE_{\eta}, \EEE_{\eta}) = 0$ if $f\colon \Ext^1(\PPP_{\varPi}, \PPP_{\varPi}) \to \Ext^2(\PPP_{\varPi}, \FFF_{\varPi'})$, which is induced by taking the Yoneda product with $\eta \in \Ext^1(\PPP_{\varPi}, \FFF_{\varPi'})$, is surjective.

By \cref{lem:ext1_aux} and the proof of \cref{lem:Ext_F_Pi}, we have $\Ext^{1}(\PPP_{\varPi}, \PPP_{\varPi}) \cong \Ext^1(j_*\sheaf{\varPi},j_*\sheaf{\varPi})$, and, for $i=1,2$,
\[\Ext^i(\PPP_{\varPi}, \FFF_{\varPi'}) \cong \Ext^{i+1}(j_*\sheaf{\varPi}, j'_*\sheaf{\varPi'}(-2)) \cong \Ext^{i+1}(j_*\sheaf{\varPi}, j'_*\sheaf{\varPi'})  \cong \C.\] 
Then the map $f$ is the same as the map \[f'\colon \Ext^1(j_*\sheaf{\varPi},j_*\sheaf{\varPi}) \to \Ext^{3}(j_*\sheaf{\varPi}, j'_*\sheaf{\varPi'}),\] 
induced by taking the Yoneda product with $\eta' \in \Ext^2(j_*\sheaf{\varPi}, j'_*\sheaf{\varPi'})$. 

At a general point $(\varPi,\varPi')$ of the point-incidence stratum of $\sI_0$, the divisor $\sI_0\subset\cF_2(Y)^2$ is smooth and both projections are smooth, by generic smoothness and \cref{lem:plane_incid}. Consider the excess normal bundle $\cM_p$, which is a 1-dimensional vector space over $p$. The restriction map $\Ext^1(j_*\sheaf{\varPi},j_*\sheaf{\varPi}) \cong \Hlg^0(\Nor{\varPi}{Y}) \to \Hlg^0(\cM_p)$ is surjective; hence we can choose some $u_1 \in \Ext^1(j_*\sheaf{\varPi},j_*\sheaf{\varPi})$ whose image generates $\cM_p$. Explicitly, formal-locally at $\varPi \cap \varPi' = \{p\}$ we have $\widehat{\sheaf{Y,p}} \cong \C[\![x_1,x_2,x_3,x_4,x_5]\!]$, and locally $\varPi$ and $\varPi'$ are defined by $x_1=x_2=x_3=0$ and $x_1=x_4=x_5=0$. The local image of $u_1$ is dual to $x_1$. Then the Yoneda product $\eta' \circ u_1$ sends the generator $[\eta'] = [u_2 \wedge u_3]$ of $\Ext^{2}(j_*\sheaf{\varPi}, j'_*\sheaf{\varPi'})$ to the generator $[u_1 \wedge u_2 \wedge u_3]$ of $\Ext^{3}(j_*\sheaf{\varPi}, j'_*\sheaf{\varPi'})$. In particular, the map $f'$ is non-zero. Since it has one-dimensional target, it is surjective. This finishes the proof.
\end{proof}

\begin{lemma}[][lem:k1+k2_ext2_2]
    Let $\varPi\subset Y$ be a 2-plane in $Y$. Let $\EEE_{\eta} \in \Modss(\Ku(Y), \kappa_1 + \kappa_2)$ be the object associated to a non-trivial extension $\eta \in \Ext^1(\PPP_{\varPi},\FFF_{\varPi})$. Then $\Ext^2 (\EEE_{\eta}, \EEE_{\eta}) \ne 0$.
\end{lemma}
\begin{proof}
    Consider the double complex given as in \eqref{seq:E_eta} with $\varPi = \varPi'$. We have $\Ext^2(\FFF_{\varPi'}, \PPP_{\varPi}) \cong \Hom(\PPP_{\varPi}, \PPP_{\varPi'})^{\dual} \cong \C$. Then $\Ext^2(\EEE_{\eta}, \PPP_{\varPi}) \cong \Ext^2(\FFF_{\varPi'}, \PPP_{\varPi}) \cong \C$. Since $\Ext^2(\EEE_{\eta}, \EEE_{\eta}) \to \Ext^2(\EEE_{\eta}, \PPP_{\varPi})$ is surjective, we must have $\Ext^2 (\EEE_{\eta}, \EEE_{\eta}) \ne 0$.
\end{proof}

\begin{remark}
        \cref{lem:k1+k2_ext2,lem:k1+k2_ext2_2} show that $\ext^2(E,E)$ is non-constant on a connected component of the moduli space $\Mods(\Ku(Y), \kappa_1 + \kappa_2)$. This indicates that $\Mods(\Ku(Y), \kappa_1 + \kappa_2)$ is singular. This phenomenon is notably different from the case of cubic 3-folds, where all moduli spaces of primitive characters on the Kuznetsov component are smooth \cite[Corollary 3.9]{LLPZ}, which follows from $\Ext^2(E,E)$ vanishing for the easy reason of global dimension being smaller than $2$.
\end{remark}

\subsection{Restriction to cubic 4-folds}\label{subsec:res_to_cubic4}

In this subsection, we continue using the notation in \cref{subsec:Fano}, where $Y$ denotes a general cubic 5-fold, and $X \subset Y$ a very general hyperplane section. The stability conditions on $\Ku(X)$ and $\Ku(Y)$ are denoted by $\sigma_X$ and $\sigma_Y$ respectively. 

\begin{proposition}[Non-emptiness of smooth loci][prop:Ext2_zero_general]
    For any primitive character $v \in \Knum(\Ku(Y))$, the open locus in $\Mods(\Ku(Y),v)$ defined by $\Ext^2(E,E)=0$ is non-empty. It contains a {\good} object.
\end{proposition}

\begin{theorem}[Restriction Theorem][thm:res]
  Let $Y$ be a general cubic 5-fold, and $X \subset Y$ a very general hyperplane section. Let $i\colon X \hookrightarrow Y$ be the inclusion, and choose a stability condition $\sigma_X$ in the $\widetilde{\GL_2^+}(\R)$-orbit\footnote{The orbit is in fact unique by \cite[Corollary 3.9]{FP23}.} on $\Ku(X)$ constructed in \cite{BLMS}, normalized compatibly with the basis $\{i^*\kappa_1,i^*\kappa_2\}$. For any primitive character $v \in \Knum(\Ku(Y))$, there exists a non-empty open subset $\UUUs_Y(v) \subset \Modss(\Ku(Y),v)$ such that the pull-back object $i^*E \in \Ku(X)$ is $\sigma_X$-stable for all $E \in \UUUs_Y(v)$.
\end{theorem}

We will prove \cref{prop:Ext2_zero_general} and Restriction \cref{thm:res} simultaneously by induction on $|v|^2$. We need the following observation.

We first check that $i^*E \in \Ku(X)$ for any $E \in \Ku(Y)$.
Consider the restriction triangle:
    \begin{equation}\label{ses:res_trig_E}
                  \begin{tikzcd}[ampersand replacement=\&,cramped]
                      {E(-1)} \& {E} \& {i_* i^*E} \& {}
                      \arrow[from=1-1, to=1-2]
                      \arrow[from=1-2, to=1-3]
                      \arrow["{+1}", from=1-3, to=1-4]
                  \end{tikzcd} 
              \end{equation}
    For $k \in \{0,1,2,3\}$ we have $\Ext^\bullet(\sheaf{Y}(k), E) = 0$. Then for $k \in \{0,1,2\}$, the triangle above gives $\Ext^\bullet(\sheaf{Y}(k), i_*i^*E) = \Ext^\bullet(\sheaf{X}(k), i^*E) = 0$. Hence $i^*E \in \Ku(X)$. Note that this gives a pair of adjoint functors $i^* \dashv \pr_Y \circ i_*$ between $\Ku(Y)$ and $\Ku(X)$.

\begin{lemma}[][]
    Let $E \in \Ku(Y)$. Then the object $\operatorname{pr}_Yi_*i^*E \in \Ku(Y)$ fits into the distinguished triangle:
    \begin{equation}
                  \begin{tikzcd}[ampersand replacement=\&,cramped]
                      {E} \& {\operatorname{pr}_Y i_* i^*E} \& {\mathsf{S}_{\Ku(Y)}E[-2]} \& {}
                      \arrow[from=1-1, to=1-2]
                      \arrow[from=1-2, to=1-3]
                      \arrow["{+1}", from=1-3, to=1-4]
                  \end{tikzcd} \label{ses:res_general}
              \end{equation}
\end{lemma}
\begin{proof}
    Apply the projection functor $\operatorname{pr}_Y = \rmut{\sheaf{Y}(-1)}\circ\rmut{\sheaf{Y}(-2)}\circ\lmut{\sheaf{Y}}\circ \lmut{\sheaf{Y}(1)}$ to the triangle \eqref{ses:res_trig_E}. Note that $E(-1) \in \ord{\sheaf{Y},\sheaf{Y}(1)}^\perp \cap {}^{\perp}\ord{\sheaf{Y}(-2)}$. We have that 
    \[\operatorname{pr}_Y E(-1) = \rmut{\sheaf{Y}(-1)}E(-1) = \rmut{\sheaf{Y}}E \otimes \sheaf{Y}(-1) = \mathsf{O}^{-1}_{\Ku(Y)}E = \mathsf{S}_{\Ku(Y)}E[-3],\]
    where we have used the definition of $\mathsf{O}_{\Ku(Y)}$ and \cref{lem:O_S_KuY}. Then we have 
    \begin{equation}
                  \begin{tikzcd}[ampersand replacement=\&,cramped]
                      {\mathsf{S}_{\Ku(Y)}E[-3]} \& {E} \& {\operatorname{pr}_Y i_* i^*E} \& {}
                      \arrow[from=1-1, to=1-2]
                      \arrow[from=1-2, to=1-3]
                      \arrow["{+1}", from=1-3, to=1-4]
                  \end{tikzcd} 
              \end{equation}
    After rotation we obtain the desired triangle.
\end{proof}
Recall that $i^*\kappa_1$ and $i^*\kappa_2$ parametrize the objects $\FFF_{\ell}$ and $\PPP_{\ell}$ in $\Ku(X)$ respectively (see \eqref{seq:P_ell}). Then $\{i^*\kappa_1, i^*\kappa_2\}$ forms a basis of the lattice $\Knum(\Ku(X))$, with the Euler pairing given by (see \cite{BLMS})
\[\chi_{\Ku(X)} = \begin{pmatrix} -2 & -1 \\ -1 & -2 \end{pmatrix}.\]

\begin{proof}[Proof of \cref{prop:Ext2_zero_general} and Restriction \cref{thm:res}]
    Fix a primitive class $v$ in $\Knum(\Ku(Y))$. Define
    \[
        \UUUs_Y(v) \coloneqq
        \ab\{E \in \ModsY(\Ku(Y),v) \mid \Ext^2(E,E)=0,\ i^*E \in \ModsX(\Ku(X),i^*v)
        \}.
    \]
    By upper semicontinuity of $\ext^2$ and openness of stability, $\UUUs_Y(v)$ is open in $\ModsY(\Ku(Y),v)$. We prove by induction on $|v|^2 = -\chi(v,v)$ the stronger statement that, for every non-empty open subset $U \subset \cF_2(Y)$, there is a {\good} object $E \in \UUUs_Y(v)$ with $\operatorname{pl}(E) \subset U$.

    For $|v|^2 = 1$, after rotation we may assume that 
    $v = \kappa_1$. We can take any object $E = \FFF_{\varPi} \in \Modss^\circ(\Ku(Y),\kappa_1)$ with $\varPi \in U$. Then the statements follow from \cref{lem:Ext_F_Pi} and the proof of \cref{cor:Lag_imm}.
    
    For $|v|^2 = 3$, after rotation we may assume that $v = \kappa_1 + \kappa_2$. By \cref{lem:plane_incid} and the dimension estimate for $\sI_1$, we may choose a general pair $(\varPi,\varPi')\in\sI_0\cap(U\times U)$ with $\varPi\cap\varPi'=\{p\}$. The extension $\EEE_{\eta} \in \Modss^\circ(\Ku(Y),\kappa_1 + \kappa_2)$ in \cref{lem:k1+k2_ext2} then satisfies $\Ext^2(\EEE_{\eta}, \EEE_{\eta}) = 0$. Its restriction is stable by the argument below, applied to $E_+=\PPP_{\varPi}$ and $E_-=\FFF_{\varPi'}$; here $\Ext^2(E,E_+)=0$ follows from the exact sequences in \eqref{seq:E_eta}.

    For $|v|^2 > 3$, we proceed by induction. Pick the Pick's decomposition $(v_+,v_-)$ of $v$. By the induction hypothesis, we first take a {\good} object $E_{+} \in \UUUs_Y(v_{+})$ with $\operatorname{pl}(E_+)\subset U$. Note that 
    \[\{\varPi' \in U \mid \exists\, \varPi \in \operatorname{pl}(E_+),\ \varPi \cap \varPi' \ne \varnothing\}\]
    is a finite union of one-dimensional closed subsets in $U$. Let $U'$ be its complement. By the induction hypothesis, we can take a {\good} object $E_- \in \UUUs_Y(v_-)$ with $\operatorname{pl}(E_-) \subset U'$. In particular, $\varPi_+ \cap \varPi_- = \varnothing$ for all $\varPi_+ \in \operatorname{pl}(E_+)$ and $\varPi_- \in \operatorname{pl}(E_-)$. Note that $v_+$ and $v_-$ lie in the same sextant; then $E_+$ and $\mathsf{S}_{\Ku(Y)}E_-[-2]$ lie in adjacent sextants, and likewise for $E_-$ and $\mathsf{S}_{\Ku(Y)}E_+[-2]$. Since the Serre functor action does not change $\operatorname{pl}(E_{\pm})$, \cref{prop:phase_Hom_c} gives $\Hom(E_+, \mathsf{S}_{\Ku(Y)}E_-[-2]) = 0$ and $\Hom(E_-, \mathsf{S}_{\Ku(Y)}E_+[-2]) = 0$. By Serre duality, we also have 
    \begin{equation}
	\Ext^2(E_+, E_-) = \Ext^2(E_-, E_+) = 0. \label{eq:ext2_pm_zero}
    \end{equation}
    By the proof of \cref{thm:moduli_nonempty} we have $\Ext^1(E_+,E_-) \ne 0$. A non-trivial extension $0 \to E_- \to E \to E_+ \to 0$ gives a {\good} object $E \in \Modss^{\circ}(\Ku(Y),v)$. Since $E_{\pm} \in \UUUs_Y(v_{\pm})$ we have $\Ext^2(E_+,E_+) = 0$ and $\Ext^2(E_-,E_-) = 0$. Combining with \eqref{eq:ext2_pm_zero}, we have $\Ext^2(E,E) = 0$. Moreover we have that $\operatorname{pl}(E) = \operatorname{pl}(E_+) \cup \operatorname{pl}(E_-) \subset U$. 

    Finally we check that $i^*E$ is $\sigma_X$-stable. Applying  $\Hom(E_+,-)$ to the triangle \eqref{ses:res_general}, we have the long exact sequence
    \begin{equation}
                  \begin{tikzcd}[ampersand replacement=\&,cramped]
                       {} \& {\Hom(E_+,E)} \& {\Hom(E_+, \operatorname{pr}_Yi_*i^*E)} \& {\Hom(E_+, \mathsf{S}_{\Ku(Y)}E[-2])} \& {}
                      \arrow[from=1-1, to=1-2]
                      \arrow[from=1-2, to=1-3]
                      \arrow[from=1-3, to=1-4]
                      \arrow[from=1-4, to=1-5]
                  \end{tikzcd} 
    \end{equation}
    We have $\Hom(E_+,E) = 0$ since $\phi(E_+) > \phi(E)$, and $\Hom(E_+, \mathsf{S}_{\Ku(Y)}E[-2]) \cong \Ext^2(E, E_+)^\dual = 0$ since $\Ext^2(E_+, E_+) = \Ext^2(E_-, E_+) = 0$. It follows that 
    \[ \Hom(E_+, \operatorname{pr}_Yi_*i^*E) \cong \Hom(i^*E_+, i^*E) = 0. \]
    By the induction hypothesis, $i^*E_{\pm}$ are $\sigma_X$-stable. Note that $(i^*v_+, i^*v_-)$ is also a Pick's decomposition of the primitive character $i^*v \in \Knum(\Ku(X))$. Moreover, the vanishing $\Hom(i^*E_+,i^*E)=0$ shows that the restricted extension does not split. The argument of \cref{lem:Ext1_nonzero} therefore shows that $i^*E$ is $\sigma_X$-stable. This concludes the induction.
\end{proof} 
 
The major consequence of \cref{thm:res} is the following generalization of \cref{cor:Lag_imm} to higher-dimensional moduli spaces. This result is in the same spirit as \cite[Theorem 8.3]{LLPZ} or \cite[Theorem A.4]{FGLZ} for cubic 4-folds and cubic 3-folds. Recall that for a primitive character $w \in \Knum(\Ku(X))$ and a generic stability condition $\sigma_X$ on $\Ku(X)$, the moduli space $\ModssX(\Ku(X), w)$ is a smooth projective hyper-K\"ahler variety of dimension $2 - \chi_X(w,w)$ (\cite[Theorem 29.2]{BLMNPS}).
\begin{corollary}[][cor:Lag_imm_general]
  For primitive $v \in \Knum(\Ku(Y))$, the moduli space $\ModssY(\Ku(Y), v)$ contains a smooth open subset $\UUUs_Y(v)$ such that there exists a rational map $r_v\colon \bar{\UUUs_Y(v)} \dashrightarrow \ModssX(\Ku(X), i^*v)$ induced by the pull-back $i^*\colon \Ku(Y) \to \Ku(X)$, whose image is a Lagrangian subvariety of the hyper-K\"ahler variety $\ModssX(\Ku(X), i^*v)$. Moreover, $r_v|_{\UUUs_Y(v)}$ is unramified and hence $r_v$ is generically finite onto its image.
\end{corollary}
\begin{proof}
    Let $\UUUs_Y(v)$ be the open subset in \cref{thm:res}. The vanishing $\Ext^2(E,E) = 0$ for all $E \in \UUUs_Y(v)$ implies that $\UUUs_Y(v)$ is smooth. The rest of the proof is similar to that of \cref{cor:Lag_imm}. Applying $\Hom(E,-)$ to the triangle \eqref{ses:res_general}, we have the long exact sequence
    \begin{equation}
        \begin{tikzcd}[ampersand replacement=\&,cramped]
         {} \& {\Hom(E, \mathsf{S}_{\Ku(Y)}E[-2])} \& {\Ext^1(E,E)} \& {\Ext^1(i^*E, i^*E)} \& {}
        \arrow[from=1-1, to=1-2]
        \arrow[from=1-2, to=1-3]
        \arrow[from=1-3, to=1-4]
        \arrow[from=1-4, to=1-5]
        \end{tikzcd}
    \end{equation}
    By Serre duality we have $\Hom(E, \mathsf{S}_{\Ku(Y)}E[-2]) \cong \Ext^2(E,E)^\dual = 0$. Therefore the tangent map of $r_v$ at $E \in \UUUs_Y(v)$, which is given by the map $\Ext^1(E,E) \to \Ext^1(i^*E, i^*E)$, is injective. Hence $r_v|_{\UUUs_Y(v)}$ is unramified and therefore generically finite onto $\img r_v$. 
    For $v = a\kappa_1 + b\kappa_2$, we have 
    \begin{align*}
        &\dim \UUUs_Y(v) = 1 - \chi_Y(v,v) = a^2 + ab + b^2 + 1;\\ 
        &\dim \ModssX(\Ku(X), i^*v) = 2 - \chi_X(i^*v,i^*v) = 2a^2 + 2ab + 2b^2 + 2.
    \end{align*}
    We deduce that $\img r_v$ has dimension exactly half of that of $\ModssX(\Ku(X), i^*v)$. For the Lagrangian property, it remains to check that the holomorphic symplectic form restricts to zero on $r_v(\UUUs_Y(v))$, which also follows from that $\Ext^2(E,E) = 0$ for all $E \in \UUUs_Y(v)$.
\end{proof} 

\printbibliography[heading=bibintoc]

\end{document}